\documentclass[reqno]{amsart}
\usepackage{amssymb}
\usepackage{graphicx}
\usepackage{amsmath}
\usepackage{mathbbol}

\usepackage[usenames, dvipsnames]{color}
\usepackage{verbatim}
\usepackage{mathrsfs}
\usepackage{bm}
\usepackage{cite}
\allowdisplaybreaks[4]

\numberwithin{equation}{section}

\newtheorem{theorem}{Theorem}[section]

\newtheorem{lemma}[theorem]{Lemma}
\newtheorem{prop}[theorem]{Proposition}
\newtheorem{example}[theorem]{Example}

\theoremstyle{definition}
\newtheorem{remark}[theorem]{Remark}

\theoremstyle{definition}

\theoremstyle{definition}

\makeatletter
\def\dashint{\operatorname%
{\,\,\text{\bf-}\kern-.98em\DOTSI\intop\ilimits@\!\!}}
\makeatother

\def\\det{\text{\det}}

\def\.5{\frac{1}{2}}

\newcommand{\RN}[1]{%
  \textup{\uppercase\expandafter{\romannumeral#1}}%
}

\renewcommand{\epsilon}{\varepsilon}

\newcounter{marnote}

\begin{document}

\title[Asymptotics of the stress concentration]{Asymptotics of the stress concentration in high-contrast elastic composites}
\author[H.G. Li]{Haigang Li}
\address[H.G. Li]{School of Mathematical Sciences, Beijing Normal University, Laboratory of MathematiCs and Complex Systems, Ministry of Education, Beijing 100875, China.}
\email{hgli@bnu.edu.cn}

\author[L.J. Xu]{Longjuan Xu}
\address[L.J. Xu]{Department of Mathematics, National University of Singapore, 10 Lower Kent Ridge Road, Singapore 119076.}
\email{ljxu311@163.com}


\date{\today} 


\maketitle
\begin{abstract}
A long-standing area of materials science research has been the study of electrostatic, magnetic, and elastic fields in composite with densely packed inclusions whose material properties differ from that of the background. For a general elliptic system, when the coefficients are piecewise H\"{o}lder continuous and uniformly bounded, an $\varepsilon$-independent bound of the gradient was obtained by Li and Nirenberg \cite{ln}, where $\varepsilon$ represents the distance  between the interfacial surfaces. However, in high-contrast composites, when $\varepsilon$ tends to zero, the stress always concentrates in the narrow regions. As a contrast to the uniform boundedness result of Li and Nirenberg, in order to investigate the role of $\varepsilon$ played in such kind of concentration phenomenon, in this paper we establish the blow-up asymptotic expressions of the gradients of solutions to the Lam\'{e} system with partially infinite coefficients in dimensions two and three. We discover the relationship between the blow-up rate of the stress and the relative convexity of adjacent surfaces, and find  a family of blow-up factor matrices with respect to the boundary data. Therefore, this work completely solves the Babu\u{s}ka problem on blow-up analysis of stress concentration in high-contrast composite media. Moreover, as a byproduct of these local analysis, we establish an extended Flaherty-Keller formula on the global effective elastic property of a periodic composite with densely packed fibers, which is related to the ``Vigdergauz microstructure'' in the shape optimization of fibers.
\end{abstract}

\section{Introduction}

In this paper we are concerned with the blow-up behavior of the gradients of solutions to a class of elliptic systems, stimulated by the study of composite media with closely spaced interfacial boundaries. It is a long-standing area of material science research to study the high concentration of electrostatics, magnetic, and elastic fields in high-contrast composites with densely packed inclusions since the time of Maxwell and Reyleigh. This requires an understanding of micro-structural effects, especially from the distances (say, $\varepsilon$)  between inclusions, because when the inclusions are close to touching, the charge density becomes nearly singular. To evaluate the electrostatic fields (where the potential function is scalar-valued), the potential theory, Fourier analysis, and numerical method have been fully developed. While, for the elastic field (where the deformation displacement is vector-valued), in order to predict damage initiation and growth in carbon-fiber epoxy composites at the fiber scale level,  Babu\u{s}ka, et al. \cite{ba} assumed the systems of linear elasticity 
$$\mathcal{L}_{\lambda, \mu}u=\mu\Delta u+(\lambda+\mu)\nabla(\nabla\cdot u)$$
in unidirectional composites to numerically analyze the residual stresses and stresses due to mechanical loads, where $u=(u^1,u^2,u^3)^{\mathrm{T}}$ expresses the  displacement. Obviously, this multiscale problem need more regorous mathematical treatment and numerical analysis to control the errors of the analysis. On the other hand, we emphasize that there is a significant difficulty in applying the method developed for scalar equations to systems of equations. For instance, the maximum principle does not hold for the Lam\'{e} system. Due to these difficulties on PDE theory and numerical analysis as well as the importance in practical applications, it arouses great interest of many applied mathematicians and engineers. In the last  two decades, there has been an extensive study on the gradient estimates of solutions to elliptic equations and systems with discontinuous coefficients, to show whether the stress remains bounded or blows up when inclusions touch or nearly touch.

Bonnetier and Vogelius \cite{bv} considered the elliptic equation with piecewise constant coefficients in dimension two
$$\nabla(a_k(x)\nabla u)=0\quad\mbox{in}~D,$$
where the scalar $u$ is the out of plane displacement, $D$  represents the cross-section of a fiber-reinforced composite taken perpendicular to the fibers, containing a finite number of inhomohenuities, which are very closely spaced and may possible touch. The coefficients $0<a_k(x)<\infty$ take two different constant values, after rescaling,
\begin{align*}
a_k(x)&=k\quad\mbox{for~}x~\mbox{inside~the~cross-sections~of~the~fibers},\\
a_k(x)&=1\quad\mbox{elsewhere~in~}D.
\end{align*}
Despite the discontinuity of the coefficient along the interfaces, they proved that any variational solution $u$ is in $W^{1,\infty}$, which actually improves a classical regularity result due to De Giorgi and Nash, which asserts that $H^1$ solution is in some H\"{o}lder class. A general result was established by Li and Vogelius \cite{lv} for a class of divergence form elliptic equations with piecewise H\"{o}lder continuous coefficients. They obtained a uniform bound of $|\nabla u|$ regardless of $\varepsilon$ in any dimension $d\geq2$. Li and Nirenberg \cite{ln} extended the results in \cite{lv} to general elliptic systems including systems of elasticity. This, in particular, answered in the affirmative the question that is naturally led to by the above mentioned numerical indication in \cite{ba} for the boundedness of the stress as $\varepsilon$ tends to zero. Dong and Xu \cite{dx} further showed that a $W^{1,1}$ weak solution is Lipschitz and piecewise $C^1$. Recently, Dong and Li \cite{dl} used an image charge method to construct a Green's function for two adjacent circular inclusions and obtained more interesting higher-order derivative estimates for non-homogeneous equations making clear their specific dependence on $k$ and $\varepsilon$ exactly. But for more general elliptic equations and systems, and more general shape of inclusions, it is still an open problem to estimate higher-order derivatives in any dimension. We draw the attention of readers to the open problem on page 894 of \cite{ln}.

As mentioned above, the concentration of the stresses is greatly influenced by the thickness of the ligament between inclusions. To figure out the influence from this thickness $\varepsilon$, one assumes that the material parameters of the inclusions degenerate to infinity. However, this makes the situation become quite different. As a matter of fact, in 1960's, in the context of electrostatics, Keller \cite{keller1} computed the effective electrical conductivity for a composite medium consisting of a dense cubic array of identical perfectly conducting spheres (that is, $k$ degenerates to $\infty$) imbedded in a matrix medium and first discovered that it becomes infinite when sphere inclusions touch each other. Keller found that this singularity is not contained in the expressions given by Maxwell, and by Meredith and Tobias \cite{mt}. See also Budiansky and Carrier \cite{bc}, and Markenscoff \cite{m}. Rigorous proofs were later carried out  by Ammari et al. \cite{akl,aklll} for the case of circular inclusions by using layer potential method, together with the maximum principle. Since then, there is a long list of literature in this direction of research, for example, see \cite{abtv, adkl, bly1, bly2, bt, dong, g, gb, gn, kleey, keller2, Li-Li, ll, llby, lx, lz,  ly2, lyu, y1,y2} and the references therein. It is proved that the blow-up rate of $|\nabla u|$ is $\varepsilon^{-1/2}$ in dimension two and $|\varepsilon\log\varepsilon|^{-1}$ in dimension three. From the perspective of practical application in engineering and the requirement of numerical algorithm design, it is more interesting and important to characterize the singular behavior of $\nabla u$, see \cite{ackly,kly,kly2,lhg,lwx,lyu1,lly}.

In the context of linear elasticity, for Lam\'{e} system with partially infinite coefficients,   by building an iteration technique with respect to the energy, the first author and his collaborators  overcame the difficulty caused by the lack of the maximum principle, obtained the upper and lower pointwise bounds of $|\nabla u|$, and  showed that $|\nabla u|$ may blow up on the shortest line between two adjacent inclusions, see \cite{bjl,bll1,bll2,l}. By using the polynomial function $x_{d}=|x'|^m$, $m\geq2$, as a local expression of inclusion's boundary to measure its order of convexity, Li and Hou \cite{hl} revealed the relationship between the blow-up rate of $|\nabla u|$ and the convexity order $m$. However, under the same logic as in the electrostatics problem, what one cares more about in practical applications is how to obtain an asymptotic formula to characterize the singular behaviour of $\nabla u$ in the whole narrow region between two adjacent inclusions. The main contribution of the paper is that we completely solve this problem in two physically relevant dimensions $d=2$ and $3$, and for all $m\geq2$. For $d\geq4$, the result is similar. Our asymptotic expressions of the gradients of solutions not only show the optimality of the blow-up rates, which depend only on the dimension $d$ and the convexity order $m$ of the inclusions, but also provide a family of blow-up factor matrices, which are linear functionals of boundary value data, determining whether or not blow-up occurs. Notice that when $m>2$, the curvature of the inclusions vanishes at the two nearly touch points, so in general we can not use a spherical inclusion to approximate an $m$-convex inclusion. 

The asymptotic formulas obtained above clearly reflect the local property of $\nabla u$. Beyond this, they can further influence the global property of a composite. For the effective elastic moduli of a composite, Flaherty and Keller \cite{fk} obtained an symptotic formula for a retangular array of cylinders $(m=2)$ in the nearly touching, when the cylinders are hard inclusions and showed their validity numerically. As an application of the above local asymptotic formulas, we give an extended Flaherty-Keller formula for $m$-convex inclusions, which is also related to the ``Vigdergauz microstructure'' \cite{v}, having a large volume fraction in the theory of structure optimization, see Grabovsky and Kohn \cite{gk}. 

To end this introduction, we make some comments on the corresponding numerical problem. Accurate numerical computation of the gradient in the present of closely spaced inclusions is also a well-known challenging problem in computational mathematics and sciences.  Here it should be noted that Lord Rayleigh, in his classic paper \cite{rl}, use Fourier approach to determine the effective conductivity of a composite material consisting of a periodic array of disks in a uniform background. In the case that inclusions are reasonably well separated or have conductivities close to that of the background, Rayleigh's method gives excellent result. Unfortunately, if the inclusions are close to touching and their conductivities differ greatly from that of the background, the charge density becomes nearly singular and the number of computation degrees of freedom required extremely large. Recently, a hybrid numerical method was developed by Cheng and Greengard \cite{cg}, and Cheng \cite{cheng}. Related works can be referred to Kang, Lim, and Yun \cite{kly}, and McPhedran, Poladian, and Milton \cite{mpm}. For high-contrast elastic composite, a serious difficulty arises in applying the methods for scalar equations to systems of equations. We expect our asymptotic formulas of $\nabla u$ in the narrow regions, the most difficult areas to deal with, can open up a way to do some computation for inclusions of arbitrary shape.

This paper consists of eight sections including introduction. In Section \ref{sec problems}, we first fix our domain and formulate the problem with partially infinite coefficients, and then introduce a family of vector-valued auxiliary functions with several preliminary estimates including the main ingredient Proposition \ref{prop a11} for the asymptotics of $a_{11}^{\alpha\alpha}$, $\alpha=1,\cdots,d$. In Section \ref{sec mainresults}, a family of  the blow-up factors is defined. Then our main results are stated. Theorem \ref{thm1} and Theorem \ref{thm2} are for 2-convex inclusions in 2D and 3D, respectively, Theorem \ref{thmhigher} and Theorem \ref{thmhigher2} are for $m$-convex inclusions. Finally we give an example to show the dependence on the precise geometry feature of $D_{1}$ and $D_{2}$. In Section \ref{sec pf prop a11 b}, we prove the two important ingredients Propositions \ref{prop a11} and \ref{prop converge b}, where two improved estimates,  Theorem \ref{coro v1} and Theorem \ref{coro v1d+1}, are used . The proofs of Theorems \ref{thm1} and \ref{thm2} are given in Section \ref{sec pf thm1}.  We prove Theorems \ref{thmhigher} and \ref{thmhigher2} in Section \ref{proof general}. Finally, by applying the local asymptotic formulas established in the previous sections, we give an extended Flaherty-Keller formula in Section \ref{application}. The proof of Theorem \ref{coro v1} and Theorem \ref{coro v1d+1} is given in the Appendix.

\section{Problem formulation, Decomposition and Some preliminary results}\label{sec problems}
In this section we first fix our notations and formulate the problems, then identify the key difficulties and present the strategy to solve them, and finally introduce our auxiliary functions, involving the parameters of Lam\'{e} system, and give some preliminary results.
\subsection{Problem formulation}\label{subsec prb form}
Because the aim of this paper is to study the asymptotic behavior of $\nabla u$ in the narrow region between two adjacent inclusions, we may without loss of generality restrict our attention to a situation with only two adjacent inclusions. The basic notations used in this paper follow from  \cite{bll1}. 

We use $x=(x',x_{d})$ to denote a point in $\mathbb R^{d}$, where $x'=(x_{1},\cdots,x_{d-1})$ and $d=2,3$. Let $D$ be a bounded open set in $\mathbb R^{d}$ with $C^{2}$ boundary. $D_{1}$ and $D_{2}$ are two disjoint convex open subsets in $D$ with $C^{2,\gamma}$ $(0<\gamma<1)$ boundaries, $\varepsilon$-apart, and far away from $\partial D$. That is,
\begin{equation*}
\begin{split}
\overline{D}_{1},\overline{D}_{2}\subset D,\quad
\varepsilon:=\mbox{dist}(D_{1},D_{2})>0,\quad\mbox{dist}(D_{1}\cup D_{2},\partial D)>\kappa_{0}>0,
\end{split}
\end{equation*}
where $\kappa_{0}$ is a constant independent of $\varepsilon$. We also assume that the $C^{2,\gamma}$ norms of $\partial{D}_{1}$, $\partial{D}_{2}$, and $\partial{D}$ are bounded by some positive constant independent of $\varepsilon$. Set
$$\Omega:=D\setminus\overline{D_{1}\cup D_{2}}.$$

Assume that $\Omega$ and $D_{1}\cup D_{2}$ are occupied, respectively, by two different isotropic and homogeneous materials with different Lam\'{e} constants $(\lambda, \mu)$ and $(\lambda_1, \mu_1)$. Then the elasticity tensors for the background and the inclusion can be written, respectively, as $\mathbb{C}^0$ and $\mathbb{C}^1$, with
$$C_{ijkl}^0=\lambda\delta_{ij}\delta_{kl} +\mu(\delta_{ik}\delta_{jl}+\delta_{il}\delta_{jk}),$$
and
$$C_{ijkl}^1=\lambda_1\delta_{ij}\delta_{kl} +\mu_1(\delta_{ik}\delta_{jl}+\delta_{il}\delta_{jk}),$$
where $i, j, k, l=1,2,\cdots,d$ and $\delta_{ij}$ is the kronecker symbol: $\delta_{ij}=0$ for $i\neq j$, $\delta_{ij}=1$ for $i=j$. Let $u=(u^1, u^2,\cdots,u^{d})^{\mathrm{T}}:D\rightarrow\mathbb{R}^{d}$ denote the displacement field. For a given vector-valued function $\varphi=(\varphi^1,\varphi^2,\cdots,\varphi^{d})^{\mathrm{T}}$, we consider the following Dirichlet problem for the Lam\'{e} system with pieceiwise constant coefficients:
 \begin{align}\label{Lame}
\begin{cases}
\nabla\cdot \left((\chi_{\Omega}\mathbb{C}^0+\chi_{D_{1}\cup{D}_{2}}\mathbb{C}^1)e(u)\right)=0,&\hbox{in}~~D,\\
u=\varphi, &\hbox{on}~~\partial{D},
\end{cases}
\end{align}
where $\chi_{\Omega}$ is the characteristic function of $\Omega\subset \mathbb{R}^{d}$, and
$$e(u)=\frac{1}{2}(\nabla u+(\nabla u)^{\mathrm{T}})$$
is the strain tensor.

Assume that the standard ellipticity condition holds for (\ref{Lame}), that is,
\begin{align*}
\mu>0,\quad d\lambda+2\mu>0,\quad \mu_1>0,\quad d\lambda_1+2\mu_1>0.
\end{align*}
For $\varphi\in C^{1,\gamma}(\partial D; \mathbb{R}^{d})$, it is well known that there exists a unique solution $u\in H^1(D; \mathbb{R}^{d})$ to the Dirichlet problem (\ref{Lame}), which is also the minimizer of the energy functional
$$J_1[u]:=\frac{1}{2}\int_\Omega \left((\chi_{\Omega}\mathbb{C}^0+\chi_{D_{1}\cup{D}_{2}}\mathbb{C}^1)e(u), e(u)\right)dx $$
on
\begin{align*}
H^1_\varphi(D; \mathbb{R}^{d}):=\left\{u\in  H^1(D; \mathbb{R}^{d})~\big|~ u-\varphi\in  H^1_0(D; \mathbb{R}^{d})\right\}.
\end{align*}

As mentioned previously, Li and Nirenberg \cite{ln} proved that $\nabla u$ is uniformly bounded with respect to $\varepsilon$. But, in high-contrast composite media, the concentration of $\nabla u$ is a very usual phenomenon when the distance $\varepsilon$ is sufficiently small. In order to investigate the role of $\varepsilon$ in such concentration phenomenon, let us assume that the Lam\'{e} constant in $D_{1}\cup D_{2}$   degenerates to infinite and consider this extreme case. To this end, we first introduce the linear space of rigid displacement in $\mathbb{R}^{d}$:
$$\Psi:=\{\psi\in C^1(\mathbb{R}^{d}; \mathbb{R}^{d})\ |\ \nabla\psi+(\nabla\psi)^{\mathrm{T}}=0\},$$
with a basis $\left\{\psi_{\alpha}~|~\alpha=1,2,\cdots,\frac{d(d+1)}{2}\right\}$, namely,
$$\left\{~e_{i},~x_{j}e_{k}-x_{k}e_{j}~\big|~1\leq\,i\leq\,d,~1\leq\,j<k\leq\,d~\right\},$$
where $e_{1},\cdots,e_{d}$ denote the standard basis of $\mathbb{R}^{d}$. For fixed $\lambda$ and $\mu$, denoting $u_{\lambda_1,\mu_1}$ as the solution of (\ref{Lame}), then we have \cite{bll1}
\begin{align*}
u_{\lambda_1,\mu_1}\rightarrow u\quad\hbox{in}\ H^1(D; \mathbb{R}^{d}),\quad \hbox{as}\ \min\{\mu_1, d\lambda_1+2\mu_1\}\rightarrow\infty,
\end{align*}
where $u$ is the unique $H^1(D; \mathbb{R}^{d})$ solution of
\begin{align}\label{maineqn}
\begin{cases}
\mathcal{L}_{\lambda, \mu}u:=\nabla\cdot(\mathbb{C}^0e(u))=0,\quad&\hbox{in}\ \Omega,\\
u|_{+}=u|_{-},&\hbox{on}\ \partial{D}_{i},i=1,2,\\
e(u)=0,&\hbox{in}~~D_{i},i=1,2,\\
\int_{\partial{D}_{i}}\frac{\partial u}{\partial \nu}\Big|_{+}\cdot\psi_{\alpha}=0,&i=1,2,\alpha=1,2,\cdots,\frac{d(d+1)}{2},\\
u=\varphi,&\hbox{on}\ \partial{D},
\end{cases}
\end{align}
where 
\begin{align*}
\frac{\partial u}{\partial \nu}\Big|_{+}&:=(\mathbb{C}^0e(u))\vec{n}=\lambda(\nabla\cdot u)\vec{n}+\mu(\nabla u+(\nabla u)^{\mathrm{T}})\vec{n},
\end{align*}
and $\vec{n}$ is the unit outer normal of $D_{i}$, $i=1,2$. Here and throughout this paper the subscript $\pm$ indicates the limit from outside and inside the domain, respectively. The existence, uniqueness and regularity of weak solutions to (\ref{maineqn}) can be referred to the Appendix of \cite{bll1}. We note that it suffices to consider the problem \eqref{maineqn} with $\varphi\in C^0(\partial D;\mathbb R^{d})$ replaced by $\varphi\in C^{1,\gamma}(\partial D;\mathbb R^{d})$. Indeed, it follows from the maximum principle \cite{mmn} that
$\|u\|_{L^{\infty}(D)}\leq C\|\varphi\|_{C^0(\partial D)}$.
Taking a slightly small domain $\widetilde D\subset\subset D$, then in view of the interior derivative estimates for Lam\'{e} system, we find that $\tilde\varphi:=u\big|_{\partial \widetilde D}$ satisfies
$$\|\tilde\varphi\|_{C^{1,\gamma}(\partial\widetilde D)}\leq C\|u\|_{L^{\infty}(D)}\leq C\|\varphi\|_{C^{0}(\partial D)}.$$
Without loss of generality, we assume that $\|\varphi\|_{C^{0}(\partial D)}=1$ by considering $u/\|\varphi\|_{C^{0}(\partial D)}$ if $\|\varphi\|_{C^{0}(\partial D)}>0$. If $\varphi\big|_{\partial D}=0$, then $u\equiv0$.

\subsection{Main difficulties and decomposition }\label{subsec main difficulty}

We first point out that problem \eqref{maineqn} has free boundary value feature. Although $e(u)=0$ implies $u$ in $\overline{D}_i$ is linear combination of $\psi_{\alpha}$,
\begin{equation}\label{Cialpha}
u=\sum_{\alpha=1}^{d(d+1)/2}C_i^\alpha \psi_{\alpha}\quad \text{in}~~ \overline{D}_i,
\end{equation}
these $d(d+1)$ constants $C_i^\alpha$ are free, which will be uniquely determined by $u$. We would like to emphasize that this is exactly the biggest difference with the conductivity model \cite{bly1}, where only two free constants need us to handle in any  dimension. It is the increase of the number of free contants that makes elastic problem quite difficult to deal with. Therefore, how to determine such many constants is one of main difficulties we need to solve. 

Our strategy in spirit follows from \cite{bll1,bll2}. First, by continuity of $u$ across $\partial D_{i}$, we can decompose the solution of \eqref{maineqn}, 
\begin{equation}\label{decom_u}
u(x)=\sum_{i=1}^{2}\sum_{\alpha=1}^{d(d+1)/2}C_i^{\alpha}v_{i}^{\alpha}(x)+v_{0}(x),\quad x\in\,\Omega ,
\end{equation}
where $v_{i}^{\alpha},v_{0}\in{C}^{2}(\Omega;\mathbb R^d)$, respectively, satisfying
\begin{equation}\label{equ_v1}
\begin{cases}
\mathcal{L}_{\lambda,\mu}v_{i}^{\alpha}=0,&\mathrm{in}~\Omega,\\
v_{i}^{\alpha}=\psi_{\alpha},&\mathrm{on}~\partial{D}_{i},\\
v_{i}^{\alpha}=0,&\mathrm{on}~\partial{D_{j}}\cup\partial{D},~j\neq i,
\end{cases}
\quad i=1,2,~\alpha=1,\cdots,d(d+1)/2,
\end{equation}
and
\begin{equation}\label{equ_v3}
\begin{cases}
\mathcal{L}_{\lambda,\mu}v_{0}=0,&\mathrm{in}~\Omega,\\
v_{0}=0,&\mathrm{on}~\partial{D}_{1}\cup\partial{D_{2}},\\
v_{0}=\varphi,&\mathrm{on}~\partial{D}.
\end{cases}
\end{equation}
So 
\begin{equation}\label{nabla u-1}
\nabla u(x)=\sum_{i=1}^{2}\sum_{\alpha=1}^{d(d+1)/2}C_i^{\alpha}\nabla v_{i}^{\alpha}(x)+\nabla v_{0}(x),\quad x\in\,\Omega.
\end{equation}
To investigate the symptotic behavior of $\nabla u$, we need both the asymptotic formulas of  $\nabla v_{i}^{\alpha}$ and the exact value of $C_{i}^{\alpha}$. To solve $C_{i}^{\alpha}$, from the fourth line in \eqref{maineqn} and the decomposition \eqref{decom_u}, we have the following linear system of these free constants $C_i^\alpha$,
\begin{equation}\label{equ-decompositon}
\sum_{i=1}^2\sum\limits_{\alpha=1}^{\frac{d(d+1)}{2}} C_i^\alpha \int_{\partial D_j} \frac{\partial v_i^\alpha}{\partial \nu} \Big| _+\cdot \psi_\beta +\int_{\partial D_j}\frac{\partial v_0}{\partial \nu}\Big|_+\cdot \psi_\beta =0,
\end{equation}
where $j=1,\, 2$, $\beta= 1, \, \cdots, \frac{d(d+1)}{2}$. But these coefficients are all boundary integrals. By integration by parts,
\begin{align}\label{def_aij}
a_{ij}^{\alpha\beta}:=-\int_{\partial{D}_{j}}\frac{\partial{v}_{i}^{\alpha}}{\partial\nu}\Big|_{+}\cdot\psi_{\beta}=\int_{\Omega}\big(\mathbb{C}^0e(v_{i}^{\alpha}),e(v_{j}^{\beta})\big)\ dx.
\end{align}
Therefore, in order to solve $C_{i}^{\alpha}$ from \eqref{equ-decompositon}, we have to calculate the energy integral on the right hand side of \eqref{def_aij}. This in turn needs to estimate $\nabla v_{i}^{\alpha}$. In fact, even if we can have the asymptotic formulas of $\nabla v_{i}^{\alpha}$, see Theorem \ref{coro v1} and Theorem \ref{coro v1d+1} below, it is still hard to solve every $C_{i}^{\alpha}$. It seems to be a mission impossible. To avoid this difficulty, we rewrite \eqref{nabla u-1} as
\begin{equation}\label{nablau_dec}
\nabla{u}=\sum_{\alpha=1}^{d(d+1)/2}\left(C_{1}^{\alpha}-C_{2}^{\alpha}\right)\nabla{v}_{1}^{\alpha}
+\nabla u_{b},\quad\mbox{in}~\Omega,
\end{equation}
where
\begin{equation}\label{def u_b}
u_{b}:=\sum_{\alpha=1}^{d(d+1)/2}C_{2}^{\alpha}(v_{1}^{\alpha}+v_{2}^{\alpha})+v_{0}.
\end{equation}
This is because the following boundedness estimates for $|\nabla(v_{1}^{\alpha}+v_{2}^{\alpha})|$ and $|\nabla v_{0}|$, together with the boundedness of $C_{2}^{\alpha}$, $\alpha=1,\cdots,d(d+1)/2$, makes $\nabla u_{b}$ is a ``good'' term, which has no singularity in the narrow region.
 
\begin{theorem}[\cite{jlx,llby} ]\label{auxi thm}
Let $v_{i}^{\alpha}$ and $v_{0}$ be defined in \eqref{equ_v1} and \eqref{equ_v3}, respectively, $i=1,2$. Then we have
$$\|\nabla(v_{1}^{\alpha}+v_{2}^{\alpha})\|_{L^{\infty}(\Omega)}\leq C,\quad \alpha=1,\cdots,d(d+1)/2,\quad \mbox{and}\quad \|\nabla v_{0}\|_{L^{\infty}(\Omega)}\leq C.$$
\end{theorem}

Thus, we reduce the establishment of the asymptotics of $\nabla u$ to that of the asymptotics of $\nabla v_{1}^{\alpha}$, $\alpha=1,\cdots,d(d+1)/2$, and to solving $C_{1}^{\alpha}-C_{2}^{\alpha}$, $\alpha=1,\cdots,d(d+1)/2$. These are two main difficulties that we need to solve in this paper. For the former, we separate all singular terms of $\nabla v_{1}^{\alpha}$, up to constant terms, by using a family of improved auxiliary functions, which depend on the parameters of Lam\'{e} system and the geometry informations of $\partial D_{1}$ and $\partial D_{2}$. This essentially improves the results in  \cite{bll1,bll2}, where only estimates of $|\nabla v_{i}^{\alpha}|$ are obtained. For the latter, we need to characterize the coefficients $a_{ij}^{\alpha\beta}$ to solve the big linear system generated by \eqref{nablau_dec}, and the convergence of $\int_{\partial D_{1}}\frac{\partial u_{b}}{\partial \nu}\Big|_{+}\cdot\psi_{\beta}$, $\beta=1,\cdots,d(d+1)/2$; see Proposition \ref{prop converge b}. Finally, we remark that one crucial step in proving the asymptotics of $C_{1}^{\alpha}-C_{2}^{\alpha}$ is to derive the asymptotic expression of $a_{11}^{\alpha\alpha}$,  which is of independent interest, see Proposition \ref{prop a11} and Remark \ref{rmk a11} below. To state it precisely, we further  fix our domain and notations.

\subsection{Further assumptions on domain and construction of finer auxiliary functions}

Recalling the assumptions about $D_{1}$ and $D_{2}$, there exist two points $P_{1}\in\partial D_{1}$ and $P_{2}\in\partial D_{2}$, respectively, such that
$$\mbox{dist}(P_{1},P_{2})=\mbox{dist}(\partial D_{1},\partial D_{2})=\varepsilon.$$
By a translation and rotation of coordinates, if necessary, we suppose that
$$P_{1}=(0',\frac{\varepsilon}{2})\in\partial D_{1},\quad P_{2}=(0',-\frac{\varepsilon}{2})\in\partial D_{2}.$$
Now we further assume that there exits a constant $R$, independent of $\varepsilon$, such that the portions of $\partial D_{1}$ and $\partial D_{2}$ near $P_{1}$ and $P_{2}$, respectively, can be represented by
\begin{align*}
x_{d}=\frac{\varepsilon}{2}+h_{1}(x')\quad\mbox{and}\quad x_{d}=-\frac{\varepsilon}{2}+h_{2}(x'),\quad\mbox{for}~|x'|<2R.
\end{align*}
Suppose that the convexity of $\partial D_{1}$ and $\partial D_{2}$ is of order $m\geq 2$ ($m\in\mathbb N^{+}$) near the origin,
\begin{align}\label{convexity}
h_{i}(x')=
\begin{cases}
(-1)^{i+1}\kappa_{i}|x'|^{2}+O(|x'|^{2+\gamma}),&\quad m=2,\\
(-1)^{i+1}\kappa_{i}|x'|^{m}+O(|x'|^{m+1}),&\quad m\geq3,
\end{cases} 
\quad\mbox{for}~i=1,2,
\end{align}
and
\begin{equation}\label{h1h3}
\|h_{1}\|_{C^{2,\gamma}(B'_{2R})}+\|h_{2}\|_{C^{2,\gamma}(B'_{2R})}\leq C,
\end{equation}
where $\kappa_{1}$, $\kappa_{2}$, and $C$ are constants independent of $\varepsilon$. We call these inclusions \emph{$m$-convex inclusions}. For simplicity, we assume that $\kappa_{1}=\kappa_{2}=\dfrac{\kappa}{2}$.
For $0<r\leq\,2R$, set the narrow region between $\partial{D}_{1}$ and $\partial{D}_{2}$ as
\begin{equation*}
\Omega_r:=\left\{(x',x_{d})\in \mathbb{R}^{d}~\big|~-\frac{\varepsilon}{2}+h_{2}(x')<x_{d}<\frac{\varepsilon}{2}+h_{1}(x'),~|x'|<r\right\}.
\end{equation*}

By the standard theory for elliptic systems, we have
\begin{equation}\label{mv1--bounded1}
\|\nabla v_{i}^\alpha(x)\|_{L^{\infty}(\Omega\setminus\Omega_{R})}\leq\,C,\quad \alpha=1,\cdots,\frac{d(d+1)}{2},~i=1,2.
\end{equation}
Therefore, in the following we only need to deal with the problems in $\Omega_{R}$. To this end, we denote
$$\delta(x'):=\varepsilon+h_{1}(x')-h_{2}(x'),$$
and introduce a scalar auxiliary function $\bar{u}\in C^2(\mathbb{R}^d)$ such that 
\begin{equation}\label{vvd}
\bar{u}(x)=\frac{x_{d}+\frac{\varepsilon}{2}-h_2(x')}{\delta(x')},\quad\hbox{in}\ \Omega_{2R},
\end{equation}
$\bar{u}=1$ on $\partial{D}_{1}$, $\bar{u}=0$ on $\partial{D}_{2}\cup\partial{D}$, and $\|\bar{u}\|_{C^{2}(\Omega\setminus\Omega_{R})}\leq\,C$.
Next we use the function $\bar{u}$ to generate a family of vector-valued auxiliary functions $u_{1}^{\alpha}$, $\alpha=1,\cdots,d$, which is much finer than before used in \cite{bll1,bll2}.

For $d=2$, recalling that 
$$\psi_{1}=\begin{pmatrix}
1\\
0
\end{pmatrix}\quad \mbox{and} \quad\psi_{2}=\begin{pmatrix}
0\\
1
\end{pmatrix},$$ we define $u_{1}^{\alpha}\in C^{2}(\Omega)$ such that $u_{1}^{\alpha}=v_{1}^{\alpha}$ on $\partial\Omega$,  and, in $\Omega_{2R}$
\begin{equation}\label{auxiliary improved}
\begin{split}
u_{1}^{1}&:=\bar{u}_{1}^{1}+\tilde{u}_{1}^{1}:=
\bar{u}\psi_{1}
+\frac{\lambda+\mu}{\lambda+2\mu}f(\bar{u})\, \delta'(x_{1})\,\psi_{2},\\
u_{1}^{2}&:=\bar{u}_{1}^{2}+\tilde{u}_{1}^{2}:=
\bar{u}\psi_{2}
+\frac{\lambda+\mu}{\mu}f(\bar{u})\,\delta'(x_{1})\,\psi_{1},
\end{split}
\end{equation}
where $f(\bar{u}):=\dfrac{1}{2}\Big(\bar{u}-\frac{1}{2}\Big)^{2}-\dfrac{1}{8}$, $f(\bar{u})=0$ on $(\partial D_{1}\cup\partial D_{2})\cap\{|x_{1}|<R\}$, and  $\|u_{1}^{\alpha}\|_{C^{2}(\Omega\setminus\Omega_{R})}\leq C$. 

For $d=3$, noting that $$\psi_{1}=\begin{pmatrix}
1\\
0\\
0
\end{pmatrix},\quad\psi_{2}=\begin{pmatrix}
0\\
1\\
0
\end{pmatrix}, \quad\mbox{and}\quad \psi_{3}=\begin{pmatrix}
0\\
0\\
1
\end{pmatrix},$$ we define $u_{1}^{\alpha}\in C^{2}(\Omega)$ such that, in $\Omega_{2R}$
\begin{equation}\label{auxiliary improved dim3}
\begin{split}
u_{1}^{\alpha}&:=\bar{u}_{1}^{\alpha}+\tilde{u}_{1}^{\alpha}:=\bar{u}\psi_{\alpha}
+\frac{\lambda+\mu}{\lambda+2\mu}f(\bar{u})\, \partial_{x_{\alpha}}\delta\,\psi_{3},\qquad\qquad\alpha=1,2,\\
u_{1}^{3}&:=\bar{u}_{1}^{3}+\tilde{u}_{1}^{3}:=\bar{u}\psi_{3}
+\frac{\lambda+\mu}{\mu}f(\bar{u})\, (\partial_{x_{1}}\delta\,\psi_{1}+\partial_{x_{2}}\delta\,\psi_{2}).
\end{split}
\end{equation}
Notice that the corrected terms $\tilde{u}_{1}^{\alpha}$ depend on the Lam\'{e} parameters $\lambda$ and $\mu$, which can help us capture all singular terms of $\nabla v_{i}^{\alpha}$,  $\alpha=1,\cdots,d$.

\subsection{Asymptotic expression of $\nabla v_{1}^{\alpha}$, $\alpha=1,\cdots,d$.}
We now use these auxiliary functions $u_{1}^{\alpha}$ to obtain the asymptotics of $\nabla v_{1}^{\alpha}$ in the narrow region $\Omega_{R}$,  $\alpha=1,\cdots,d$. 

\begin{theorem}\label{coro v1}
For $d=2,3$, let $v_{1}^{\alpha}\in{H}^1(\Omega; \mathbb{R}^{d})$ be the weak solutions of \eqref{equ_v1}. Then for sufficiently small  $0<\varepsilon<1/2$, we have
	\begin{align}\label{mv1-bounded1}
	\nabla v_{1}^\alpha(x)=\nabla u_{1}^{\alpha}+O(1),\quad\alpha=1,\cdots,d,\quad x\in\Omega_{R}.
	\end{align}
\end{theorem}

Here we would like to emphasize the importance of the introduction of $\tilde{u}_{1}^{\alpha}$. Although $\bar{u}_{1}^{\alpha}$ can be used to obtain the upper bound estimates of $|\nabla v_{1}^{\alpha}|$ by $|\nabla(v_{1}^{\alpha}-\bar{u}_{1}^{\alpha})|\leq\frac{C}{\sqrt{\varepsilon}}$ derived in \cite[Corollary 5.2]{hl}, it as well shows to us it is possible that there remains more other singular terms in $\nabla(v_{1}^{\alpha}-\bar{u}_{1}^{\alpha})$. As it turns out, the appearance of $\nabla\tilde{u}_{1}^{\alpha}$ can find all the singular terms of $\nabla v_{1}^{\alpha}$ and make $|\nabla(v_{1}^{\alpha}-\bar{u}_{1}^{\alpha}-\tilde{u}_{1}^{\alpha})|$ be bounded. The proof is an adaption of the iteration technique with respect to the energy which was first used  in \cite{llby} and further developed  in \cite{bll1,bll2} to obtain the gradient estimates. We leave it to the Appendix.

The asymptotic expression \eqref{mv1-bounded1} is an essential improvement of the estimate $|\nabla(v_{1}^{\alpha}-\bar{u}_{1}^{\alpha})|\leq\frac{C}{\sqrt{\varepsilon}}$. More importantly, it also allows us to obtain the asymptotics of $a_{11}^{\alpha\alpha}$, $\alpha=1,\cdots,d$, defined in  \eqref{def_aij}. This kind of formulas is also the key to establish the asymptotics of $C_{1}^{\alpha}-C_{2}^{\alpha}$, $\alpha=1,\cdots,d$. We here give the results for $m=2$. For the general cases $m\geq3$, see Section \ref{proof general} below.

\begin{prop}\label{prop a11}[The asymptotics of $a_{11}^{\alpha\alpha}$]
Under the assumptions \eqref{convexity} and \eqref{h1h3} with $m=2$, we have, for sufficiently small $0<\varepsilon<1/2$,
	
(i) for $d=2$, 
\begin{equation*}
a_{11}^{11}=\frac{\pi\mu}{\sqrt{\kappa}}\cdot\frac{1}{\sqrt{\varepsilon}}\left(1+O(\varepsilon^{\frac{\gamma+3}{8}})\right)\quad \mbox{and}\quad a_{11}^{22}=\frac{\pi(\lambda+2\mu)}{\sqrt{\kappa}}\cdot\frac{1}{\sqrt{\varepsilon}}\left(1+O(\varepsilon^{\frac{\gamma+3}{8}})\right);
	\end{equation*}
	
	(ii) for $d=3$, there exist constants $\mathcal{C}_{3}^{*\alpha}$ and $\mathcal{C}_{3}^{*3}$, independent of $\varepsilon$, such that for $\alpha=1,2$,
	\begin{equation*}
	a_{11}^{\alpha\alpha}=
	\frac{\pi\mu}{\kappa}|\log\varepsilon|+\mathcal{C}_{3}^{*\alpha}+O(\varepsilon^{1/8})\quad
	\mbox{and}\quad
	a_{11}^{33}=
	\frac{\pi(\lambda+2\mu)}{\kappa}|\log\varepsilon|+\mathcal{C}_{3}^{*3}+O(\varepsilon^{1/8}).
	\end{equation*}
\end{prop}

\begin{remark}\label{rmk a11}
Here we would point out that if we only use $\bar{u}_{1}^{\alpha}$ as the auxiliary function like in \cite{bll1}, it is also not possible to obtain these asymptotic formula for $a_{11}^{\alpha\alpha}$. The details can be found in the proof of Proposition \ref{prop a11}, see Section \ref{sec pf prop a11 b} below. So the introduction of $\tilde{u}_{1}^{\beta}$ is an essential improvement. Its advantage will be also shown in the calculation of other integrals in later sections, for instance, to calculate the global effective elastic property of a periodic composite containing $m$-convex inclusions. This relates the ``Vigdergauz microstructure'' in the shape optimation of fibers. For more details, see Section  \ref{application}.
\end{remark}

\begin{remark}
As we know, in electrostatics, the condenser capacity of $\partial D_1$ relative to $\partial(D\setminus D_{2})$ is given by
$$\mbox{Cap}(D_{1}):=-\int_{\partial D_{1}}\frac{\partial v_{1}}{\partial\nu}=\int_{\Omega}|\nabla v_{1}|^{2}\ dx,$$
where $v_{1}\in{C}^{2}(\Omega)$ satisfies
\begin{equation*}
\begin{cases}
\Delta v_{1}=0,&\mathrm{in}~\Omega,\\
v_{1}=1,&\mathrm{on}~\partial{D}_{1},\\
v_{1}=0,&\mathrm{on}~\partial{D_{2}}\cup\partial{D}.
\end{cases}
\end{equation*}
Henthforth, in this sense $a_{11}^{\alpha\alpha}$ is an ``elasticity capacity'' of $\partial D_{1}$ relative to $\partial(D\setminus D_{2})$.
\end{remark}

For $\alpha=d+1,\cdots,d(d+1)/2$, we use the auxiliary functions as in \cite{bll1,bll2}, that is,
\begin{align}\label{auxi fun d+1}
u_{1}^{\alpha}:=\bar{u}\psi_{\alpha}.
\end{align}
\begin{theorem}\label{coro v1d+1}({\cite{bll1,bll2}})
For $d=2,3$, let $v_{1}^{\alpha}\in{H}^1(\Omega; \mathbb{R}^{d})$ be the weak solutions of \eqref{equ_v1}. Then for sufficiently small  $0<\varepsilon<1/2$, we have
	\begin{align*}
	\nabla v_{1}^\alpha(x)=\nabla u_{1}^{\alpha}+O(1),\quad\alpha=d+1,\cdots,d(d+1)/2,\quad x\in\Omega_{R}.
	\end{align*}
\end{theorem}

We remark that Theorems \ref{coro v1} and \ref{coro v1d+1} also hold true for $v_{2}^{\alpha}$ by replacing $\bar{u}$ with $\underline{u}$, where $\underline{u}\in C^2(\mathbb{R}^d)$ is a scalar function satisfying $\underline{u}=1$ on $\partial{D}_{2}$, $\underline{u}=0$ on $\partial{D}_{1}\cup\partial D$, $\underline{u}=1-\bar{u}$ in $\Omega_{2R}$, and $\|\underline{u}\|_{C^{2}(\Omega\setminus\Omega_{R})}\leq\,C$. In this case we denote the auxiliary functions by $u_{2}^{\alpha}$.

Throughout the paper, unless otherwise stated, we use $C$ to denote some positive constant, whose values may vary from line to line, depending only on $d, \kappa_0, R$, and an upper bound of the $C^{2,\gamma}$ norms of $\partial D_{1}, \partial D_{2}$, and $\partial D$, but not on $\varepsilon$. We call a constant having such dependence a {\it universal constant}.

\section{Main results}\label{sec mainresults}

In this section, we state our main theorems. First, in Subsection \ref{subsec 2-convex}, for the 2-convex inclusion case we introduce a family of blow-up factors, which is a linear functional of boundary data $\varphi$ and then give the asymptotic formulas of $\nabla u$ in Theorem \ref{thm1} for 2D and Theorem \ref{thm2} for 3D under the assumption that the domains satisfy some proper symmetry condition. The results for the general $m$-convex inclusions are presented in Subsection \ref{relationship Du}. We find that the blow-up rates depend on the dimension $d$ and the convexity order $m$, and the blow-up points shift away from the origin with the increasing of $m$. Finally, in Subsection \ref{subsec example}, we give an example that $h_{1}(x')$ and $h_{2}(x')$ are assumed to be $x_{3}=\pm\frac{\kappa}{2}|x_{1}|^{m}\pm\frac{\kappa'}{2}|x_{2}|^{m}$, respectively,  for  $|x'|<2R$, to show how the geometric parameters $\kappa$ and $\kappa'$ influence the blow-up of $\nabla u$. It turns out that the mean curvature of inclusions plays the role.

\subsection{{For the $2$-convex inclusions}}\label{subsec 2-convex}
As mentioned in Subsection \ref{subsec main difficulty}, $|\nabla u_{b}|$ is uniformly bounded with respect to $\varepsilon$. In order to derive the asymptotics of the gradient of solution with respect to the sufficiently small parameter $\varepsilon>0$, we consider the case when two inclusions touch each other. Let $u^{*}$ be the solution of
\begin{align}\label{maineqn touch}
\begin{cases}
\mathcal{L}_{\lambda, \mu}u^{*}=0,\quad&\hbox{in}\ \Omega^{*},\\
u^{*}=\sum_{\alpha=1}^{d(d+1)/2}C_{*}^{\alpha}\psi_{\alpha},&\hbox{on}\ \partial D_{1}^{*}\cup\partial D_{2}^{*},\\
\int_{\partial{D}_{1}^{*}}\frac{\partial{u}^{*}}{\partial\nu}\big|_{+}\cdot\psi_{\beta}+\int_{\partial{D}_{2}^{*}}\frac{\partial{u}^{*}}{\partial\nu}\big|_{+}\cdot\psi_{\beta}=0,&\beta=1,\cdots,d(d+1)/2,\\
u^{*}=\varphi,&\hbox{on}\ \partial{D},
\end{cases}
\end{align}
where
$$D_{1}^{*}:=\{x\in\mathbb R^{d}\big| x+P_{1}\in D_{1}\},\quad D_{2}^{*}:=D_{2},\quad \mbox{and}\quad \Omega^{*}:=D\setminus\overline{D_{1}^{*}\cup D_{2}^{*}},$$
and the constants $C_{*}^{\alpha}$, $\alpha=1,\cdots,d(d+1)/2$, are uniquely determined by minimizing the energy
\begin{equation*}
\int_{\Omega^{*}}\left(\mathbb{C}^{(0)}
e(v),e(v)\right)dx
\end{equation*}
in an admissible function space $$
\mathcal{A}^{*}:=\left\{v\in{H}^{1}(D; \mathbb{R}^{d})
~\big|~e(v)=0~~\mbox{in}~{D}^{*}_{1}\cup{D}^{*}_{2},~\mbox{and}~v=\varphi~\mbox{on}~\partial{D}\right\}.$$ We emphasize that we use the third line of \eqref{maineqn touch}, which implies that total flux of $u^{*}$ along the boundaries of both two inclusions is zero,  to make a distinction with the forth line of \eqref{maineqn}. This kind of limit function for conductivity problem was also used in \cite{g,gn,kleey,lhg}. We define the functionals of $\varphi$, for $\beta=1,\cdots,d(d+1)/2$,
\begin{align}\label{def_bj}
b_{1}^{\beta}[\varphi]:=\int_{\partial{D}_{1}}\frac{\partial{u}_{b}}{\partial\nu}\Big|_{+}\cdot\psi_{\beta}\quad \mbox{and}\quad b_{1}^{*\beta}[\varphi]:=\int_{\partial{D}_{1}^{*}}\frac{\partial u^{*}}{\partial \nu}\Big|_{+}\cdot\psi_{\beta},
\end{align}
and denote
$$\rho_{d}(\varepsilon)=
\begin{cases}
\sqrt{\varepsilon},&\quad d=2,\\
|\log\varepsilon|^{-1},&\quad d=3.
\end{cases}$$
We will prove that $b_{1}^{*\beta}[\varphi]$ are convergent to $b_{1}^{\beta}[\varphi]$, with rates $\rho_{d}(\varepsilon)$.

Here we need some symmetric assumptions on the domain and the boundary data. First, we assume that 
\begin{align*}
({\rm S_{1}}): ~&~D_{1}\cup D_{2}~\mbox{ and~} D ~\mbox{are~ symmetric~ with~ repect~ to~ each~ axis~}x_{i}, \mbox{~and~}\\&~\mbox{ the superplane~} \{x_{d}=0\};
\end{align*}
For boundary data $\varphi$, we assume that in dimension two, $\varphi^{1}(x_{1},x_{2})$ is odd with respect to $x_{2}$ and $\varphi^{2}(x_{1},x_{2})$ is even with respect to $x_{2}$; and in dimension three, $\varphi^{i}(x)$ is odd for $i=1,2,3$, that is, for $x\in\partial D$,
\begin{align*}
&({\rm S_{2}}):~~\varphi^{1}(x_{1},x_{2})=-\varphi^{1}(x_{1},-x_{2}),~ \varphi^{2}(x_{1},x_{2})=\varphi^{2}(x_{1},-x_{2}),\quad\mbox{for}~ d=2;\\
&({\rm S_{3}}):~~	\varphi^{i}(x)=-\varphi^{i}(-x),\quad\,i=1,2,3,~\quad \mbox{for}~d=3.
\end{align*}
We shall use $O(1)$ to denote those quantities satisfying $|O(1)|\leq C$, for some constant independent of $\varepsilon$. We assume that for some $\delta_0>0$,
\begin{equation*}
\delta_0\leq \mu, d\lambda+2\mu\leq\frac{1}{\delta_0}.
\end{equation*}
Then 

\begin{prop}\label{prop converge b}
For $d=2,3$, assume that the domain satisfies \eqref{convexity}, \eqref{h1h3} and $({\rm S_{1}})$. If $\varphi$ satisfies $({\rm S_{2}})$ or $({\rm S_{3}})$, then,
for sufficiently small $\varepsilon>0$, 
\begin{equation}\label{converge b1 12}
b_{1}^{\beta}[\varphi]-b_{1}^{*\beta}[\varphi]=
O(\rho_{d}(\varepsilon)),\quad\beta=1,\cdots,d(d+1)/2.
\end{equation}
\end{prop}

The proof will be given in Section \ref{sec pf prop a11 b}. We would like to point out that the functionals $b_{1}^{*\beta}[\varphi]$ will determine whether or not the blow-up occurs, we call them {\it blow-up factors}. For more details, see the proof of Theorem \ref{thm1} below.

The first asymptotic expression of $\nabla u$ in dimension $d=2$ and for $m=2$ follows. 

\begin{theorem}\label{thm1}
For $d=2$, let $D_{1},D_{2}\subset D$ be defined as above and satisfy \eqref{convexity}, \eqref{h1h3} with $m=2$, and $({\rm S_{1}})$.
Let $u\in{H}^{1}(D;\mathbb{R}^{2})\cap{C}^{1}(\overline{\Omega};\mathbb{R}^{2})$
be the solution to \eqref{maineqn}. Then, if $\varphi$ satisfies  $({\rm S_{2}})$, then for sufficiently small $0<\varepsilon<1/2$ and for $x\in\Omega_{R}$,
\begin{align}\label{asymptotic d=2}
\nabla u(x)=\frac{\sqrt{\kappa}}{\pi}\cdot\sqrt{\varepsilon}\left(\frac{b_{1}^{*1}[\varphi]}{\mu}\nabla u_{1}^{1}+\frac{b_{1}^{*2}[\varphi]}{\lambda+2\mu}\nabla u_{1}^{2}\right)(1+O(\varepsilon^{\frac{\gamma+3}{8}}))+O(1)\|\varphi\|_{C^{0}(\partial D)},
\end{align}
where $u_{1}^{\alpha}$ are specific functions, constructed in \eqref{auxiliary improved}, $\alpha=1,2$.
\end{theorem}

\begin{remark}
If $D_{1}\cup D_{2}$ and $D$ are disks, for example, 
$$D_{1}=B_{1}(0,1+\varepsilon/2), \quad D_{1}=B_{1}(0,-1-\varepsilon/2), \quad\mbox{and}~~ D=B_{3}(0),$$
then the result in Theorem  \ref{thm1} holds for $\varphi=(x_{2},0)^{\mathrm T}$. More generally, we can choose $\varphi=(cx_{2},\tilde cx_{1})^{\mathrm T}$ for some constants $c$ and $\tilde c$.
\end{remark}

\begin{remark}\label{rem thm dim2}
Recalling the definition of $u_{1}^{\alpha}$, \eqref{auxiliary improved}, a direct calculation yields
\begin{align*}
\nabla u_{1}^{\alpha}(x_{1},x_{2})=\frac{1}{\delta(x_{1})}E_{\alpha 2}+O(1),\quad x\in\Omega_{\varepsilon},
\end{align*}
where ``$E_{\alpha\beta}$'' denotes the basic matrix with only one non-zero entry 1 in the $\alpha^{th}$ row and the $\beta^{th}$ column. So in a neighborhood of the origin, say $\Omega_{\varepsilon}$, we find that \eqref{asymptotic d=2} becomes
\begin{align}\label{formula dim2}
\nabla u(x_{1},x_{2})=\frac{\sqrt{\kappa}}{\pi}\mathbb B_2[\varphi]\cdot\frac{\sqrt{\varepsilon}}{\delta(x_{1})}(1+O(\sqrt\varepsilon))+O(1)\|\varphi\|_{C^{0}(\partial D)},
\end{align} 
where 
$$\mathbb B_2[\varphi]:=\frac{1}{\mu}b_{1}^{*1}[\varphi]
E_{12}
+\frac{1}{\lambda+2\mu}b_{1}^{*2}[\varphi]E_{22}.$$
\end{remark}
For this moment, to get the exact asymptotic expression of $\nabla u$ near the origin, we only need to evaluate these boundary integrals $b_{1}^{*\beta}$ in $\mathbb B_2[\varphi]$, which can be computed by numerical method in practical problems.  We would like to point out that they no longer depend on $\varepsilon$ and there is no singularity near the origin. It is completely a computation problem. We leave it to interested readers.

\begin{theorem}\label{thm2}
For $d=3$, let $D_{1},D_{2}\subset D$ be defined as above and satisfy \eqref{convexity}, \eqref{h1h3} with $m=2$, and $({\rm S_{1}})$. Let $u\in{H}^{1}(D;\mathbb{R}^{3})\cap{C}^{1}(\overline{\Omega};\mathbb{R}^{3})$
	is the solution to \eqref{maineqn}. Then if $\varphi$ satisfies $({\rm S_{3}})$, 
	we have, for sufficiently small $0<\varepsilon<1/2$, and for $x\in\Omega_{R}$,  
	\begin{align}\label{asymp Du R3}
	\nabla u
	&=\frac{\kappa}{\pi}\cdot\frac{1}{|\log\varepsilon|}
	\left(\sum_{\alpha=1}^{2}\frac{b_{1}^{*\alpha}[\varphi]}{\mu}\nabla u_{1}^{\alpha}+\frac{b_{1}^{*3}[\varphi]}{\lambda+2\mu}\nabla u_{1}^{3}\right)\left(1+O(|\log\varepsilon|^{-1})\right)\nonumber\\
	&\quad\quad\quad+O(1)\|\varphi\|_{C^{0}(\partial D)},
	\end{align}
where $u_{1}^{\alpha}$ are specific functions, constructed in  \eqref{auxiliary improved dim3}.	
\end{theorem}

\begin{remark}
	
(1) If $\varphi^{i}=x_{i}$ for $i=1,2,3$, $D_{1}\cup D_{2}$ and $D$ are spheres satisfying $({\rm S_{1}})$, then Theorem \ref{thm2} holds true.
	
(2)  Similarly as in Remark \ref{rem thm dim2}, for $x\in\Omega_{\sqrt{\varepsilon}}$, from \eqref{asymp Du R3}, 
\begin{align}\label{formula dim3}
\nabla u(x',x_{3})=\frac{\kappa}{\pi|\log\varepsilon|}\mathbb B_3[\varphi]\cdot\frac{1}{\delta(x')}\left(1+O(|\log\varepsilon|^{-1})\right)+O(1)\|\varphi\|_{C^{0}(\partial D)},
\end{align}
where 
\begin{align}\label{def B_1}
\mathbb B_3[\varphi]:=\frac{1}{\mu}\left(b_{1}^{*1}E_{13}+b_{1}^{*2}E_{23}\right)
+\frac{1}{\lambda+2\mu}b_{1}^{*3}[\varphi]E_{33}.
\end{align}

(3)
It is worth mentioning that the asymptotic formulas \eqref{formula dim2} and \eqref{formula dim3} immediately imply that the blow-up rates of $|\nabla u|$, $\varepsilon^{-1/2}$ for $d=2$, and $(\varepsilon|\log\varepsilon|)^{-1}$ for $d=3$, obtained in \cite{bll1,bll2,l}, are optimal.
\end{remark}

\begin{remark}
If $\varphi=(\varphi^{1},\varphi^{2},\varphi^{3})^{\mathrm T}$ satisfies another symmetric condition
\begin{equation*}
\begin{cases}
\varphi^{1}(x_{1},x_{2},x_{3})=-\varphi^{1}(x_{1},x_{2},-x_{3})=-\varphi^{1}(x_{1},-x_{2},-x_{3}),\\
\varphi^{2}(x_{1},x_{2},x_{3})=-\varphi^{2}(x_{1},x_{2},-x_{3})=\varphi^{2}(x_{1},-x_{2},-x_{3}),\\
\varphi^{3}(x_{1},x_{2},x_{3})=\varphi^{3}(x_{1},x_{2},-x_{3})=\varphi^{3}(x_{1},-x_{2},-x_{3}),
\end{cases}
\end{equation*}
then by replicating the proof of Proposition \ref{conv C alpha} below, we obtain $C_{1}^{\alpha}=C_{2}^{\alpha}$, $\alpha=2,\cdots,6$. Thus, \eqref{asymp Du R3} becomes more simpler as follows
\begin{align*}
\nabla u
=\frac{\kappa}{\pi}\cdot\frac{1}{|\log\varepsilon|}
\frac{1}{\mu}b_{1}^{*1}[\varphi]\nabla u_{1}^{1}\left(1+O(|\log\varepsilon|^{-1})\right)+O(1)\|\varphi\|_{C^{0}(\partial D)}.
\end{align*}
\end{remark}

\subsection{For the $m$-convex inclusions, $m\geq3$}\label{relationship Du}

In the following, we shall reveal the role of the order of the relatively convexity between $\partial D_{1}$ and $\partial D_{2}$, $m$, playing in the asymptotics of $\nabla u$. Define 
$$Q_{d,m}=2\int_{0}^{\infty}\frac{t^{d-2}}{1+t^{m}}\ dt,\quad \quad\widetilde{Q}_{d,m}=2\int_{0}^{\infty}\frac{t^{d}}{1+t^{m}}\ dt;$$
and
\begin{align*}
\rho_{m,2}(\varepsilon)=
\begin{cases}
|\log\varepsilon|^{-1},&~m=3,\\
\varepsilon^{1/4},&~m=4,\\
0,&~ m\geq5,
\end{cases}
~\quad\quad\mbox{and}~\quad~E(\kappa,\varepsilon,m)=
\begin{cases}
\frac{3\kappa}{2}\frac{1}{|\log\varepsilon|},&~m=3,\\
\frac{\kappa^{3/m}\varepsilon^{1-3/m}}{\widetilde{Q}_{2,m}},&~ m\geq4.
\end{cases}
\end{align*}
Then 
\begin{theorem}\label{thmhigher}
For $d=2$, let $D_{1},D_{2}\subset D$ be defined as above and satisfy \eqref{convexity} and \eqref{h1h3} with $m\geq3$. 
Let $u\in{H}^{1}(D;\mathbb{R}^{2})\cap{C}^{1}(\overline{\Omega};\mathbb{R}^{2})$
	be the solution to \eqref{maineqn}. Then for sufficiently small $0<\varepsilon<1/2$ and $x\in \Omega_{R}$, we have
	\begin{align}\label{Du d=2 m}
	\nabla u(x)&=\frac{\kappa^{1/m}\varepsilon^{1-1/m}}{Q_{2,m}}\left(\frac{b_{1}^{*1}[\varphi]}{\mu}\nabla u_{1}^{1}+\frac{b_{1}^{*2}[\varphi]}{\lambda+2\mu}\nabla u_{1}^{2}\right)\left(1+O(\rho_{m,2}(\varepsilon))\right)\nonumber\\
	&\quad\quad+E(\kappa,\varepsilon,m)\frac{b_{1}^{*3}[\varphi]}{\lambda+2\mu}\nabla u_{1}^{3}\left(1+O(\rho_{m,2}(\varepsilon))\right)+O(1)\|\varphi\|_{C^{0}(\partial D)},
	\end{align}
	where $u_{1}^{\alpha}$ are defined in \eqref{auxiliary improved}, $\alpha=1,2$, and $u_{1}^{3}$ is defined in \eqref{auxi fun d+1}.
\end{theorem}

\begin{remark}
	We would like to remark that the pointwise upper bound estimates of $|\nabla u|$ in \cite{hl} imply that when $m\geq d+1$, the maximum of the upper bounds obtain at $x\in\Omega_{R}$ and $|x'|=c\varepsilon^{1/m}$ for some positive constant $c>0$. Therefore,  for $x\in \Omega_{\varepsilon^{1/(m(m-1))}}$, recalling the definition of $u_{1}^{\alpha}$, we have 
	\begin{align*}
	\nabla u(x_{1},x_{2})&=
	\frac{\kappa^{1/m}\varepsilon^{1-1/m}}{Q_{2,m}}\mathbb B_{2,\rm{I}}[\varphi]\cdot\frac{1}{\delta(x_{1})}\left(1+O(\rho_{m,2}(\varepsilon))\right)\\
	&\quad+E(\kappa,\varepsilon,m)\mathbb B_{2,\rm{II}}[\varphi]\cdot\frac{x_{1}}{\delta(x_{1})}\left(1+O(\rho_{m,2}(\varepsilon))\right)+O(1)\|\varphi\|_{C^{0}(\partial D)},
	\end{align*}
	where 
	\begin{align*}
	\mathbb B_{2,\text{I}}[\varphi]:=\frac{1}{\mu}b_{1}^{*1}[\varphi]
	E_{12}
	+\frac{1}{\lambda+2\mu}b_{1}^{*2}[\varphi]E_{22},\quad \mathbb B_{2,\text{II}}[\varphi]:=\frac{1}{\lambda+2\mu}b_{1}^{*3}[\varphi]
	E_{22}.
	\end{align*}
\end{remark}

\begin{remark}
	If $\mathbb B_{2,\text{I}}[\varphi]=0$ for some $\varphi$, then the concentration mechanism of the stress is determined by $\mathbb B_{2,\text{II}}[\varphi]$. Thus, Theorem \ref{thmhigher}, combining with the upper bounds in \cite{hl}, implies that the blow-up occurs at the segments $\mathcal{S}:=\{(\bar{x}_{1},\bar{x}_{2})\in\Omega_{R}\big|~~  |\bar{x}_{1}|=\varepsilon^{1/m}\}$, with blow-up rate $\frac{1}{|\log\varepsilon|\varepsilon^{2/3}}$ when $m=3$ and $\varepsilon^{-\frac{2}{m}}$ when $m\geq4$. Consequently, the gradient will not blow up any more and $\mathcal{S}$ will be more and more away from the shortest segment $\overline{P_{1}P_{2}}$ as $m$ goes to infinity. From \eqref{convexity} with $R<1$, we can see that when $m\rightarrow\infty$, the boundaries of $\partial D_{1}$ and $\partial D_{2}$ parallel. However, in fact, it was showed in \cite{hl} that there is no blow-up in this situation. The result of Theorem \ref{thmhigher} describes this diffuse process of the stress concentration phenomenon when $m$ changes from $2$ to $+\infty$.
\end{remark}

Finally, when $d=3$, we define
\begin{align*}
\rho_{m,3}(\varepsilon)=
\begin{cases}
|\log\varepsilon|^{-1},&~m=4,\\
\varepsilon^{1-4/m},&~ 5\leq m\leq 7,\\
0,&~ m\geq 8;
\end{cases}\quad\mbox{and}~
F(\kappa,\varepsilon,m)=
\begin{cases}
\frac{2\kappa}{\pi|\log\varepsilon|},&~ m=4,\\
\frac{\kappa^{4/m}\varepsilon^{1-4/m}}{\pi \widetilde{Q}_{3,m}},&~ m\geq 5.
\end{cases}
\end{align*}
Then 
\begin{theorem}\label{thmhigher2}
For $d=3$, let $D_{1},D_{2}\subset D$ be defined as above and satisfy \eqref{convexity} and \eqref{h1h3} with $m\geq3$. Let $u\in{H}^{1}(D;\mathbb{R}^{3})\cap{C}^{1}(\overline{\Omega};\mathbb{R}^{3})$ be the solution to \eqref{maineqn}. 
	
(i) When $m=3$, we assume that $D_{1}\cup D_{2}$ and $D$ satisfies $({\rm S_{1}})$. Then if $\varphi$ satisfies 
	$({\rm S_{3}})$, then for sufficiently small $0<\varepsilon<1/2$, and $x\in\Omega_{R}$,  
	\begin{align}\label{formula 33}
	\nabla u(x)&=
	\frac{\kappa^{2/3}}{\pi}\cdot\frac{\varepsilon^{1/3}}{ Q_{3,3}}\left(\sum_{\alpha=1}^{2}\frac{b_{1}^{*\alpha}[\varphi]}{\mu}\nabla u_{1}^{\alpha}+\frac{b_{1}^{*3}[\varphi]}{\lambda+2\mu}\nabla u_{1}^{3}\right)\Big(1+O(\varepsilon^{1/3})\Big)\nonumber\\
	&\quad\quad\quad+O(1)\|\varphi\|_{C^{0}(\partial D)},
	\end{align}
	where $u_{1}^{\alpha}$ are specific functions, constructed in \eqref{auxiliary improved dim3}.	
	
	(ii) When $m\geq4$, for sufficiently small $0<\varepsilon<1/2$ and $x\in\Omega_{R}$,
	\begin{align}\label{formula d3m4}
	\nabla u(x)
	&=\frac{\kappa^{2/m}\varepsilon^{1-2/m}}{\pi Q_{3,m}}\left(\sum_{\alpha=1}^{2}\frac{b_{1}^{*\alpha}[\varphi]}{\mu}\nabla u_{1}^{\alpha}+\frac{b_{1}^{*3}[\varphi]}{\lambda+2\mu}\nabla u_{1}^{3}\right)\left(1+O(\rho_{m,3}(\varepsilon))\right)\nonumber\\
	&\quad\quad+F(\kappa,\varepsilon,m)\left(\frac{b_{1}^{*4}[\varphi]}{\mu}\nabla u_{1}^{4}+\sum_{\alpha=5}^{6}\frac{b_{1}^{*\alpha}[\varphi]}{\lambda+2\mu}\nabla u_{1}^{\alpha}\right)\left(1+O(\rho_{m,3}(\varepsilon))\right)\nonumber\\
	&\quad\quad+O(1)\|\varphi\|_{C^{0}(\partial D)},
	\end{align}
	where $u_{1}^{\alpha}$ are specific functions, constructed in \eqref{auxiliary improved dim3} and \eqref{auxi fun d+1}, respectively.
\end{theorem}

\begin{remark}
The above results, together with Theorem \ref{thmhigher}, imply that the blow-up rate of $|\nabla u|$ depends on the space dimension $d$, the convexity order $m$, and the first term's coefficient $\kappa$.  Furthermore, when $m\geq d+1$, we do not need the symmetric assumption about domains and the boundary data, since in this case, $C_{1}^{\alpha}-C_{2}^{\alpha}$ tends to zero as $\varepsilon\rightarrow0$; see the proof of Theorem \ref{thmhigher} and Theorem \ref{thmhigher2} in Section \ref{proof general}. For more generalized cases, we refer readers to Example \ref{example R3} below.  Our method can also be applied to study the cases in dimensions $d\geq4$, which is left to the interested readers.
\end{remark}

\subsection{An example}\label{subsec example}
Finally, we give an example in dimension three to show the dependence of the constants in above asymptotics upon the mean curvature of the inclusions more precisely.
\begin{example}\label{example R3}
For $m\geq2$, we denote the top and bottom boundaries of  $\Omega_{R}$ as
\begin{align*}
\Gamma_{R}^{+}=\left\{x\in\mathbb R^3~\big|~ x_{3}=\frac{\varepsilon}{2}+\frac{\kappa}{2}|x_{1}|^m+\frac{\kappa'}{2}|x_{2}|^m\right\}
\end{align*}
and
\begin{align*}
\Gamma_{R}^{-}=\left\{x\in\mathbb R^3~\big|~ x_{3}=-\frac{\varepsilon}{2}-\frac{\kappa}{2}|x_{1}|^m-\frac{\kappa'}{2}|x_{2}|^m\right\},
\end{align*}
where $\kappa$ and $\kappa'$ are two positive constants, may different. For $\theta\in[0,2\pi]$, denote
$$E_{m}(\theta):=\sin^{\frac{2}{m}-1}\theta\cos^{\frac{2}{m}+1}\theta+\cos^{\frac{2}{m}-1}\theta\sin^{\frac{2}{m}+1}\theta,$$
$$F_{m}(\theta):=\left(\frac{\cos^{2}\theta}{\kappa}\right)^{2/m}+\left(\frac{\sin^{2}\theta}{\kappa'}\right)^{2/m},$$
and
$$\quad G_{m}:=\int_{0}^{2\pi}E_{m}(\theta)\ d\theta,\quad \widetilde{G}_{m}:=\int_{0}^{2\pi}E_{m}(\theta)F_{m}(\theta)\ d\theta.$$
Then the results in Theorems \ref{thm2} and \ref{thmhigher2} hold true, except that
	
(i) if $m=2$, $\kappa$ in \eqref{asymp Du R3} is replaced by $\sqrt{\kappa\kappa'}$, the square root of the relative Gauss curvature of $\partial D_{1}$ and $\partial D_{2}$; if $m=3$, $\frac{\kappa^{2/3}}{\pi}$ in \eqref{formula 33} becomes $\frac{3(\kappa\kappa')^{1/3}}{2 G_{3}}$;
	
(ii) if $m\geq4$, the terms $\frac{\kappa^{2/m}}{\pi}$ and $\frac{\kappa^{4/m}}{\pi}$ in \eqref{formula d3m4} become $\frac{m(\kappa\kappa')^{1/m}}{G_{m}}$ and $\frac{m(\kappa\kappa')^{2/m}}{\widetilde{G}_{m}}$, respectively.
\end{example}

\section{Proof of Proposition \ref{prop a11} and Proposition \ref{prop converge b}}\label{sec pf prop a11 b}
This section is devoted to proving Proposition \ref{prop a11} and Proposition \ref{prop converge b}, for $d=2,3$, which are two main ingredients to establish our asymptotic formulas of $\nabla u$. Actually, our argument also holds in higher dimensions $d\geq4$ with a slight modification. We first need some preliminary estimates on $v_{1}^{\alpha}$, $\alpha=1,\cdots,d(d+1)/2$.

\subsection{Auxiliary estimates for $v_{1}^{\alpha}$}

Suppose $v_{1}^{*\alpha}$ satisfies
\begin{equation}\label{equ_vi*alpha}
\begin{cases}
\mathcal{L}_{\lambda,\mu}v_{1}^{*\alpha}=0,&\mathrm{in}~\Omega^{*},\\
v_{1}^{*\alpha}=\psi_{\alpha},&\mathrm{on}~\partial{D}_{1}^{*}\setminus\{0\},\\
v_{1}^{*\alpha}=0,&\mathrm{on}~\partial{D_{2}^{*}}\cup\partial{D},
\end{cases}
\quad\alpha=1,\cdots,d(d+1)/2.
\end{equation}	
We will prove that $v_{1}^{\alpha}\rightarrow v_{1}^{*\alpha}$, for each $\alpha$, with proper convergence rates.

Define
$$V:=D\setminus\overline{D_{1}\cup D_{2}\cup D_{1}^{*}\cup D_{2}^{*}},$$
see Figure \ref{touch}, and
$$\mathcal{C}_{r}:=\left\{x\in\mathbb R^{d}\big| |x'|<r,~-\frac{\varepsilon}{2}+2\min_{|x'|=r}h_{2}(x')\leq x_{d}\leq\frac{\varepsilon}{2}+2\max_{|x'|=r}h_{1}(x')\right\},\quad r<R.$$
Then we have
\begin{lemma}\label{lem difference v11}
	Let $v_{1}^{\alpha}$ and $v_{1}^{*\alpha}$ satisfy \eqref{equ_v1} and \eqref{equ_vi*alpha}, respectively. Then we have
	\begin{align}\label{difference v11}
	|(v_{1}^{\alpha}-v_{1}^{*\alpha})(x)|\leq C\varepsilon^{1/2},\quad \alpha=1,\cdots,d,\quad x\in V\setminus \mathcal{C}_{\varepsilon^{1/4}},
	\end{align}
	and
	\begin{align}\label{difference v13}
	|(v_{1}^{\alpha}-v_{1}^{*\alpha})(x)|\leq C\varepsilon^{2/3},\quad \alpha=d+1,\cdots,d(d+1)/2,\quad x\in V\setminus \mathcal{C}_{\varepsilon^{1/3}},
	\end{align}
	where $C$ is a {\it universal constant}. 
\end{lemma}

\begin{proof}
For $\varepsilon=0$, we  define similarly the auxiliary functions, $\bar{u}^{*}$ and $u_{1}^{*\alpha}$, as limits of $\bar{u}$ and $u_{1}^{\alpha}$, where $\bar{u}^{*}\in C^2(\mathbb{R}^d)$ satisfies $\bar{u}^{*}=1$ on $\partial{D}_{1}^{*}$, $\bar{u}^{*}=0$ on $\partial{D}_{2}^{*}\cup\partial{D}$ and
	\begin{equation*}
	\bar{u}^{*}=\frac{x_{d}-h_2(x')}{h_1(x')-h_2(x')},\quad\hbox{in}\ \Omega_{2R}^{*},\quad \|\bar{u}^{*}\|_{C^{2}(\Omega^{*}\setminus\Omega_{R}^{*})}\leq\,C,
	\end{equation*}
	where $\Omega_r^{*}:=\left\{(x',x_{d})\in \mathbb{R}^{d}~\big|~h_{2}(x')<x_{d}<h_{1}(x'),~|x'|<r\right\}$, $r<R$. A direct computation yields
	\begin{equation}\label{es bar u}
	|\partial_{x'}\bar{u}(x)|\leq\frac{C|x'|}{\varepsilon+|x'|^2},\qquad \partial_{x_{d}}\bar{u}(x)=\frac{1}{\delta(x')},\quad x\in\Omega_{R},
	\end{equation}
	and 
	\begin{equation}\label{es bar u*}
	|\partial_{x'}\bar{u}^{*}(x)|\leq\frac{C}{|x'|},\quad \frac{1}{C|x'|^2}\leq\partial_{x_{d}}\bar{u}^{*}(x)\leq\frac{C}{|x'|^2},\quad x\in\Omega_R^{*}.
	\end{equation}
Recalling the construction of $u_{1}^{\alpha}$ in \eqref{auxiliary improved}, \eqref{auxiliary improved dim3}, and \eqref{auxi fun d+1}, we can construct $u_{1}^{*\alpha}$ in the same way.

\begin{figure}[t]
	\begin{minipage}[c]{0.9\linewidth}
		\centering
		\includegraphics[width=2.5in]{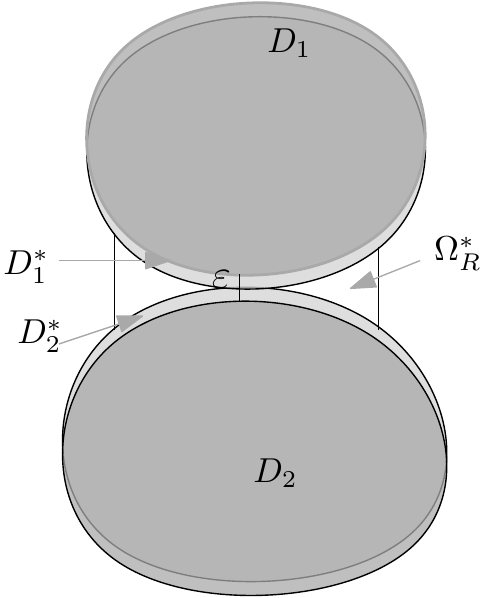}
		\caption{Region \small $\Omega_{R}^{*}$.}
		\label{touch}
	\end{minipage}
\end{figure}

{\bf Case 1.} $\alpha=1,\cdots,d$. By using  \eqref{auxiliary improved}, \eqref{auxiliary improved dim3}, \eqref{es bar u} and \eqref{es bar u*}, we have for $x\in \Omega_R^{*}$,
\begin{align}\label{difference bar u11}
|\partial_{x_{d}}(u_{1}^{\alpha}-u_{1}^{*\alpha})|&\leq|\partial_{x_{d}}(\bar{u}_{1}^{\alpha}-\bar{u}_{1}^{*\alpha})|+|\partial_{x_{d}}(\tilde{u}_{1}^{\alpha}-\tilde{u}_{1}^{*\alpha})|\nonumber\\
&\leq\frac{C\varepsilon}{|x'|^{2}(\varepsilon+|x'|^{2})}+\frac{C\varepsilon}{|x'|^{3}}\leq \frac{C\varepsilon}{|x'|^{2}(\varepsilon+|x'|^{2})}\left(1+\frac{\varepsilon}{|x'|}\right).
\end{align}
	
	Notice that $v_{1}^{\alpha}-v_{1}^{*\alpha}$ satisfies
	\begin{equation*}
	\begin{cases}
	\mathcal{L}_{\lambda,\mu}(v_{1}^{\alpha}-v_{1}^{*\alpha})=0,&\mathrm{in}~V,\\
	v_{1}^{\alpha}-v_{1}^{*\alpha}=\psi_{\alpha}-v_{1}^{*\alpha},&\mathrm{on}~\partial{D}_{1}\setminus D_{1}^{*},\\
	v_{1}^{\alpha}-v_{1}^{*\alpha}=-v_{1}^{*\alpha},&\mathrm{on}~\partial{D}_{2}\setminus D_{2}^{*},\\
	v_{1}^{\alpha}-v_{1}^{*\alpha}=v_{1}^{\alpha}-\psi_{\alpha},&\mathrm{on}~\partial{D}_{1}^{*}\setminus(D_{1}\cup\{0\}),\\
	v_{1}^{\alpha}-v_{1}^{*\alpha}=v_{1}^{\alpha},&\mathrm{on}~\partial{D}_{2}^{*}\setminus D_{2},\\
	v_{1}^{\alpha}-v_{1}^{*\alpha}=0,&\mathrm{on}~\partial{D}.
	\end{cases}
	\end{equation*}
	By \eqref{mv1--bounded1}, we have
	\begin{equation}\label{out v11 *}
	|\partial_{x_{d}}v_{1}^{*\alpha}|\leq C,\quad\mbox{in}~\Omega^{*}\setminus\Omega_{R}^{*}.
	\end{equation}
	For $x\in\partial{D}_{1}\setminus D_{1}^{*}\subset\Omega^{*}\setminus\Omega_{R}^{*}$ (see Figure \ref{touch}), by using mean value theorem and \eqref{out v11 *}, 
	\begin{align}\label{partial D11}
	|(v_{1}^{\alpha}-v_{1}^{*\alpha})(x',x_{d})|&=|(\psi_{\alpha}-v_{1}^{*\alpha})(x',x_{d})|\nonumber\\
	&=|v_{1}^{*\alpha}(x',x_{d}-\varepsilon/2)-v_{1}^{*\alpha}(x',x_{d})|\leq C\varepsilon.
	\end{align}
	For $x\in\partial{D}_{1}^{*}\setminus(D_{1}\cup\mathcal{C}_{\varepsilon^{\theta}})$, by mean value theorem again and Theorem \ref{coro v1}, we have
	\begin{align}\label{partial D11*}
	|(v_{1}^{\alpha}-v_{1}^{*\alpha})(x',x_{d})|&=|(v_{1}^{\alpha}-\psi_{\alpha})(x',x_{d})|\nonumber\\
	&=|v_{1}^{\alpha}(x',x_{d})-v_{1}^{\alpha}(x',x_{d}+\varepsilon/2)|\leq\frac{C\varepsilon}{\varepsilon+|x'|^{2}}\leq C\varepsilon^{1-2\theta},
	\end{align}
	where $0<\theta<1$ is some constant to be determined later. Similarly, for $x\in\partial{D}_{2}\setminus D_{2}^{*}$, we have
	\begin{align}\label{partial D21}
	|(v_{1}^{\alpha}-v_{1}^{*\alpha})(x',x_{d})|\leq C\varepsilon,
	\end{align}
	and for $x\in\partial{D}_{2}^{*}\setminus(D_{2}\cup\mathcal{C}_{\varepsilon^{\theta}})$, we have
	\begin{align}\label{partial D21*}
	|(v_{1}^{\alpha}-v_{1}^{*\alpha})(x',x_{d})|\leq C\varepsilon^{1-2\theta}.
	\end{align}
	
	For $x\in\Omega_{R}^{*}$ with $|x'|=\varepsilon^{\theta}$, it follows from Theorem \ref{coro v1} and \eqref{difference bar u11} that
	\begin{align}\label{partial x2 v11}
	\left|\partial_{x_{d}}(v_{1}^{\alpha}-v_{1}^{*\alpha})(x',x_{d})\right|
	&=\left|\partial_{x_{d}}(v_{1}^{\alpha}-u_{1}^{\alpha})+\partial_{x_{d}}(u_{1}^{\alpha}-u_{1}^{*\alpha})
	+\partial_{x_{d}}(u_{1}^{*\alpha}-v_{1}^{*\alpha})\right|(x',x_{d})\nonumber\\
	&\leq C\left(1+\frac{\varepsilon}{|x'|^{2}(\varepsilon+|x'|^{2})}\right)\leq C\left(1+\frac{1}{\varepsilon^{4\theta-1}}\right).
	\end{align}
	Thus, for $x\in\Omega_{R}^{*}$ with $|x'|=\varepsilon^{\theta}$, by using the triangle inequality, \eqref{partial D21*}, the mean value theorem, and \eqref{partial x2 v11}, we have
	\begin{align}\label{es v11*}
	\left|(v_{1}^{\alpha}-v_{1}^{*\alpha})(x',x_{d})\right|&\leq\left|(v_{1}^{\alpha}-v_{1}^{*\alpha})(x',x_{d})-(v_{1}^{\alpha}-v_{1}^{*\alpha})(x',h_{2}(x')\right|
	+C\varepsilon^{1-2\theta}\nonumber\\
	&\leq\left|\partial_{x_{d}}(v_{1}^{\alpha}-v_{1}^{*\alpha})\right|\cdot(h_{1}(x')-h_{2}(x'))\Big|_{|x'|=\varepsilon^{\theta}}+C\varepsilon^{1-2\theta}\nonumber\\
	&\leq C\left(\varepsilon^{2\theta}+\varepsilon^{1-2\theta}\right).
	\end{align}
By taking $2\theta=1-2\theta$, we get $\theta=\frac{1}{4}$. Substituting it into \eqref{partial D11*}, \eqref{partial D21*}, and \eqref{es v11*}, and using \eqref{partial D11}, \eqref{partial D21}, and $v_{1}^{\alpha}-v_{1}^{*\alpha}=0$ on $\partial D$, we obtain
	\begin{align}\label{difference v11 boundary}
	|v_{1}^{\alpha}-v_{1}^{*\alpha}|\leq C\varepsilon^{1/2},\quad\mbox{on}~\partial{(V\setminus\mathcal{C}_{\varepsilon^{1/4}})}.
	\end{align}
	Applying the maximum principle for Lam\'{e} systems in $V\setminus\mathcal{C}_{\varepsilon^{1/4}}$, see \cite{mmn}, we get
	\begin{align*}
	|(v_{1}^{\alpha}-v_{1}^{*\alpha})(x)|\leq C\varepsilon^{1/2},\quad\mbox{in}~ V\setminus \mathcal{C}_{\varepsilon^{1/4}}.
	\end{align*}
	
	{\bf Case 2.} $\alpha=d+1,\cdots,d(d+1)/2$. By using  \eqref{auxi fun d+1}, \eqref{es bar u}, and \eqref{es bar u*}, we obtain for $x\in \Omega_R^{*}$,
	\begin{align}\label{difference bar u13}
	|\partial_{x_{d}}(u_{1}^{\alpha}-u_{1}^{*\alpha})|&\leq|\partial_{x_{d}}(\bar{u}-\bar{u}^{*})\psi_{\alpha}|+|(\bar{u}-\bar{u}^{*})\partial_{x_{d}}\psi_{\alpha}|\nonumber\\
	&\leq C\left(1+\frac{\varepsilon}{|x'|(\varepsilon+|x'|^{2})}\right).
	\end{align}
	By using Theorem \ref{coro v1d+1}, we find that \eqref{partial D11*} and \eqref{partial D21*} become
	\begin{align*}
	|(v_{1}^{\alpha}-v_{1}^{*\alpha})(x',x_{d})|\leq C\varepsilon^{1-\theta},\quad x\in\Big(\partial{D}_{1}^{*}\setminus(D_{1}\cup\mathcal{C}_{\varepsilon^{\theta}})\Big)\cup\Big(\partial{D}_{2}^{*}\setminus(D_{2}\cup\mathcal{C}_{\varepsilon^{\theta}})\Big),
	\end{align*}
	where $0<\theta<1$ is some constant to be fixed later. For $x\in\Omega_{R}^{*}$ with $|x'|=\varepsilon^{\theta}$,  \eqref{es v11*} becomes
	\begin{align*}
	\left|(v_{1}^{\alpha}-v_{1}^{*\alpha})(x',x_{d})\right|\leq C\left(\varepsilon^{2\theta}+\varepsilon^{1-\theta}\right).
	\end{align*}
	By taking $2\theta=1-\theta$, we get $\theta=\frac{1}{3}$. We henceforth obtain
	\begin{align*}
	|v_{1}^{\alpha}-v_{1}^{*\alpha}|\leq C\varepsilon^{2/3},\quad\mbox{on}~(V\setminus\mathcal{C}_{\varepsilon^{1/3}}).
	\end{align*}
	The proof of Lemma \ref{lem difference v11} is finished.
\end{proof}

\subsection{Convergence of $\frac{C_{1}^{\alpha}+C_{2}^{\alpha}}{2}-C_{*}^{\alpha}$, $\alpha=1,\cdots,d(d+1)/2$}
We use the notation in \cite{l} in the following. 

It follows from \eqref{nablau_dec} and the forth line of \eqref{maineqn} that
\begin{align}\label{C1C2_d}
\left\{
\begin{aligned}
\sum_{\alpha=1}^{d(d+1)/2}C_{1}^{\alpha}a_{11}^{\alpha\beta}+\sum_{\alpha=1}^{d(d+1)/2}C_{2}^{\alpha}a_{21}^{\alpha\beta}-\tilde{b}_{1}^{\beta}&=0,\\
\sum_{\alpha=1}^{d(d+1)/2}C_{1}^{\alpha}a_{12}^{\alpha\beta}+\sum_{\alpha=1}^{d(d+1)/2}C_{2}^{\alpha}a_{22}^{\alpha\beta}-\tilde{b}_{2}^{\beta}&=0,
\end{aligned}
\right.\quad\quad~~\beta=1,\cdots,d(d+1)/2,
\end{align}
where $a_{ij}^{\alpha\beta}$ is defined in \eqref{def_aij}, and 
\begin{equation}\label{bjbeta}
\tilde{b}_{j}^{\beta}=\int_{\partial D_j}\frac{\partial v_0}{\partial \nu}\Big|_+\cdot \psi_\beta.
\end{equation}  

For the first equation of \eqref{C1C2_d}, we have
\begin{equation*}
\sum_{\alpha=1}^{3}(C_{1}^{\alpha}+C_{2}^{\alpha})a_{11}^{\alpha\beta}+\sum_{\alpha=1}^{3}C_{2}^{\alpha}(a_{21}^{\alpha\beta}-a_{11}^{\alpha\beta})-\tilde{b}_{1}^{\beta}=0,
\end{equation*}
and
\begin{equation*}
\sum_{\alpha=1}^{3}(C_{1}^{\alpha}+C_{2}^{\alpha})a_{21}^{\alpha\beta}+\sum_{\alpha=1}^{3}C_{1}^{\alpha}(a_{11}^{\alpha\beta}-a_{21}^{\alpha\beta})-\tilde{b}_{1}^{\beta}=0.
\end{equation*}
Adding these two equations together leads to
\begin{equation}\label{C1C2_1}
\sum_{\alpha=1}^{3}(C_{1}^{\alpha}+C_{2}^{\alpha})(a_{11}^{\alpha\beta}+a_{21}^{\alpha\beta})+\sum_{\alpha=1}^{3}(C_{1}^{\alpha}-C_{2}^{\alpha})(a_{11}^{\alpha\beta}-a_{21}^{\alpha\beta})-2\tilde{b}_{1}^{\beta}=0.
\end{equation}
Similarly, for the second equation of \eqref{C1C2_d}, we have 
\begin{equation}\label{C1C2_22}
\sum_{\alpha=1}^{3}(C_{1}^{\alpha}+C_{2}^{\alpha})(a_{12}^{\alpha\beta}+a_{22}^{\alpha\beta})+\sum_{\alpha=1}^{3}(C_{1}^{\alpha}-C_{2}^{\alpha})(a_{12}^{\alpha\beta}-a_{22}^{\alpha\beta})-2\tilde{b}_{2}^{\beta}=0.
\end{equation}
Then from \eqref{C1C2_1} and \eqref{C1C2_22}, we obtain
\begin{equation}\label{C1+C2_1}
\sum_{\alpha=1}^{3}\frac{C_{1}^{\alpha}+C_{2}^{\alpha}}{2}a^{\alpha\beta}
+\sum_{\alpha=1}^{3}\frac{C_{1}^{\alpha}-C_{2}^{\alpha}}{2}(a_{11}^{\alpha\beta}-a_{22}^{\alpha\beta}+a_{12}^{\alpha\beta}-a_{21}^{\alpha\beta})-(\tilde{b}_{1}^{\beta}+\tilde{b}_{2}^{\beta})=0.
\end{equation}
where 
$$a^{\alpha\beta}=\sum_{i,j=1}^{2}a_{ij}^{\alpha\beta}=-\left(\int_{\partial{D}_{1}}\frac{\partial{v}^{\alpha}}{\partial\nu}\big|_{+}\cdot\psi_{\beta}+\int_{\partial{D}_{2}}\frac{\partial{v}^{\alpha}}{\partial\nu}\big|_{+}\cdot\psi_{\beta}\right).$$
and $v^{\alpha}:=v_{1}^{\alpha}+v_{2}^{\alpha}$ satisfies
\begin{equation}\label{def valpha}
\begin{cases}
\mathcal{L}_{\lambda,\mu}v^{\alpha}=0,&\mathrm{in}~\Omega,\\
v^{\alpha}=\psi_{\alpha},&\mathrm{on}~\partial{D}_{1}\cup\partial D_{2},\\
v^{\alpha}=0,&\mathrm{on}~\partial{D}.
\end{cases}
\end{equation}

Let $v^{*\alpha}:=v_{1}^{*\alpha}+v_{2}^{*\alpha}$ satisfy
\begin{equation}\label{def valpha*}
\begin{cases}
\mathcal{L}_{\lambda,\mu}v^{*\alpha}=0,&\mathrm{in}~\Omega^{*},\\
v^{*\alpha}=\psi_{\alpha},&\mathrm{on}~\partial{D}_{1}^{*}\cup\partial{D}_{2}^{*},\\
v^{*\alpha}=0,&\mathrm{on}~\partial{D}.
\end{cases}
\end{equation}
and
$v_{0}^{*}$ satisfy
\begin{equation}\label{equ_v3*}
\begin{cases}
\mathcal{L}_{\lambda,\mu}v_{0}^{*}=0,&\mathrm{in}~\Omega^{*},\\
v_{0}^{*}=0,&\mathrm{on}~\partial{D}_{1}^{*}\cup\partial{D_{2}^{*}},\\
v_{0}^{*}=\varphi,&\mathrm{on}~\partial{D},
\end{cases}
\end{equation}
By the definition of $u^{*}$ in \eqref{maineqn touch}, we find that
$$u^{*}=\sum_{\alpha=1}^{3}C_{*}^{\alpha}v^{*\alpha}+v_{0}^{*}.$$
From the third line of \eqref{maineqn touch}, we have
\begin{align}\label{equ_C*alpha}
&\sum_{\alpha=1}^{3}C_{*}^{\alpha}a_{*}^{\alpha\beta}-(\tilde{b}_{1}^{*\beta}+\tilde{b}_{2}^{*\beta})=0,\quad\beta=1,2,3,
\end{align}
where
\begin{align}\label{b*jbeta}
a_{*}^{\alpha\beta}=-\int_{\partial{D}_{1}^{*}\cup\partial{D}_{2}^{*}}\frac{\partial{v}^{*\alpha}}{\partial\nu}\big|_{+}\cdot\psi_{\beta},\quad
\mbox{and}~~\tilde{b}_{j}^{*\beta}=\int_{\partial{D}_{j}^{*}}\frac{\partial{v}_{0}^{*}}{\partial\nu}\big|_{+}\cdot\psi_{\beta},\quad j=1,2.
\end{align}

We next use the symmetry properties of the domain and the boundary data to prove the convergence of $\frac{C_{1}^{\alpha}+C_{2}^{\alpha}}{2}-C_{*}^{\alpha}$, Proposition \ref{conv C alpha} below. Before this, we first prove the following two lemmas. 

\begin{lemma}\label{es b1 b1* beta=1}
	Let $\tilde{b}_{1}^{\beta}$ and $\tilde{b}_{1}^{*\beta}$ be defined in \eqref{bjbeta} and \eqref{b*jbeta}, respectively. Then 
	\begin{equation}\label{est v0-v0*}
	\left|\tilde{b}_{1}^{\beta}-\tilde{b}_{1}^{*\beta}\right|\leq C\|\varphi\|_{L^{1}(\partial D)}
	\begin{cases}
	\varepsilon^{1/2},~\beta=1,\cdots,d,\\
	\varepsilon^{2/3},~\beta=d+1,\cdots,\frac{d(d+1)}{2}.
	\end{cases}
	\end{equation}
\end{lemma}

\begin{proof}
	It follows from \eqref{equ_v3}, \eqref{equ_v3*}, and the integration by parts that
	\begin{align*}
	\tilde{b}_{1}^{\beta}=\int_{\partial D_{1}}\frac{\partial v_{0}}{\partial \nu}\Big|_{+}\cdot\psi_{\beta}&=\int_{\partial D_{1}}\frac{\partial v_{0}}{\partial \nu}\Big|_{+}\cdot v_{1}^{\beta}=-\int_{\partial D}\frac{\partial v_{1}^{\beta}}{\partial \nu}\Big|_{+}\cdot\varphi,
	\end{align*}
	and
	\begin{align*}
	\tilde{b}_{1}^{*\beta}=\int_{\partial D_{1}^{*}}\frac{\partial v_{0}^{*}}{\partial \nu}\Big|_{+}\cdot\psi_{\beta}&=-\int_{\partial D}\frac{\partial v_{1}^{*\beta}}{\partial \nu}\Big|_{+}\cdot\varphi.
	\end{align*}
	Thus, we have
	\begin{equation}\label{difference v1}
\tilde{b}_{1}^{\beta}-\tilde{b}_{1}^{*\beta}=-\int_{\partial D}\frac{\partial(v_{1}^{\beta}-v_{1}^{*\beta})}{\partial \nu}\Big|_{+}\cdot\varphi.
	\end{equation}
	By using Lemma \ref{lem difference v11} and	the standard boundary gradient estimates for Lam\'e system (see \cite{adn}), we have \eqref{est v0-v0*}.
	We hence finish the proof of Lemma \ref{es b1 b1* beta=1}.
\end{proof}

Similar to Lemma \ref{es b1 b1* beta=1}, we can get 
\begin{lemma}\label{lem valpha1}
	Let $v^{\alpha}$ and $v^{*\alpha}$ be defined by \eqref{def valpha} and \eqref{def valpha*}, respectively, $\alpha=1,\cdots,d(d+1)/2$. Then
	\begin{align}\label{est v11 v11*}
	\left|\int_{\partial D_{1}}\frac{\partial v^{\alpha}}{\partial \nu}\Big|_{+}\cdot\psi_{\beta}-\int_{\partial D_{1}^{*}}\frac{\partial v^{*\alpha}}{\partial \nu}\Big|_{+}\cdot\psi_{\beta}\right|\leq
	C\|\psi_{\alpha}\|_{L^{\infty}(\partial D)}
	\begin{cases}
	\varepsilon^{1/2},~\beta=1,\cdots,d,\\
	\varepsilon^{2/3},~\beta=d+1,\cdots,\frac{d(d+1)}{2}.
	\end{cases}
	\end{align}
\end{lemma}

\begin{proof}
	For $\alpha=1,\cdots,d(d+1)/2$, it follows from \eqref{def valpha} that
	\begin{equation*}
	\begin{cases}
	\mathcal{L}_{\lambda,\mu}(v^{\alpha}-\psi_{\alpha})=0,&\mbox{in}~\Omega,\\
	v^{\alpha}-\psi_{\alpha}=0,&\mbox{on}~\partial{D}_{1}\cup\partial{D}_{2},\\
	v^{\alpha}-\psi_{\alpha}=-\psi_{\alpha},&\mbox{on}~\partial D.
	\end{cases}
	\end{equation*}
	By using the integration by parts, we have for $\alpha=1,\cdots, d(d+1)/2$,
	\begin{align*}
	\int_{\partial D_{1}}\frac{\partial v^{\alpha}}{\partial \nu}\Big|_{+}\cdot\psi_{\beta}=\int_{\partial D_{1}}\frac{\partial (v^{\alpha}-\psi_{\alpha})}{\partial \nu}\Big|_{+}\cdot v_{1}^{\beta}=\int_{\partial D}\frac{\partial v_{1}^{\beta}}{\partial \nu}\Big|_{+}\cdot\psi_{\alpha}.
	\end{align*}
	Similarly,
	\begin{align*}
	\int_{\partial D_{1}^{*}}\frac{\partial v^{*\alpha}}{\partial \nu}\Big|_{+}\cdot\psi_{\beta}=\int_{\partial D}\frac{\partial v_{1}^{*\beta}}{\partial \nu}\Big|_{+}\cdot\psi_{\alpha}.
	\end{align*}
	Hence,
	\begin{equation*}
	\int_{\partial D_{1}}\frac{\partial v^{\alpha}}{\partial \nu}\Big|_{+}\cdot\psi_{\beta}-\int_{\partial D_{1}^{*}}\frac{\partial v^{*\alpha}}{\partial \nu}\Big|_{+}\cdot\psi_{\beta}=\int_{\partial D}\frac{\partial(v_{1}^{\beta}-v_{1}^{*\beta})}{\partial \nu}\Big|_{+}\cdot\psi_{\alpha}.
	\end{equation*}
	Thus, similar to  \eqref{est v0-v0*}, we obtain \eqref{est v11 v11*}. 
	The proof of Lemma \ref{lem valpha1} is finished.	
\end{proof}

\begin{prop}\label{conv C alpha}
For $d=2,3$, assume that $D_{1}\cup D_{2}$ and $D$ satisfies $({\rm S_{1}})$, and $\varphi$ satisfies $({\rm S_{2}})$ or $({\rm S_{3}})$. Let $C_{1}^{\alpha}, C_{2}^{\alpha}$, and $C_{*}^{\alpha}$ be defined in \eqref{Cialpha} and \eqref{maineqn touch}, respectively. Then
\begin{equation}\label{conv C C*}
\left|\frac{C_{1}^{\alpha}+C_{2}^{\alpha}}{2}-C_{*}^{\alpha}\right|\leq C\rho_{d}(\varepsilon),\quad \alpha=1,\cdots,d(d+1)/2.
\end{equation}
\end{prop}
	
\begin{proof}

(1) We first prove the case when $d=2$. 
First, on one hand, using the symmetry of the domain with respect to the origin and the boundary conditions of $v_{i}^{\alpha}$, we have, for $\alpha=1,2$,
\begin{align*}
v_{2}^{\alpha}(x_{1},x_{2})\big|_{\partial D_{1}}=v_{1}^{\alpha}(-x_{1},-x_{2})\big|_{\partial D_{2}}&=0,~~v_{2}^{\alpha}(x_{1},x_{2})\big|_{\partial D_{2}}=v_{1}^{\alpha}(-x_{1},-x_{2})\big|_{\partial D_{1}}=\psi_{\alpha},\\
v_{2}^{\alpha}(x_{1},x_{2})\big|_{\partial D}&=v_{1}^{\alpha}(-x_{1},-x_{2})\big|_{\partial D}=0.
\end{align*}
The using the equation $\mathcal{L}_{\lambda, \mu}v_{i}^{\alpha}=0$, one can see that $$v_{2}^{\alpha}(x_{1},x_{2})=v_{1}^{\alpha}(-x_{1},-x_{2}),\quad\mbox{in}~\Omega,\quad~\alpha=1,2.$$ 
Therefore,
\begin{equation}\label{origin_sym0}
a_{11}^{\alpha\beta}=a_{22}^{\alpha\beta},\quad a_{12}^{\alpha\beta}=a_{21}^{\alpha\beta},\quad~\alpha,\beta=1,2.
\end{equation}
Similarly, by taking 
\begin{align*} 
v_{2}^{3}(x_{1},x_{2})\big|_{\partial D_{1}}=-v_{1}^{3}(-x_{1},-x_{2})\big|_{\partial D_{2}}&=0,~~v_{2}^{3}(x_{1},x_{2})\big|_{\partial D_{2}}=-v_{1}^{3}(-x_{1},-x_{2})\big|_{\partial D_{1}}=\psi_{3},\\
v_{2}^{3}(x_{1},x_{2})\big|_{\partial D}&=-v_{1}^{3}(-x_{1},-x_{2})\big|_{\partial D}=0,
\end{align*}
we have $v_{2}^{3}(x_{1},x_{2})=-v_{1}^{3}(-x_{1},-x_{2})$ are the solutions of \eqref{equ_v1}. Thus, 
\begin{align}\label{origin_sym}
a_{11}^{\alpha 3}=-a_{22}^{\alpha 3},\quad a_{11}^{33}=a_{22}^{33},\quad a_{12}^{\alpha 3}=-a_{21}^{\alpha 3},\quad a_{12}^{3\alpha}=-a_{21}^{3\alpha},\quad\quad~~ \alpha=1,2.
\end{align}

On the other hand, by using the symmetry of the domain with respect to $\{x_{2}=0\}$ and the boundary conditions of $v_{i}^{\alpha}$, we find that 
\begin{align*}
v_{2}^{2}(x_{1},x_{2})=\begin{pmatrix}
(v_{2}^{2})^1(x_{1},x_{2})\\\\
(v_{2}^{2})^2(x_{1},x_{2})
\end{pmatrix}
=\begin{pmatrix}
-(v_{1}^{2})^1(x_{1},-x_{2})\\\\
(v_{1}^{2})^2(x_{1},-x_{2})
\end{pmatrix}
\end{align*}
and
\begin{align*}
v_{2}^{3}(x_{1},x_{2})=\begin{pmatrix}
(v_{2}^{3})^1(x_{1},x_{2})\\\\
(v_{2}^{3})^2(x_{1},x_{2})
\end{pmatrix}
=\begin{pmatrix}
-(v_{1}^{3})^1(x_{1},-x_{2})\\\\
(v_{1}^{3})^2(x_{1},-x_{2})
\end{pmatrix}
\end{align*}
are also the solutions of \eqref{equ_v1}, respectively. Then we have
\begin{align}\label{sym v2}
\left(\mathbb{C}^{(0)}
e(v_{2}^{2}),e(v_{2}^{3})\right)&=\lambda\left(-\partial_{x_{1}}(v_{1}^{2})^1-\partial_{x_{2}}(v_{1}^{2})^2\right)\left(-\partial_{x_{1}}(v_{1}^{3})^1-\partial_{x_{2}}(v_{1}^{3})^2\right)\nonumber\\
&\quad+\mu\Big(2\partial_{x_{1}}(v_{1}^{2})^1\cdot\partial_{x_{1}}(v_{1}^{3})^1+2\partial_{x_{2}}(v_{1}^{2})^2\cdot\partial_{x_{2}}(v_{1}^{3})^2\nonumber\\
&\quad+(\partial_{x_{2}}(v_{1}^{2})^1+\partial_{x_{1}}(v_{1}^{2})^2)(\partial_{x_{2}}(v_{1}^{3})^1+\partial_{x_{1}}(v_{1}^{3})^2)\Big)\nonumber\\
&=\left(\mathbb{C}^{(0)}
e(v_{1}^{2}),e(v_{1}^{3})\right).
\end{align}
Thus, 
\begin{equation}\label{23_sym}
a_{22}^{23}=\int_{\Omega}\left(\mathbb{C}^{(0)}
e(v_{2}^{2}),e(v_{2}^{3})\right)dx=\int_{\Omega}\left(\mathbb{C}^{(0)}
e(v_{1}^{2}),e(v_{1}^{3})\right)dx=a_{11}^{23}.
\end{equation}

Similarly, 
\begin{align*}
v_{2}^{1}(x_{1},x_{2})=\begin{pmatrix}
(v_{2}^{1})^1(x_{1},x_{2})\\\\
(v_{2}^{1})^2(x_{1},x_{2})
\end{pmatrix}
=\begin{pmatrix}
(v_{1}^{1})^1(x_{1},-x_{2})\\\\
-(v_{1}^{1})^2(x_{1},-x_{2})
\end{pmatrix}
\end{align*}
admits \eqref{equ_v1}. Then
\begin{align*}
\left(\mathbb{C}^{(0)}
e(v_{2}^{1}),e(v_{2}^{3})\right)&=\lambda\left(\partial_{x_{1}}(v_{1}^{1})^1+\partial_{x_{2}}(v_{1}^{1})^2\right)\left(-\partial_{x_{1}}(v_{1}^{3})^1-\partial_{x_{2}}(v_{1}^{3})^2\right)\\
&\quad+\mu\Big(-2\partial_{x_{1}}(v_{1}^{1})^1\cdot\partial_{x_{1}}(v_{1}^{3})^1-2\partial_{x_{2}}(v_{1}^{1})^2\cdot\partial_{x_{2}}(v_{1}^{3})^2\\
&\quad+(-\partial_{x_{2}}(v_{1}^{1})^1-\partial_{x_{1}}(v_{1}^{1})^2)(\partial_{x_{2}}(v_{1}^{3})^1+\partial_{x_{1}}(v_{1}^{3})^2)\Big)\\
&=-\left(\mathbb{C}^{(0)}
e(v_{1}^{1}),e(v_{1}^{3})\right).
\end{align*}
Hence,
\begin{equation}\label{13_sym}
a_{22}^{13}=-a_{11}^{13}.
\end{equation}
Similarly, we have
\begin{equation}\label{alphabeta_sym}
\begin{split}
&a_{22}^{12}=-a_{11}^{12},\quad a_{12}^{12}=-a_{21}^{12},\quad a_{12}^{21}=-a_{21}^{21},\\
&a_{12}^{13}=-a_{21}^{13},\quad a_{12}^{31}=-a_{21}^{31},\quad a_{12}^{23}=a_{21}^{23},\quad a_{12}^{32}=a_{21}^{32}.
\end{split}
\end{equation}
Combining \eqref{origin_sym0}--\eqref{alphabeta_sym}, we obtain
\begin{equation*}
a_{11}^{12}=a_{22}^{12}=a_{12}^{12}=a_{21}^{12}=a_{12}^{21}=a_{21}^{21}=a_{11}^{23}=a_{22}^{23}=a_{12}^{23}=a_{21}^{23}=a_{12}^{32}=a_{21}^{32}=0.
\end{equation*}
Therefore, we obtain
\begin{align}\label{matrix}
A:=\begin{pmatrix}
~a_{11}^{\alpha\beta}&a_{21}^{\alpha\beta}~\\\\
~a_{21}^{\alpha\beta}&a_{22}^{\alpha\beta}~
\end{pmatrix}_{\alpha,\beta=1,2,3}=\begin{pmatrix}
~a_{11}^{11}&0&a_{11}^{13}&a_{12}^{11}&0&a_{12}^{13}~\\\\
~0&a_{11}^{22}&0&0&a_{12}^{22}&0~\\\\
~a_{11}^{13}&0&a_{11}^{33}&-a_{12}^{13}&0&a_{12}^{33}~\\\\
~a_{12}^{11}&0&-a_{12}^{13}&a_{11}^{11}&0&-a_{11}^{13}~\\\\
~0&a_{12}^{22}&0&0&a_{11}^{22}&0~\\\\
~a_{12}^{13}&0&a_{12}^{33}&-a_{11}^{13}&0&a_{11}^{33}
\end{pmatrix}.
\end{align}
By $({\rm S_{2}})$, a direct calculation which is similar to \eqref{sym v2} yields
\begin{equation}\label{|symtilb}
\tilde{b}_{1}^{1}=-\tilde{b}_{2}^{1},\quad \tilde{b}_{1}^{2}=\tilde{b}_{2}^{2},\quad \tilde{b}_{1}^{3}=\tilde{b}_{2}^{3}.
\end{equation}

By \eqref{C1C2_d}, we have
$$AX=b,$$
where 
$$X=(C_{1}^1,C_{1}^{2},C_{1}^3,C_{2}^{1},C_{2}^2,C_{2}^{3})^{\mathrm{T}},\quad\mbox{ and }~~b=(\tilde{b}_{1}^{1},\tilde{b}_{1}^{2},\tilde{b}_{1}^{3},\tilde{b}_{2}^{1},\tilde{b}_{2}^{2},\tilde{b}_{2}^{3})^{\mathrm{T}}.$$ By Cramer's rule and a direct calculation, we get
\begin{align*}
C_{1}^{3}&=\frac{(a_{11}^{22})^{2}-(a_{12}^{22})^{2}}{\det{A}}\begin{vmatrix}
~a_{11}^{11}&\tilde{b}_{1}^{1}&a_{12}^{11}&a_{12}^{13}~\\\\
~a_{11}^{13}&\tilde{b}_{1}^{3}&-a_{12}^{13}&a_{12}^{33}~\\\\
~a_{12}^{11}&-\tilde{b}_{1}^{1}&a_{11}^{11}&-a_{11}^{13}~\\\\
~a_{12}^{13}&\tilde{b}_{1}^{3}&-a_{11}^{13}&a_{11}^{33}
\end{vmatrix}\\
&=\frac{(a_{11}^{22})^{2}-(a_{12}^{22})^{2}}{\det{A}}\left(\begin{vmatrix}
~a_{11}^{11}-a_{12}^{11}&\tilde{b}_{1}^{1}~\\\\
~a_{11}^{13}+a_{12}^{13}&\tilde{b}_{1}^{3}
\end{vmatrix}\cdot\begin{vmatrix}
~a_{11}^{11}+a_{12}^{11}&a_{12}^{13}-a_{11}^{13}~\\\\
~a_{12}^{13}-a_{11}^{13}&a_{11}^{33}-a_{12}^{33}
\end{vmatrix}\right),
\end{align*}
and similarly,
\begin{align*}
C_{2}^{3}&=\frac{(a_{11}^{22})^{2}-(a_{12}^{22})^{2}}{\det{A}}\left(\begin{vmatrix}
~a_{11}^{11}-a_{12}^{11}&-\tilde{b}_{1}^{1}~\\\\
~-a_{11}^{13}-a_{12}^{13}&\tilde{b}_{1}^{3}
\end{vmatrix}\cdot\begin{vmatrix}
~a_{11}^{11}+a_{12}^{11}&a_{11}^{13}-a_{12}^{13}~\\\\
~a_{11}^{13}-a_{12}^{13}&a_{11}^{33}-a_{12}^{33}
\end{vmatrix}\right).
\end{align*}
Hence, 
\begin{align}\label{C1 23}
C_{1}^{3}=C_{2}^{3}.
\end{align}
Thus, by using \eqref{matrix}--\eqref{C1 23}, \eqref{C1+C2_1} becomes 
\begin{align}\label{C1+C2_100}
\begin{pmatrix}
~a^{11}&0&0~\\\\
~0&a^{22}&0~\\\\
~0&0&a^{33}
\end{pmatrix}\begin{pmatrix}
~\frac{C_{1}^{1}+C_{2}^{1}}{2}~\\\\
\frac{C_{1}^{2}+C_{2}^{2}}{2}\\\\
\frac{C_{1}^{3}+C_{2}^{3}}{2}
\end{pmatrix}
=\begin{pmatrix}
~0~\\\\
2\tilde{b}_{1}^{2}\\\\
2\tilde{b}_{1}^{3}-(C_{1}^{1}-C_{2}^{1})(a_{11}^{13}+a_{12}^{13})
\end{pmatrix}.
\end{align}

Similar to \eqref{matrix} and \eqref{|symtilb}, by replicating the computation used in \eqref{sym v2}, we get 
$$a_{*}^{12}=a_{*}^{21}=a_{*}^{13}=a_{*}^{31}=a_{*}^{23}=a_{*}^{32}=0$$
and
$$\tilde{b}_{1}^{*1}=-\tilde{b}_{2}^{*1},\quad \tilde{b}_{1}^{*2}=\tilde{b}_{2}^{*2},\quad\tilde{b}_{1}^{*3}=\tilde{b}_{2}^{*3}.$$
Then \eqref{equ_C*alpha} becomes
\begin{align}\label{C1+C2_*}
\begin{pmatrix}
~a_{*}^{11}&0&0~\\\\
~0&a_{*}^{22}&0~\\\\
~0&0&a_{*}^{33}
\end{pmatrix}\begin{pmatrix}
~C_{*}^{1}~\\\\
C_{*}^{2}\\\\
C_{*}^{3}
\end{pmatrix}
=\begin{pmatrix}
~0~\\\\
2\tilde{b}_{1}^{*2}\\\\
2\tilde{b}_{1}^{*3}
\end{pmatrix}.
\end{align}

Finally, we turn to the estimate of $\frac{C_{1}^{\alpha}+C_{2}^{\alpha}}{2}-C_{*}^{\alpha}$, $\alpha=1,2,3$.  
Combining \eqref{C1+C2_100} and \eqref{C1+C2_*}, we have
\begin{align}\label{sym d=2}
\begin{pmatrix}
~a^{11}&0&0~\\\\
~0&a^{22}&0~\\\\
~0&0&a^{33}
\end{pmatrix}\begin{pmatrix}
~\frac{C_{1}^{1}+C_{2}^{1}}{2}-C_{*}^{1}~\\\\
\frac{C_{1}^{2}+C_{2}^{2}}{2}-C_{*}^{2}\\\\
\frac{C_{1}^{3}+C_{2}^{3}}{2}-C_{*}^{3}
\end{pmatrix}
=\begin{pmatrix}
~\mathcal{B}^{1}~\\\\
\mathcal{B}^{2}\\\\
\mathcal{B}^{3}
\end{pmatrix},
\end{align}
where 
\begin{align*}
\mathcal{B}^{1}&=-C_{*}^{1}(a^{11}-a_{*}^{11})=O(\sqrt{\varepsilon}),\\
\mathcal{B}^{2}&=2(\tilde{b}_{1}^{2}-\tilde{b}_{1}^{*2})-C_{*}^{2}(a^{2 2}-a_{*}^{22})=O(\sqrt{\varepsilon}),\\
\mathcal{B}^{3}&=2(\tilde{b}_{1}^{3}-\tilde{b}_{1}^{*3})-C_{*}^{3}(a^{3 3}-a_{*}^{33})-(C_{1}^{1}-C_{2}^{1})(a_{11}^{13}+a_{12}^{13})=O(\sqrt{\varepsilon}),
\end{align*}
here we used Lemmas \ref{es b1 b1* beta=1} and \ref{lem valpha1}. By Cramer's rule, we obtain 
\begin{align*}
\frac{C_{1}^{\alpha}+C_{2}^{\alpha}}{2}-C_{*}^{\alpha}=O(\sqrt{\varepsilon}),\quad\alpha=1,2,3.
\end{align*}

(2) For $d=3$, we can similarly prove by using the symmetry property of domain that 
$$C_{1}^{\alpha}=C_{2}^{\alpha},\quad \alpha=4,5,6.$$
We omit the details here.  Similar to \eqref{sym d=2}, we have
\begin{align*}
\begin{pmatrix}
~a^{11}&0&0&0~\\\\
~\vdots&\vdots&\vdots&\vdots~\\\\
~0&0&0&a^{66}
\end{pmatrix}\begin{pmatrix}
~\frac{C_{1}^{1}+C_{2}^{1}}{2}-C_{*}^{1}~\\\\
\vdots\\\\
\frac{C_{1}^{6}+C_{2}^{6}}{2}-C_{*}^{6}
\end{pmatrix}
=\begin{pmatrix}
~\mathcal{B}^{1}~\\\\
\vdots\\\\
\mathcal{B}^{6}
\end{pmatrix},
\end{align*}
where $\mathcal{B}^{\beta}=O(|\log\varepsilon|^{-1})$ as $\varepsilon\rightarrow0$, $\beta=1,\cdots,6$. Hence, by Cramer's rule, we get \eqref{conv C C*} for $d=3$. Proposition \ref{conv C alpha} is thus proved.
\end{proof}

\subsection{Proof of Proposition \ref{prop converge b}}\label{subsec b}
In this section, we aim to prove Proposition \ref{prop converge b}.  
Recalling \eqref{def u_b}, \eqref{maineqn touch}, and the definitions of $b_{1}^{\beta}[\varphi]$ and $b_{1}^{*\beta}[\varphi]$, \eqref{def_bj},  we have
\begin{align}\label{difference bi beta}
&b_{1}^{\beta}[\varphi]-b_{1}^{*\beta}[\varphi]=\int_{\partial{D}_{1}}\frac{\partial{u}_{b}}{\partial\nu}\Big|_{+}\cdot\psi_{\beta}-\int_{\partial{D}_{1}^{*}}\frac{\partial u^{*}}{\partial \nu}\Big|_{+}\cdot\psi_{\beta}\nonumber\\
&=\int_{\partial D_{1}}\frac{\partial v_{0}}{\partial \nu}\Big|_{+}\cdot\psi_{\beta}-\int_{\partial D_{1}^{*}}\frac{\partial v_{0}^{*}}{\partial \nu}\Big|_{+}\cdot\psi_{\beta}+\sum_{\alpha=1}^{d(d+1)/2}C_{2}^{\alpha}\int_{\partial D_{1}}\frac{\partial v^{\alpha}}{\partial \nu}\Big|_{+}\cdot\psi_{\beta}\nonumber\\
&\quad-\sum_{\alpha=1}^{d(d+1)/2}C_{*}^{\alpha}\int_{\partial D_{1}^{*}}\frac{\partial v^{*\alpha}}{\partial \nu}\Big|_{+}\cdot\psi_{\beta}\nonumber\\
&=\int_{\partial D_{1}}\frac{\partial v_{0}}{\partial \nu}\Big|_{+}\cdot\psi_{\beta}-\int_{\partial D_{1}^{*}}\frac{\partial v_{0}^{*}}{\partial \nu}\Big|_{+}\cdot\psi_{\beta}+\sum_{\alpha=1}^{d(d+1)/2}C_{2}^{\alpha}\Bigg(\int_{\partial D_{1}}\frac{\partial v^{\alpha}}{\partial \nu}\Big|_{+}\cdot\psi_{\beta}\nonumber\\
&\quad-\int_{\partial D_{1}^{*}}\frac{\partial v^{*\alpha}}{\partial \nu}\Big|_{+}\cdot\psi_{\beta}\Bigg)+\sum_{\alpha=1}^{d(d+1)/2}\Big(C_{2}^{\alpha}-C_{*}^{\alpha}\Big)\int_{\partial D_{1}^{*}}\frac{\partial v^{*\alpha}}{\partial \nu}\Big|_{+}\cdot\psi_{\beta},
\end{align}
where $v_{0}$, $v_{0}^{*}$, $v^{\alpha}$, and $v^{*\alpha}$ are, respectively, defined by \eqref{equ_v3}, \eqref{equ_v3*}, \eqref{def valpha}, and \eqref{def valpha*}. From the proof of Proposition \ref{conv C alpha}, $C_{1}^{\alpha}=C_{2}^{\alpha}$ for $\alpha=d+1,\cdots,d(d+1)/2$, and the estimates of $C_{1}^{\alpha}-C_{2}^{\alpha}$ (see for instance, \cite[Proposition 4.2]{bll1} and\cite[Proposition 4.1]{bll2}), we have
\begin{equation}\label{C* Ci}
|C_{*}^{\alpha}-C_{2}^{\alpha}|\leq C\rho_{d}(\varepsilon),\quad \alpha=1,\cdots,d(d+1)/2.
\end{equation}
By using \eqref{difference bi beta}, \eqref{C* Ci}, Lemmas \ref{es b1 b1* beta=1} and \ref{lem valpha1}, we have
\begin{align}\label{estimate b11 b11*}
|b_{1}^\beta[\varphi]-b_{1}^{*\beta}[\varphi]|\leq C\rho_{d}(\varepsilon),
\end{align}
here we used the fact that $|\nabla v^{*\alpha}|\leq C$ in $\Omega^{*}$, see Theorem \ref{auxi thm}. Therefore, Proposition \ref{prop converge b} is proved.

\subsection{Proof of Proposition \ref{prop a11} \Big(The symptotics of $a_{11}^{\alpha\alpha}$\Big)}\label{sec pf propa11}
We will use Theorems \ref{coro v1} and \ref{coro v1d+1}, and Lemma \ref{lem difference v11} to prove Proposition \ref{prop a11}.  Let us first prove the case in dimention two.
\begin{proof}[\bf Proof of Proposition \ref{prop a11} for $d=2$.]
	(1) First consider $a_{11}^{11}$. We  divide it into three parts:
	\begin{align}\label{representation a11}
	a_{11}^{11}&=\int_{\Omega}\big(\mathbb{C}^0e(v_{1}^{1}),e(v_{1}^{1})\big)\ dx\nonumber\\
	&=\int_{\Omega\setminus\Omega_{R}}\big(\mathbb{C}^0e(v_{1}^{1}),e(v_{1}^{1})\big)\ dx+\int_{\Omega_{R}\setminus\Omega_{\varepsilon^{1/8}}}\big(\mathbb{C}^0e(v_{1}^{1}),e(v_{1}^{1})\big)\ dx\nonumber\\
	&\quad+\int_{\Omega_{\varepsilon^{1/8}}}\big(\mathbb{C}^0e(v_{1}^{1}),e(v_{1}^{1})\big)\ dx=:\mbox{I}_{1}+\mbox{I}_{2}+\mbox{I}_{3}.
	\end{align}
In the follow we estimate $\mbox{I}_{1}$, $\mbox{I}_{2}$, and $\mbox{I}_{3}$ one by one.
	
	{\bf Step 1. Claim: }There exists a constant $M_{1}^{*}$, independent of $\varepsilon$, such that
		\begin{equation}\label{est I1}
		\mbox{I}_{1}=M_{1}^{*}+O(\varepsilon^{1/4}).
		\end{equation}

Notice that
	$$\mathcal{L}_{\lambda,\mu}(v_{1}^{1}-v_{1}^{*1})=0,\quad x\in D\setminus D_{1}\cup D_{2}\cup D_{1}^{*}\cup D_{2}^{*}\cup\Omega_{R},$$
	and
	$$0\leq|v_{1}^{1}|, |v_{1}^{*1}|\leq1,\quad x\in D\setminus D_{1}\cup D_{2}\cup D_{1}^{*}\cup D_{2}^{*}\cup\Omega_{R}.$$
	Then since $\partial D_{1}, \partial D_{1}^{*}, \partial D_{2}$, and $\partial D$ are $C^{2,\alpha}$, we have
	\begin{equation}\label{DD v1 1}
	|\nabla^{2}(v_{1}^{1}-v_{1}^{*1})|\leq|\nabla^{2}v_{1}^{1}|+|\nabla^{2}v_{1}^{*1}|\leq C\quad\mbox{in}~D\setminus(D_{1}\cup D_{1}^{*}\cup D_{2}\cup D_{2}^{*}\cup\Omega_{R}).
	\end{equation}
	Moreover, we obtain from \eqref{difference v11} that
	\begin{equation}\label{es v11 v1*1}
	\|v_{1}^{1}-v_{1}^{*1}\|_{L^{\infty}(D\setminus(D_{1}\cup D_{1}^{*}\cup D_{2}\cup D_{2}^{*}\cup\Omega_{\varepsilon^{1/4}}))}\leq C\varepsilon^{1/2}.
	\end{equation}
By using the interpolation inequality, \eqref{DD v1 1}, and \eqref{es v11 v1*1}, we obtain
	\begin{equation}\label{D v1 *1}
	|\nabla(v_{1}^{1}-v_{1}^{*1})|\leq C\varepsilon^{1/2(1-\frac{1}{2})}=C\varepsilon^{1/4}\quad\mbox{in}~D\setminus(D_{1}\cup D_{1}^{*}\cup D_{2}\cup D_{2}^{*}\cup\Omega_{R}).
	\end{equation}
	
	Denote
	$$M_{1}^{*}:=\int_{\Omega^{*}\setminus\Omega_{R}^{*}}\big(\mathbb{C}^0e(v_{1}^{*1}),e(v_{1}^{*1})\big)\ dx.$$
	Then
		\begin{align*}
		\mbox{I}_{1}-M_{1}^{*}
		&=\int_{\Omega\setminus(D_{1}^{*}\cup\Omega_{R})}\Big(\big(\mathbb{C}^0e(v_{1}^{1}),e(v_{1}^{1})\big)-\big(\mathbb{C}^0e(v_{1}^{*1}),e(v_{1}^{*1})\big)\Big)\ dx\\
		&\quad+\int_{D_{1}^{*}\setminus(D_{1}\cup\Omega_{R})}\big(\mathbb{C}^0e(v_{1}^{1}),e(v_{1}^{1})\big)\ dx-\int_{D_{1}\setminus D_{1}^{*}}\big(\mathbb{C}^0e(v_{1}^{*1}),e(v_{1}^{*1})\big)\ dx.
		\end{align*}
It follows from $|D_{1}^{*}\setminus(D_{1}\cup\Omega_{R})|\leq C\varepsilon$, $|D_{1}\setminus D_{1}^{*}|\leq C\varepsilon$, and the boundedness of $|\nabla v_{1}^{1}|$ and $|\nabla v_{1}^{*1}|$ in $D_{1}^{*}\setminus(D_{1}\cup\Omega_{R})$ and  $D_{1}\setminus D_{1}^{*}$, respectively, that
\begin{align}\label{est bdd out}
\left|\int_{D_{1}^{*}\setminus(D_{1}\cup\Omega_{R})}\big(\mathbb{C}^0e(v_{1}^{1}),e(v_{1}^{1})\big)\ dx-\int_{D_{1}\setminus D_{1}^{*}}\big(\mathbb{C}^0e(v_{1}^{*1}),e(v_{1}^{*1})\big)\ dx\right|\leq C\varepsilon.
\end{align}
So by using \eqref{D v1 *1} and \eqref{est bdd out}, we have
	\begin{align*}
	\mbox{I}_{1}-M_{1}^{*}
	&=\int_{\Omega\setminus(D_{1}^{*}\cup\Omega_{R})}\big(\mathbb{C}^0e(v_{1}^{1}-v_{1}^{*1}),e(v_{1}^{1}-v_{1}^{*1})\big)\ dx\\
	&\quad+2\int_{\Omega\setminus(D_{1}^{*}\cup\Omega_{R})}\big(\mathbb{C}^0e(v_{1}^{*1}),e(v_{1}^{1}-v_{1}^{*1})\big)\ dx+O(\varepsilon)\\
	&=O(\varepsilon^{1/4}).
	\end{align*}
We henceforth get \eqref{est I1}.
	
{\bf Step 2.} Proof of 
\begin{equation}\label{est I2}
\mbox{I}_{2}=\mbox{I}_{2}^{*}+O(\varepsilon^{1/8}),
\end{equation}
where
$$\mbox{I}_{2}^{*}=\int_{\Omega_{R}^{*}\setminus\Omega_{\varepsilon^{1/8}}^{*}}\big(\mathbb{C}^0e(v_{1}^{*1}),e(v_{1}^{*1})\big)\ dx.$$

We further divide $\mbox{I}_{2}-\mbox{I}_{2}^{*}$ into three terms:	
	\begin{align}\label{divide I2}
	&\mbox{I}_{2}-\mbox{I}_{2}^{*}\nonumber\\
	&=\int_{(\Omega_{R}\setminus\Omega_{\varepsilon^{1/8}})\setminus(\Omega_{R}^{*}\setminus\Omega_{\varepsilon^{1/8}}^{*})}\big(\mathbb{C}^0e(v_{1}^{1}),e(v_{1}^{1})\big)\ dx+2\int_{\Omega_{R}^{*}\setminus\Omega_{\varepsilon^{1/8}}^{*}}\big(\mathbb{C}^0e(v_{1}^{*1}),e(v_{1}^{1}-v_{1}^{*1})\big)\ dx\nonumber\\
	&\quad+\int_{\Omega_{R}^{*}\setminus\Omega_{\varepsilon^{1/8}}^{*}}\big(\mathbb{C}^0e(v_{1}^{1}-v_{1}^{*1}),e(v_{1}^{1}-v_{1}^{*1})\big)\ dx\nonumber\\
	&=:\mbox{I}_{2,1}+\mbox{I}_{2,2}+\mbox{I}_{2,3}.
	\end{align}
	
	For $\varepsilon^{1/8}\leq|z_{1}|\leq R$, we rescale $\Omega_{|z_{1}|+|z_{1}|^{2}}\setminus\Omega_{|z_{1}|}$ into a nearly cube $Q_{1}$ in unit size, and $\Omega_{|z_{1}|+|z_{1}|^{2}}^{*}\setminus\Omega_{|z_{1}|}^{*}$ into $Q_{1}^{*}$ by using the following change of variables:
	\begin{align*}
	\begin{cases}
	x_{1}-z_{1}=|z_{1}|^{2}y_{1},\\
	x_{2}=|z_{1}|^{2}y_{2}.
	\end{cases}
	\end{align*}
	After rescaling, let
	$$V_{1}^{1}=v_{1}^{1}(z_{1}+z_{1}^{2}y_{1},|z_{1}|^{2}y_{2}),\quad\mbox{in}~Q_{1},\quad\mbox{and}\quad V_{1}^{*1}=v_{1}^{*1}(z_{1}+z_{1}^{2}y_{1},|z_{1}|^{2}y_{2}),\quad\mbox{in}~Q_{1}^{*}.$$
	By the same reason that leads to \eqref{DD v1 1}, we get
	$$|\nabla^{2}V_{1}^{1}|\leq C\quad\mbox{in}~Q_{1},\quad\mbox{and}~|\nabla^{2}V_{1}^{*1}|\leq C\quad\mbox{in}~Q_{1}^{*}.$$
	Using the interpolation inequality, we obtain
	$$|\nabla(V_{1}^{1}-V_{1}^{*1})|\leq C\varepsilon^{1/4}.$$
	
	Rescaling back to $v_{1}^{1}-v_{1}^{*1}$, we get
	\begin{equation}\label{difference Dv1 1}
	|\nabla(v_{1}^{1}-v_{1}^{*1})|\leq C\varepsilon^{1/4}|x_{1}|^{-2}\quad\mbox{in}~\Omega_{R}^{*}\setminus\Omega_{\varepsilon^{1/8}}^{*}.
	\end{equation}
	Similarly, we have
	\begin{equation}\label{Dv11}
	|\nabla v_{1}^{1}|\leq C|x_{1}|^{-2}\quad\mbox{in}~\Omega_{R}\setminus\Omega_{\varepsilon^{1/8}},
	\end{equation}
	and
	\begin{equation}\label{Dv11*}
	|\nabla v_{1}^{*1}|\leq C|x_{1}|^{-2}\quad\mbox{in}~\Omega_{R}^{*}\setminus\Omega_{\varepsilon^{1/8}}^{*}.
	\end{equation}
	Now, by \eqref{Dv11} and $|(\Omega_{R}\setminus\Omega_{\varepsilon^{1/8}})\setminus(\Omega_{R}^{*}\setminus\Omega_{\varepsilon^{1/8}}^{*})|\leq C\varepsilon$, we have
	\begin{align*}
	|\mbox{I}_{2,1}|\leq C\varepsilon\int_{\varepsilon^{1/8}<|x_{1}|\leq R}\frac{dx_{1}}{|x_{1}|^{4}}\leq C\varepsilon^{5/8}.
	\end{align*}
	Also, by using \eqref{difference Dv1 1} and \eqref{Dv11*}, we obtain
		\begin{align*}
	|\mbox{I}_{2,2}|\leq C\varepsilon^{1/4}\int_{\varepsilon^{1/8}<|x_{1}|\leq R}\frac{dx_{1}}{|x_{1}|^{2}}\leq C\varepsilon^{1/8},
		\end{align*}
	and by \eqref{difference Dv1 1}, we get
	\begin{align*}
	|\mbox{I}_{2,3}|\leq C\varepsilon^{1/2}\int_{\varepsilon^{1/8}<|x_{1}|\leq R}\frac{dx_{1}}{|x_{1}|^{2}}\leq C\varepsilon^{3/8}.
	\end{align*}
	Substituting the estimates above into \eqref{divide I2}, we obtain \eqref{est I2}.
	
	{\bf Step 3.} We next further approximate $\mbox{I}_{2}^{*}$ by some specific functions. Note that
	\begin{align*}
	\mbox{I}_{2}^{*}&=\int_{\Omega_{R}^{*}\setminus\Omega_{\varepsilon^{1/8}}^{*}}\big(\mathbb{C}^0e(u_{1}^{*1}),e(u_{1}^{*1})\big)\ dx+2\int_{\Omega_{R}^{*}\setminus\Omega_{\varepsilon^{1/8}}^{*}}\big(\mathbb{C}^0e(u_{1}^{*1}),e(v_{1}^{*1}-u_{1}^{*1})\big)\ dx\nonumber\\ &\quad+\int_{\Omega_{R}^{*}\setminus\Omega_{\varepsilon^{1/8}}^{*}}\big(\mathbb{C}^0e(v_{1}^{*1}-u_{1}^{*1}),e(v_{1}^{*1}-u_{1}^{*1})\big)\ dx.
	\end{align*}
	By using Theorem \ref{coro v1}, for the second term, we have
	\begin{align*}
	2\left|\int_{\Omega_{\varepsilon^{1/8}}^{*}}\big(\mathbb{C}^0e(u_{1}^{*1}),e(v_{1}^{*1}-u_{1}^{*1})\big)\ dx\right|\leq C\varepsilon^{1/8},
	\end{align*}
	and for the third term,
	\begin{align*}
	\left|\int_{\Omega_{\varepsilon^{1/8}}^{*}}\big(\mathbb{C}^0e(v_{1}^{*1}-u_{1}^{*1}),e(v_{1}^{*1}-u_{1}^{*1})\big)\ dx\right|\leq C\varepsilon^{3/8}.
	\end{align*}	
	Hence,
	\begin{align}\label{est I2*}
	\mbox{I}_{2}^{*}=\int_{\Omega_{R}^{*}\setminus\Omega_{\varepsilon^{1/8}}^{*}}\big(\mathbb{C}^0e(u_{1}^{*1}),e(u_{1}^{*1})\big)\ dx+M_{2}^{*}+O(\varepsilon^{1/8}),
	\end{align}
	where 
	$$M_{2}^{*}:=2\int_{\Omega_{R}^{*}}\big(\mathbb{C}^0e(u_{1}^{*1}),e(v_{1}^{*1}-u_{1}^{*1})\big)\ dx+\int_{\Omega_{R}^{*}}\big(\mathbb{C}^0e(v_{1}^{*1}-u_{1}^{*1}),e(v_{1}^{*1}-u_{1}^{*1})\big)\ dx$$
	is a constant independent of $\varepsilon$. Coming back to \eqref{representation a11}, and using \eqref{est I1}, \eqref{est I2}, and  \eqref{est I2*}, so far we obtain
	\begin{align}\label{equality a11}
	a_{11}^{11}&=\mbox{I}_{3}+\int_{\Omega_{R}^{*}\setminus\Omega_{\varepsilon^{1/8}}^{*}}\big(\mathbb{C}^0e(u_{1}^{*1}),e(u_{1}^{*1})\big)\ dx+M_{1}^{*}+M_{2}^{*}+O(\varepsilon^{1/8}).
	\end{align}

{\bf Step 4.} Now we are in a position to complete the rest of the proof by direct computations.
First, similar to \eqref{est I2*}, we obtain 
\begin{align}\label{equality II}
\mbox{I}_{3}&=\int_{\Omega_{\varepsilon^{1/8}}}\big(\mathbb{C}^0e(v_{1}^{1}),e(v_{1}^{1})\big)\ dx\nonumber\\
&=\int_{\Omega_{\varepsilon^{1/8}}}\big(\mathbb{C}^0e(u_{1}^{1}),e(u_{1}^{1})\big)\ dx+2\int_{\Omega_{\varepsilon^{1/8}}}\big(\mathbb{C}^0e(u_{1}^{1}),e(v_{1}^{1}-u_{1}^{1})\big)\ dx\nonumber\\
&\quad+\int_{\Omega_{\varepsilon^{1/8}}}\big(\mathbb{C}^0e(v_{1}^{1}-u_{1}^{1}),e(v_{1}^{1}-u_{1}^{1})\big)\ dx\nonumber\\
&=\int_{\Omega_{\varepsilon^{1/8}}}\big(\mathbb{C}^0e(u_{1}^{1}),e(u_{1}^{1})\big)\ dx+O(\varepsilon^{1/8}).
\end{align}
Second,  by a direct computation, we obtain in $\Omega_{R}$,
\begin{align}\label{cal auxiliary}
|\partial_{x_{1}}(\bar{u}_{1}^{1})^{1}|,~|\partial_{x_{2}}(\tilde{u}_{1}^{1})^{2}|\leq\frac{C|x_{1}|}{\delta(x_{1})},\quad |\partial_{x_{1}}(\tilde{u}_{1}^{1})^{2}|\leq C,~\partial_{x_{2}}(\bar{u}_{1}^{1})^{1}=\frac{1}{\delta(x_{1})}.
\end{align}
For any $u=(u^{1},u^{2})^{\mathrm{T}}$, recalling the definition of $\mathbb{C}^0$, a direct calculation yields
	\begin{align}\label{formula Eu}		
	&\big(\mathbb{C}^0e(u),e(u)\big)\nonumber\\
	&=\lambda(\partial_{x_{1}}u^1+\partial_{x_{2}}u^2)^2+\mu\big(2(\partial_{x_{1}}u^1)^2+(\partial_{x_{2}}u^1+\partial_{x_{1}}u^2)^2+2(\partial_{x_{2}}u^2)^2\big).
	\end{align}
Substituting $u_{1}^{1}$ and $u_{1}^{*1}$ into \eqref{formula Eu} and using \eqref{cal auxiliary}, we have
	\begin{align*}
	&\int_{\Omega_{\varepsilon^{1/8}}}\big(\mathbb{C}^0e(u_{1}^{1}),e(u_{1}^{1})\big)\ dx+\int_{\Omega_{R}^{*}\setminus\Omega_{\varepsilon^{1/8}}^{*}}\big(\mathbb{C}^0e(u_{1}^{*1}),e(u_{1}^{*1})\big)\ dx\nonumber\\
	&=\mu\int_{|x_{1}|\leq \varepsilon^{1/8}}\frac{dx_{1}}{\varepsilon+\kappa|x_{1}|^{2}+o(|x_1|^{2})}+\mu\int_{\varepsilon^{1/8}<|x_{1}|\leq R}\frac{dx_{1}}{\kappa|x_{1}|^{2}+o(|x_1|^{2})}+O(\varepsilon^{1/8}).
	\end{align*}
	Notice that for the first term,
	\begin{align*}
	\int_{|x_{1}|\leq \varepsilon^{1/8}}\frac{dx_{1}}{\varepsilon+\kappa|x_{1}|^{2}+o(|x_1|^{2})}=\int_{|x_{1}|\leq \varepsilon^{1/8}}\frac{dx_{1}}{\varepsilon+\kappa|x_{1}|^{2}}+O(\varepsilon^{\frac{\gamma-1}{8}}),
	\end{align*}
	and for the second term,
	\begin{align*}
	\int_{\varepsilon^{1/8}<|x_{1}|\leq R}\frac{dx_{1}}{\kappa|x_{1}|^{2}+o(|x_1|^{2})}=\int_{\varepsilon^{1/8}<|x_{1}|\leq R}\frac{dx_{1}}{\kappa|x_{1}|^{2}}+O(\varepsilon^{\frac{\gamma-1}{8}}).
	\end{align*}
Let two right hand sides subtract,
	\begin{align*}
	\int_{\varepsilon^{1/8}<|x_{1}|\leq R}\left(\frac{1}{\kappa|x_{1}|^{2}}-\frac{1}{\varepsilon+\kappa|x_{1}|^{2}}\right)\ dx_1=O(\varepsilon^{5/8}).
	\end{align*}
	Then we obtain 
	\begin{align}\label{formula e(u1)}
	&\int_{\Omega_{\varepsilon^{1/8}}}\big(\mathbb{C}^0e(u_{1}^{1}),e(u_{1}^{1})\big)\ dx+\int_{\Omega_{R}^{*}\setminus\Omega_{\varepsilon^{1/8}}^{*}}\big(\mathbb{C}^0e(u_{1}^{*1}),e(u_{1}^{*1})\big)\ dx\nonumber\\
	&=\frac{\pi\mu}{\sqrt{\kappa}\sqrt{\varepsilon}}+O(|\log\varepsilon|).
	\end{align}
	
	Substituting \eqref{formula e(u1)} into \eqref{equality a11} and using \eqref{equality II}, we have
	\begin{equation}\label{a11 asym}
	a_{11}^{11}=\frac{\pi\mu}{\sqrt{\kappa}\sqrt{\varepsilon}}+O(|\log\varepsilon|)=\frac{\pi\mu}{\sqrt{\kappa}\sqrt{\varepsilon}}\left(1+O(\varepsilon^{\frac{\gamma+3}{8}})\right).
	\end{equation}
	
	(2) For $a_{11}^{22}$, we apply the auxiliary function defined in \eqref{auxiliary improved}. The rest of the proof is the same as in (1), except that we substitute $u_{1}^{2}$ and $u_{1}^{*2}$ into \eqref{formula Eu}. A direct calculation gives
	\begin{align}\label{cal auxiliary22}
	|\partial_{x_{1}}(\bar{u}_{1}^{2})^{2}|,~|\partial_{x_{2}}(\tilde{u}_{1}^{2})^{1}|\leq\frac{C|x_{1}|}{\delta(x_{1})},\quad |\partial_{x_{1}}(\tilde{u}_{1}^{2})^{1}|\leq C,~\partial_{x_{2}}(\bar{u}_{1}^{2})^{2}=\frac{1}{\delta(x_{1})}.
	\end{align}
	By using \eqref{cal auxiliary22} and the same argument as that in \eqref{formula e(u1)}, we obtain
	\begin{equation*}
	a_{11}^{22}=\frac{\pi(\lambda+2\mu)}{\sqrt{\kappa}\sqrt{\varepsilon}}\left(1+O(\varepsilon^{\frac{\gamma+3}{8}})\right).
	\end{equation*}
	So we complete the proof of Proposition  \ref{prop a11} in dimension two.
\end{proof}

\begin{proof}[\bf Proof of Proposition \ref{prop a11} for $d=3$.]
The proof is similar to the above until \eqref{formula Eu}. For $u=(u^{1},u^{2},u^{3})^{\mathrm{T}}$, \eqref{formula Eu} becomes 
\begin{align}\label{formula Eu dim3}	
		&\big(\mathbb{C}^0e(u),e(u)\big)\nonumber\\
		&=\lambda\Big(\partial_{x_{1}}u^1+\partial_{x_{2}}u^2+\partial_{x_{3}}u^3\Big)^2+\mu\Big(2(\partial_{x_{1}}u^1)^2+2(\partial_{x_{2}}u^2)^2+2(\partial_{x_{3}}u^3)^2\nonumber\\
		&\quad+(\partial_{x_{2}}u^1+\partial_{x_{1}}u^2)^2+(\partial_{x_{3}}u^1+\partial_{x_{1}}u^3)^2+(\partial_{x_{2}}u^3+\partial_{x_{3}}u^2)^2\Big).
		\end{align}
Substituting the specific function $u_{1}^{1}$, defined by  \eqref{auxiliary improved dim3}, into \eqref{formula Eu dim3} and recalling that in $\Omega_{R}$, 
		\begin{align*}
		|\partial_{x_{k}}u_{1}^{1}|\leq\frac{C|x'|}{\delta(x')}\quad\mbox{and}\quad \partial_{x_{3}}u_{1}^{1}=\frac{1}{\delta(x')},\quad k=1,2,
		\end{align*}
we find that \eqref{a11 asym} becomes 
\begin{align*}
a_{11}^{11}
&=\mu\int_{|x'|\leq \varepsilon^{1/8}}\frac{dx'}{\varepsilon+\kappa|x'|^{2}+o(|x'|^{2})}+\mu\int_{\varepsilon^{1/8}<|x'|\leq R}\frac{dx'}{\kappa|x'|^{2}+o(|x'|^{2})}+O(\varepsilon^{1/8})\nonumber\\
&=\mu\int_{|x'|\leq R}\frac{dx'}{\varepsilon+\kappa|x'|^{2}}+O(\varepsilon^{1/8})\nonumber\\
&=\frac{\pi\mu}{\kappa}|\log\varepsilon|+\mathcal{C}_{3}^{*1}+O(\varepsilon^{1/8}),
\end{align*}
where $\mathcal{C}_{3}^{*1}$ is a constant independent of $\varepsilon$ and $R$. Similarly,
\begin{align*}
a_{11}^{22}=\frac{\pi\mu}{\kappa}|\log\varepsilon|+\mathcal{C}_{3}^{*2}+O(\varepsilon^{1/8}),\quad\mbox{and}\quad a_{11}^{33}=\frac{\pi(\lambda+2\mu)}{\kappa}|\log\varepsilon|+\mathcal{C}_{3}^{*3}+O(\varepsilon^{1/8}),
\end{align*}
where $\mathcal{C}_{3}^{*2}$ and $\mathcal{C}_{3}^{*3}$ are constants independent of $\varepsilon$ and $R$. Hence, Proposition \ref{prop a11} in dimension three is proved.
\end{proof}

\section{The proof of Theorem \ref{thm1} and Theorem \ref{thm2}}\label{sec pf thm1}
In this section, we prove our main results, Theorem \ref{thm1} and Theorem \ref{thm2}. Because the estimates of $\nabla v_{i}^{\alpha}$ have been proved in Theorems \ref{auxi thm}, \ref{coro v1}, and \ref{coro v1d+1}, recalling the proof of Proposition \ref{conv C alpha},  we have $C_{1}^{\alpha}=C_{2}^{\alpha}$, $\alpha=d+1,\cdots,d(d+1)/2$, and \eqref{nablau_dec} becomes
\begin{equation}\label{nablau_dec_d}
\nabla{u}=\sum_{\alpha=1}^{d}\left(C_{1}^{\alpha}-C_{2}^{\alpha}\right)\nabla{v}_{1}^{\alpha}
+\nabla u_{b},\quad\mbox{in}~\Omega.
\end{equation}	
Thus, it remains to derive the asymptotics of $C_{1}^{\alpha}-C_{2}^{\alpha}$, $\alpha=1,\cdots,d$. To achieve this, 
we use \eqref{nablau_dec_d} and the forth line of \eqref{maineqn} to obtain
\begin{align}\label{C1C2_2}
\sum_{\alpha=1}^{d}(C_{1}^{\alpha}-C_{2}^{\alpha})a_{11}^{\alpha\beta}&=b_{1}^{\beta},\quad~~\beta=1,\cdots,d(d+1)/2,
\end{align}
where $a_{11}^{\alpha\beta}$ is defined in \eqref{def_aij}, and $b_{1}^{\beta}$ is defined in \eqref{def_bj}.   Let us first consider $d=2$.

\subsection{{\bf Proof of Theorem \ref{thm1}}}
\begin{proof}
We denote \eqref{C1C2_2} in block matrix as follows:
\begin{align}\label{block matrix}
\begin{pmatrix}
~a_{11}^{11}&0~\\\\
~0&a_{11}^{22}
\end{pmatrix}\begin{pmatrix}
~C_{1}^{1}-C_{2}^{1}~\\\\
C_{1}^{2}-C_{2}^{2}
\end{pmatrix}
=\begin{pmatrix}
~b_{1}^{1}~\\\\
b_{1}^{2}
\end{pmatrix},
\end{align}
here we used $a_{11}^{12}=a_{11}^{21}=0$ in \eqref{matrix}. 

By using Cramer's rule, Proposition \ref{prop a11}, and Proposition \ref{prop converge b}, we have
\begin{align}\label{differece C1 C2}
C_{1}^{1}-C_{2}^{1}
=\frac{b_{1}^{1}}{a_{11}^{11}}=\frac{\sqrt{\kappa}}{\pi\mu}\cdot\sqrt{\varepsilon}\left(b_{1}^{*1}[\varphi]+O(\sqrt{\varepsilon})\right)\left(1+O(\varepsilon^{\frac{\gamma+3}{8}})\right).
\end{align}
Similarly,
\begin{align}\label{differece C1 C2000}
C_{1}^{2}-C_{2}^{2}
=\frac{\sqrt{\kappa}}{\pi(\lambda+2\mu)}\cdot\sqrt{\varepsilon}\left(b_{1}^{*2}[\varphi]+O(\sqrt{\varepsilon})\right)\left(1+O(\varepsilon^{\frac{\gamma+3}{8}})\right).
\end{align}
Then for any $x\in\Omega_{R}$, by using \eqref{nablau_dec_d},  Theorems \ref{auxi thm} and \ref{coro v1},  \eqref{differece C1 C2}, and \eqref{differece C1 C2000}, we obtain
\begin{align*}
\nabla u&=\sum_{\alpha=1}^{2}\left(C_{1}^{\alpha}-C_{2}^{\alpha}\right)\nabla u_{1}^{\alpha}+O(1)\\
&=\frac{\sqrt{\kappa\varepsilon}}{\pi\mu}b_{1}^{*1}[\varphi]\nabla u_{1}^{1}\left(1+O(\varepsilon^{\frac{\gamma+3}{8}})\right)+\frac{\sqrt{\kappa\varepsilon}}{\pi(\lambda+2\mu)}b_{1}^{*2}[\varphi]\nabla u_{1}^{2}\left(1+O(\varepsilon^{\frac{\gamma+3}{8}})\right)\\
&\quad+O(1)\\
&=\frac{\sqrt{\kappa}}{\pi}\cdot\sqrt{\varepsilon}\left(\frac{b_{1}^{*1}[\varphi]}{\mu}\nabla u_{1}^{1}+\frac{b_{1}^{*2}[\varphi]}{\lambda+2\mu}\nabla u_{1}^{2}\right)(1+O(\varepsilon^{\frac{\gamma+3}{8}}))+O(1).
\end{align*}
The proof of Theorem \ref{thm1} is finished.
\end{proof}

\subsection{{\bf Proof of Theorem \ref{thm2}}}
\begin{proof}
In dimension three, in order to solve $C_{1}^{\alpha}-C_{2}^{\alpha}$ from \eqref{C1C2_2}, we rewrite it as
\begin{align}\label{block matrix d=3}
\begin{pmatrix}
~a_{11}^{11}&0&0~\\\\
~0&a_{11}^{22}&0\\\\
~0&0&a_{11}^{33}
\end{pmatrix}\begin{pmatrix}
~C_{1}^{1}-C_{2}^{1}~\\\\
C_{1}^{2}-C_{2}^{2}~\\\\
C_{1}^{3}-C_{2}^{3}
\end{pmatrix}
=\begin{pmatrix}
~b_{1}^{1}~\\\\
b_{1}^{2}~\\\\
b_{1}^{3}
\end{pmatrix},
\end{align}
here we also used the fact that $a_{11}^{\alpha\beta}=0$ by the symmetry of domains, $\alpha,\beta=1,2,3$, $\alpha\neq\beta$. Then similar to \eqref{differece C1 C2} and \eqref{differece C1 C2000}, we have
\begin{align}\label{difference C d=3}
C_{1}^{\alpha}-C_{2}^{\alpha}&=\frac{\kappa}{\pi\mu}\cdot\frac{1}{|\log\varepsilon|}b_{1}^{*\alpha}[\varphi]\left(1+O(|\log\varepsilon|^{-1})\right),\quad \alpha=1,2,\nonumber\\
C_{1}^{3}-C_{2}^{3}&=\frac{\kappa}{\pi(\lambda+2\mu)}\cdot\frac{1}{|\log\varepsilon|}b_{1}^{*3}[\varphi]\left(1+O(|\log\varepsilon|^{-1})\right).
\end{align}
For any $x\in\Omega_{R}$, by using \eqref{nablau_dec_d}, Theorems \ref{coro v1} and \ref{coro v1d+1}, and \eqref{difference C d=3}, we have
\begin{align*}
\nabla u&=
\sum_{\alpha=1}^{3}(C_{1}^{\alpha}-C_{2}^{\alpha})\nabla u_{1}^{\alpha}+O(1)\\
&=
\frac{\kappa}{\pi}\cdot\frac{1}{|\log\varepsilon|}
\left(\frac{1}{\mu}\sum_{\alpha=1}^{2}b_{1}^{*\alpha}[\varphi]\nabla u_{1}^{\alpha}+\frac{1}{\lambda+2\mu}b_{1}^{*3}[\varphi]\nabla u_{1}^{3}\right)\left(1+O(|\log\varepsilon|^{-1})\right)+O(1).
\end{align*}
The proof of Theorem \ref{thm2} is completed.
\end{proof}

\section{The proof of Theorem \ref{thmhigher} and Theorem \ref{thmhigher2}}\label{proof general}
\subsection{The proof of Theorem \ref{thmhigher}} \label{subsec prf thm}
By using a similar argument that led to Propositions \ref{prop a11} and \ref{prop converge b}, we obtain

\begin{prop}\label{prop a1133}
Under the assumptions \eqref{convexity} and \eqref{h1h3} with $m\geq3$, we have, for sufficiently small $\varepsilon>0$,
	\begin{equation*}
	a_{11}^{11}=
	\frac{\mu Q_{2,m}}{\kappa^{1/m}\varepsilon^{1-1/m}}+O(1),\quad 	a_{11}^{22}=
	\frac{(\lambda+2\mu) Q_{2,m}}{\kappa^{1/m}\varepsilon^{1-1/m}}+O(1),\quad m\geq3,
	\end{equation*}
	and
	\begin{equation*}
	a_{11}^{33}=
	\begin{cases}
	\frac{2(\lambda+2\mu)}{3\kappa}|\log\varepsilon|+O(1),\quad m=3,\\
	\frac{(\lambda+2\mu) \widetilde{Q}_{2,m}}{\kappa^{3/m}\varepsilon^{1-3/m}}+
	O(1),\quad m\geq4,
	\end{cases}
	\end{equation*}
	where
	$$Q_{2,m}=2\int_{0}^{\infty}\frac{1}{1+t^{m}}\ dt,\quad \widetilde{Q}_{2,m}=2\int_{0}^{\infty}\frac{t^{2}}{1+t^{m}}\ dt.$$
\end{prop}

\begin{prop}\label{prop converge B}
	Under the above assumptions, we have for small enough $\varepsilon>0$, if $\alpha=1,2$,
	\begin{align*}
b_1^{\alpha}[\varphi]-b_{1}^{*\alpha}[\varphi]
	=
	\begin{cases}
	O(|\log\varepsilon|^{-1}),~m=3,\\
	O(\varepsilon^{1/4}),~ m=4,\\
	O(\varepsilon^{1/3}),~ m\geq5,
	\end{cases}
	b_1^{3}[\varphi]-b_1^{*3}[\varphi]
	=
	\begin{cases}
	O(|\log\varepsilon|^{-1}),~m=3,\\
	O(\varepsilon^{1-\frac{3}{m}}),~m=4,5,\\
	O(\varepsilon^{\frac{m+2}{3m}}),~ m\geq6,
	\end{cases}
	\end{align*}
	where $b_1^{\alpha}[\varphi]$and $b_1^{*\alpha}$ are defined in \eqref{def_bj},  $\alpha=1,2,3$.
\end{prop}

We are ready to finish the proof of Theorem \ref{thmhigher}.
\begin{proof}[\bf Completion of the proof of Theorem \ref{thmhigher}.]
For any $x\in\Omega_{R}$, by using  \eqref{nablau_dec},  Theorems \ref{auxi thm} and \ref{coro v1}, we have
	\begin{align}\label{nablau R2}
	\nabla u=\sum_{\alpha=1}^{3}\left(C_{1}^{\alpha}-C_{2}^{\alpha}\right)
	\nabla u_{1}^{\alpha}+O(1),\quad m\geq3.
	\end{align}
Recalling \eqref{block matrix}, and using Cramer's rule, Propositions  \ref{prop a1133} and \ref{prop converge B}, if $m=3$, we have
\begin{align*}	C_{1}^1-C_{2}^1&=\frac{b_{1}^{1}[\varphi]}{a_{11}^{11}}=\frac{\kappa^{1/3}\varepsilon^{2/3}}{\mu Q_{2,3}}b_{1}^{*1}[\varphi]\left(1+O(|\log\varepsilon|^{-1})\right),\\
C_{1}^2-C_{2}^2
	&=\frac{b_{1}^{2}[\varphi]}{a_{11}^{22}}=\frac{\kappa^{1/3}\varepsilon^{2/3}}{(\lambda+2\mu) Q_{2,3}}b_{1}^{*2}[\varphi]\left(1+O(|\log\varepsilon|^{-1})\right)\\
	C_{1}^3-C_{2}^3
	&=\frac{3\kappa}{2(\lambda+2\mu)|\log\varepsilon|}b_{1}^{*3}[\varphi]\left(1+O(|\log\varepsilon|^{-1})\right).
	\end{align*}
	Hence, coming back to \eqref{nablau R2}, for $x\in\Omega_{R}$,
	\begin{align*}
	\nabla u
	&=\frac{\kappa^{1/3}\varepsilon^{2/3}}{ Q_{2,3}}\left(\frac{b_{1}^{*1}[\varphi]}{\mu}\nabla u_{1}^{1}+\frac{b_{1}^{*2}[\varphi]}{\lambda+2\mu}\nabla u_{1}^{2}\right)\left(1+O(|\log\varepsilon|^{-1})\right)\\
	&\quad+\frac{3\kappa}{2(\lambda+2\mu)|\log\varepsilon|}b_{1}^{*3}[\varphi]\nabla u_{1}^{3}\left(1+O(|\log\varepsilon|^{-1})\right)+O(1).
	\end{align*}
The cases $m\geq4$ follow from the same argument as above, we omit them here. This completes the proof of Theorem \ref{thmhigher}.
\end{proof}

\subsection{The proof of Theorem \ref{thmhigher2}}

(i) The proof of the case $m=3$ is very similar to that in the proof of Theorem \ref{thm1} when $d=3$ and $m=2$. We omit it here.

(ii)  For $m\geq4$, we have
\begin{prop}\label{prop a1133 R3}
	Under the assumptions \eqref{convexity} and \eqref{h1h3} with $m\geq4$, we have,  for sufficiently small $\varepsilon>0$ and $m\geq4$,
	\begin{equation*}
	a_{11}^{\alpha\alpha}=
	\frac{\pi\mu Q_{3,m}}{\kappa^{2/m}\varepsilon^{1-2/m}}+O(1),~\alpha=1,2,\quad a_{11}^{33}=
	\frac{\pi(\lambda+2\mu) Q_{3,m}}{\kappa^{2/m}\varepsilon^{1-2/m}}+O(1),
	\end{equation*}
	\begin{equation*}
	a_{11}^{44}=
	\begin{cases}
	\frac{\pi\mu}{2\kappa}|\log\varepsilon|+
	O(1),~\quad m=4,\\
	\frac{\pi\mu \widetilde{Q}_{3,m}}{\kappa^{4/m}\varepsilon^{1-4/m}}+
	O(1),~\quad m\geq5,
	\end{cases}
	\end{equation*}
	and for $\alpha=5,6$,
	\begin{equation*}
	a_{11}^{\alpha\alpha}=
	\begin{cases}
	\frac{\pi(\lambda+2\mu)}{4\kappa}|\log\varepsilon|+O(1),~\quad m=4,\\
	\frac{\pi(\lambda+2\mu) \widetilde{Q}_{3,m}}{2\kappa^{4/m}\varepsilon^{1-4/m}}+
	O(1),~\quad m\geq5,
	\end{cases}
	\end{equation*}
	where
	$$Q_{3,m}=2\int_{0}^{\infty}\frac{t}{1+t^{m}}\ dt,\quad \widetilde{Q}_{3,m}=2\int_{0}^{\infty}\frac{t^{3}}{1+t^{m}}\ dt.$$
\end{prop}

\begin{prop}\label{prop converge B R3}
	Under the above assumptions, we have for sufficiently small $\varepsilon>0$, if $\alpha=1,2,3$,
	\begin{align*}
	b_{1}^{\alpha}[\varphi]-b_{1}^{*\alpha}[\varphi]
	&=
	\begin{cases}
	O(|\log\varepsilon|^{-1}),&\quad m=4,\\
	O(\varepsilon^{1/5}),&\quad m=5,\\
	O(\varepsilon^{1/3}),&\quad m\geq6,
	\end{cases}
	\end{align*}
	and $\alpha=4,5,6$,
	\begin{align*}
b_{1}^{\alpha}[\varphi]-b_{1}^{*\alpha}[\varphi]
	&=
	\begin{cases}
	O(|\log\varepsilon|^{-1}),&\quad m=4,\\
	O(\varepsilon^{1-\frac{4}{m}}),&\quad m=5,6,\\
	O(\varepsilon^{\frac{m+2}{3m}}),&\quad m\geq7,
	\end{cases}
	\end{align*}
	where $b_{1}^{\alpha}[\varphi]$ and $b_{1}^{*\alpha}[\varphi]$ are  defined in \eqref{def_bj},  $\alpha=1,2,3,4,5,6$.
\end{prop}

Finally, we close this section by giving the proof of Theorem \ref{thmhigher2} (ii).

\begin{proof}[{\bf Proof of Theorem \ref{thmhigher2} (ii)}]
	For any $x\in\Omega_{R}$, by using  \eqref{nablau_dec},  Theorems \ref{auxi thm} and \ref{coro v1}, we have
	\begin{align}\label{nablau d=3}
	\nabla u=\sum_{\alpha=1}^{6}\left(C_{1}^{\alpha}-C_{2}^{\alpha}\right)
	\nabla u_{1}^{\alpha}+O(1),\quad m\geq3.
	\end{align}
By \eqref{nablau_dec},  Propostions  \ref{prop a1133 R3} and \ref{prop converge B R3}, we give the proof only for $m=4$. The other cases are similar. When $m=4$,
	\begin{align*}
	C_1^\alpha-C_2^\alpha
	&=
	\frac{\sqrt{\kappa}\sqrt{\varepsilon}}{\pi\mu Q_{3,4}}b_{1}^{*\alpha}[\varphi]\left(1+O(|\log\varepsilon|^{-1})\right),\quad\alpha=1,2,\\
	C_1^3-C_2^3
	&=
	\frac{\sqrt{\kappa}\sqrt{\varepsilon}}{\pi(\lambda+2\mu) Q_{3,4}}b_{1}^{*3}[\varphi]\left(1+O(|\log\varepsilon|^{-1})\right),\\
	C_1^4-C_2^4
	&=
	\frac{2\kappa}{\pi\mu|\log\varepsilon|}b_{1}^{*4}[\varphi]\left(1+O(|\log\varepsilon|^{-1})\right),
	\end{align*}
	and for $\alpha=5,6$,
	\begin{align*}
	C_1^\alpha-C_2^\alpha
	&=
	\frac{4\kappa}{\pi(\lambda+2\mu)|\log\varepsilon|}b_{1}^{*\alpha}[\varphi]\left(1+O(|\log\varepsilon|^{-1})\right).
	\end{align*}	
Substituting the above terms into \eqref{nablau d=3}, for $x\in\Omega_{R}$,
\begin{align*}
\nabla u
&=\sum_{\alpha=1}^{6}(C_{1}^{\alpha}-C_{2}^{\alpha})\nabla u_{1}^{\alpha}+O(1)\\
&=\frac{\sqrt{\kappa}\sqrt{\varepsilon}}{\pi Q_{3,4}}\left(\sum_{\alpha=1}^{2}\frac{1}{\mu}b_{1}^{*\alpha}[\varphi]\nabla u_{1}^{\alpha}+\frac{1}{\lambda+2\mu}b_{1}^{*3}[\varphi]\nabla u_{1}^{3}\right)\left(1+O(|\log\varepsilon|^{-1})\right)\\
&\quad+\frac{2\kappa}{\pi|\log\varepsilon|}\left(\frac{1}{\mu}b_{1}^{*4}[\varphi]\nabla u_{1}^{4}+\sum_{\alpha=5}^{6}\frac{1}{\lambda+2\mu}b_{1}^{*\alpha}[\varphi]\nabla u_{1}^{\alpha}\right)\left(1+O(|\log\varepsilon|^{-1})\right)+O(1).
\end{align*}
The proof is finished.
\end{proof}

\section{Application: an extended Flaherty-Keller formula}\label{application}
As an application of the asymptotic expressions in Propositions \ref{prop a11} and \ref{prop a1133}, we prove an extended Flaherty-Keller formula on the effective elastic property of a periodic composite with densely packed inclusions. We are going to follow the setting in \cite{fk,ky,ll} other than the symmetry conditions. Specifically we denote the period cell by $Y:=(-L_{1},L_{1})\times (-L_{2},L_{2})$, where $L_{1}$ and $L_{2}$ are two positive numbers. Let $D\subset Y$ be a $m$-convex domain containing the origin with $C^{2}$ boundary. Assume that $D$ is close to the horizontal boundary of $Y$ and away from the vertical boundary. Let $\varepsilon/2$ be the distance between $D$ and the the horizontal boundary of $Y$, so that $\varepsilon$ becomes the distance between two adjacent inclusions, see Figure \ref{inclusions}.
\begin{figure}[t]
	\begin{minipage}[c]{0.9\linewidth}
		\centering
		\includegraphics[width=2.5in]{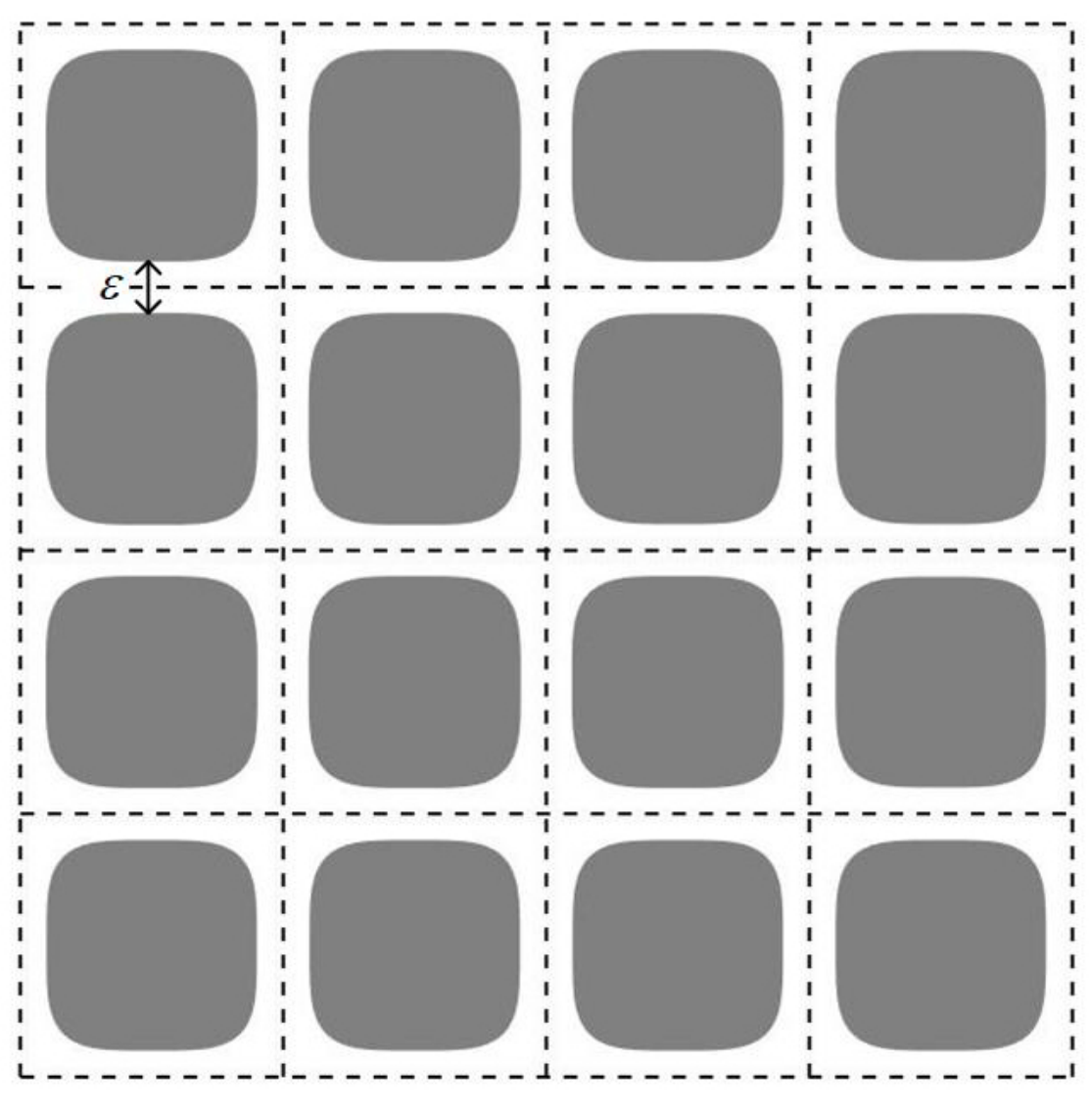}
		\caption{\small $m$-convex inclusions (say, $x^{4}+y^{4}=1$).}
		\label{inclusions}
	\end{minipage}
\end{figure}

As in \cite{ll}, after translation, we denote 
\begin{align*}
Y_{t}:=Y+(0,L_{2})=(-L_{1},L_{1})\times(0,2L_{2}),
\end{align*}
\begin{align*}
D_{1}:=D+(0,2L_{2}+\varepsilon/2),\quad D_{2}:=D+(0,\varepsilon/2),\quad Y':=Y_{t}\setminus\overline{D_{1}\cup D_{2}},
\end{align*}
and set
\begin{align*}
\Gamma_{+}:=(\partial D_{1}\cup\{x_{2}=2L_{2}+\varepsilon/2\})\cap\partial Y',\quad \Gamma_{-}:=(\partial D_{2}\cup\{x_{2}=\varepsilon/2\})\cap\partial Y'.
\end{align*}
Then we obtain the effective shear modulus $\mu_{m}^{*}$ and the effective extensional modulus $E_{m}^{*}$ defined by
\begin{align*}
\mu_{m}^{*}=\frac{L_{2}}{L_{1}}\mathcal{E}_{1}^{1},\quad E_{m}^{*}=\frac{E}{\lambda+2\mu}\frac{L_{2}}{L_{1}}\mathcal{E}_{1}^{2},
\end{align*}
where $E=\frac{\mu(3\lambda+2\mu)}{\lambda+\mu}$, and
\begin{align*}
\mathcal{E}_{1}^{\alpha}=\int_{Y'}(\mathbb{C}^0e(v_{1}^{\alpha}), e(v_{1}^{\alpha})\ dx,\quad \alpha=1,2,
\end{align*}
$v_{1}^{\alpha}\in H^{1}(Y')$ is the solution to
\begin{align*}
\begin{cases}
\mathcal{L}_{\lambda, \mu}v_{1}^{\alpha}:=\nabla\cdot(\mathbb{C}^0e(v_{1}^{\alpha}))=0,\quad&\hbox{in}\ Y',\\
v_{1}^{\alpha}=\psi_{\alpha},&\hbox{on}~\Gamma_{+},\\
v_{1}^{\alpha}=0,&\hbox{on}~\Gamma_{-},\\
\frac{\partial v_{1}^{\alpha}}{\partial \nu}\Big|_{+}=0,&\hbox{on}~x_{1}=\pm L_{1}.
\end{cases}
\end{align*}

Note that the definition of $\mathcal{E}_{1}^{\alpha}$ is similar to that of $a_{11}^{\alpha\alpha}$ in \eqref{def_aij}. Then by using Proposition \ref{prop a11} and Proposition \ref{prop a1133}, we have
\begin{theorem}\label{FK formula}
	Given $m\geq2$. As $\varepsilon\rightarrow0$, 
	\begin{align*}
	\mu_{m}^{*}=\mu\frac{L_{2}}{L_{1}}\frac{ Q_{2,m}}{\kappa^{1/m}\varepsilon^{1-1/m}}+O(1),\quad\mbox{and}\quad E_{m}^{*}=E\frac{L_{2}}{L_{1}}\frac{Q_{2,m}}{\kappa^{1/m}\varepsilon^{1-1/m}}+O(1),
	\end{align*}
	where $\kappa$ is the curvature of $\partial D$ near the origin, and 
	$$Q_{2,m}=2\int_{0}^{\infty}\frac{1}{1+t^{m}}\ dt.$$
\end{theorem}

Clearly, by a direct calculation, we find that Theorem \ref{FK formula} for $m=2$ is actually the result in \cite{ky}. Furthermore, we would like to remark that compared with \cite{ky}, our method can do not need to assume that $D_{1}$ and $D_{2}$ are symmetric.

\section{Appendix: the proof of Theorem \ref{coro v1} and Theorem \ref{coro v1d+1}}

We here give the proof of Theorem \ref{coro v1} and Theorem \ref{coro v1d+1}. The key point is that $|\mathcal{L}_{\lambda, \mu}u_{1}^{\alpha}|$ is improved to be controled by $\frac{C}{\delta(x')}$. This is due to the introduction of $\tilde{u}_{1}^{\alpha}$. Then we adapt the iteration technique first used  in \cite{llby} and further developed  in \cite{bll1,bll2}, to capture all singular terms of $\nabla v_{1}^{\alpha}$ and to obtain the asymptotic formulas.

\begin{proof}[Proof of Theorems \ref{coro v1} and \ref{coro v1d+1}.]
	
{\bf Step 1. Claim:} 
\begin{equation}\label{auxi es claim}
\begin{split}
|\mathcal{L}_{\lambda, \mu}u_{1}^{\alpha}|&\leq C
\begin{cases}
\frac{1}{\delta(x')}+|x_{1}|^{\gamma-1},&~m=2,\\
\frac{|x'|^{m-2}}{\delta(x')}\left(1+\frac{\varepsilon}{|x'|}\right),&~m\geq3,
\end{cases}
\quad\alpha=1,\cdots,d;\\
|\mathcal{L}_{\lambda, \mu}u_{1}^{\alpha}|&\leq\frac{C}{\delta(x')},\quad\alpha=d+1,\cdots,\frac{d(d+1)}{2}.
\end{split}
\end{equation}
We will prove the claim in the light of the following two cases.

{\bf Case 1. $d=2$.} A direct calculation yields in $\Omega_{R}$,
\begin{equation}\label{est auxi}
\Big|\partial_{x_{1}x_{1}}(\bar{u}_{1}^{1})^{1}\Big|,~\Big|\partial_{x_{1}x_{2}}(\tilde{u}_{1}^{1})^{2}\Big|\leq\frac{C|x_{1}|^{m-2}}{\delta(x_{1})},\quad
\Big|\partial_{x_{1}x_{1}}(\tilde{u}_{1}^{1})^{2}\Big|\leq C
\begin{cases}
\frac{|x_{1}|}{\delta(x_{1})}+|x_{1}|^{\gamma-1},&~m=2,\\
|x_{1}|^{m-3},&~m\geq3;
\end{cases}
\end{equation}
and
\begin{align}\label{cal auxi}
\partial_{x_{1}x_{2}}(\bar{u}_{1}^{1})^{1}=-\frac{\partial_{x_{1}}(h_{1}-h_{2})}{\delta^{2}(x_{1})},\quad\partial_{x_{2}x_{2}}(\tilde{u}_{1}^{1})^{2}=\frac{\lambda+\mu}{\lambda+2\mu}\frac{\partial_{x_{1}}(h_{1}-h_{2})}{\delta^{2}(x_{1})}.
\end{align}
By using \eqref{est auxi}, we have
\begin{align}\label{est L u11}
&\Big|(\mathcal{L}_{\lambda, \mu}u_{1}^{1})^{1}\Big|\nonumber\\
&=\lambda\Big(\partial_{x_{1}x_{1}}(u_{1}^{1})^{1}+\partial_{x_{2}x_{1}}(u_{1}^{1})^{2}\Big)+\mu\Big(2\partial_{x_{1}x_{1}}(u_{1}^{1})^{1}+\partial_{x_{2}x_{2}}(u_{1}^{1})^{1}+\partial_{x_{1}x_{2}}(u_{1}^{1})^{2}\Big)\nonumber\\
&=\Big|\mu\Delta (u_{1}^{1})^{1}+(\lambda+\mu)\big(\partial_{x_{1}x_{1}}(u_{1}^{1})^{1}+\partial_{x_{1}x_{2}}(u_{1}^{1})^{2}\big)\Big|\nonumber\\
&=\Big|\mu\Delta (\bar{u}_{1}^{1})^{1}+(\lambda+\mu)\big(\partial_{x_{1}x_{1}}(\bar{u}_{1}^{1})^{1}+\partial_{x_{1}x_{2}}(\tilde{u}_{1}^{1})^{2}\big)\Big|\leq\frac{C|x_{1}|^{m-2}}{\delta(x_{1})}.
\end{align}
By using \eqref{cal auxi}, we get
\begin{align*}
(\lambda+\mu)\partial_{x_{1}x_{2}}(\bar{u}_{1}^{1})^{1}+(\lambda+2\mu)\partial_{x_{2}x_{2}}(\tilde{u}_{1}^{1})^{2}=0,
\end{align*}
which means that the ``bad'' terms in \eqref{cal auxi} are eliminated. Combining this and \eqref{est auxi}, we obtain
\begin{align}\label{est L u12}
\Big|(\mathcal{L}_{\lambda, \mu}u_{1}^{1})^{2}\Big|&=\Big|(\lambda+\mu)\big(\partial_{x_{1}x_{2}}(u_{1}^{1})^{1}+\partial_{x_{2}x_{2}}(u_{1}^{1})^{2}\big)+\mu\Delta u_{1}^{2}\Big|\nonumber\\
&=\Big|\mu\partial_{x_{1}x_{1}}(\tilde{u}_{1}^{1})^{2}\Big|\leq C
\begin{cases}
\frac{|x_{1}|}{\delta(x_{1})}+|x_{1}|^{\gamma-1},&~m=2,\\
|x_{1}|^{m-3},&~m\geq3.
\end{cases}
\end{align}
We henceforth obtain from \eqref{est L u11} and \eqref{est L u12} that
\begin{align*}
\Big|\mathcal{L}_{\lambda, \mu}u_{1}^{1}\Big|\leq C
\begin{cases}
\frac{1}{\delta(x_{1})}+|x_{1}|^{\gamma-1},&~m=2,\\
\frac{|x_{1}|^{m-2}}{\delta(x_{1})}\left(1+\frac{\varepsilon}{|x_{1}|}\right),&~m\geq3.
\end{cases}
\end{align*}
Similarly, we have
\begin{align*}
\Big|\mathcal{L}_{\lambda, \mu}u_{1}^{2}\Big|\leq C
\begin{cases}
\frac{1}{\delta(x_{1})}+|x_{1}|^{\gamma-1},&~m=2,\\
\frac{|x_{1}|^{m-2}}{\delta(x_{1})}\left(1+\frac{\varepsilon}{|x_{1}|}\right),&~m\geq3.
\end{cases}
\end{align*}
Furthermore, we have
\begin{align*}
\Big|\partial_{x_{1}x_{1}}u_{1}^{3}\Big|\leq\frac{C|x_{1}|^{m-1}}{\delta(x_{1})},~\Big|\partial_{x_{1}x_{2}}u_{1}^{3}\Big|,\Big|\partial_{x_{2}x_{2}}u_{1}^{3}\Big|\leq\frac{C}{\delta(x_{1})}.
\end{align*}
Then we obtain
\begin{align*}
\Big|\mathcal{L}_{\lambda, \mu}u_{1}^{3}\Big|\leq\frac{C}{\delta(x_{1})}.
\end{align*}

{\bf Case 2. $d=3$.} We have in $\Omega_{R}$,
\begin{equation}\label{est dimen3}
\begin{split}
&\Big|\partial_{x_{1}x_{1}}(\bar{u}_{1}^1)^{1}\Big|,~\Big|\partial_{x_{2}x_{2}}(\bar{u}_{1}^{1})^{1}\Big|,~\Big|\partial_{x_{1}x_{2}}(\bar{u}_{1}^{1})^{1}\Big|,~\Big|\partial_{x_{1}x_{3}}(\tilde{u}_{1}^{1})^{3}\Big|,~\Big|\partial_{x_{2}x_{3}}(\tilde{u}_{1}^{1})^{3}\Big|\leq\frac{C|x'|^{m-2}}{\delta(x')},\\
&\Big|\partial_{x_{1}x_{1}}(\tilde{u}_{1}^{1})^{3}\Big|,~\Big|\partial_{x_{2}x_{2}}(\tilde{u}_{1}^{1})^{3}\Big|\leq C
\begin{cases}
\frac{|x'|}{\delta(x')}+1,&~m=2,\\
|x'|^{m-3},&~m\geq3;
\end{cases}
\end{split}
\end{equation}
and
\begin{align}\label{est dime3 bad}
\partial_{x_{3}x_{1}}(\bar{u}_{1}^{1})^{1}=-\frac{\partial_{x_{1}}(h_{1}-h_{2})}{\delta(x')^{2}},\quad\partial_{x_{3}x_{3}}(\tilde{u}_{1}^{1})^{3}=\frac{\lambda+\mu}{\lambda+2\mu}\frac{\partial_{x_{1}}(h_{1}-h_{2})}{\delta(x')^{2}}.
\end{align}
By \eqref{est dimen3}, we have
\begin{align}\label{est dim3 u11}
&\Big|(\mathcal{L}_{\lambda, \mu}u_{1}^{1})^{1}\Big|\nonumber\\
&=\Big|\lambda\big(\partial_{x_{1}x_{1}}(u_{1}^{1})^{1}+\partial_{x_{3}x_{1}}(u_{1}^{1})^{3}\big) +\mu\big(2\partial_{x_{1}x_{1}}(u_{1}^{1})^{1}+\partial_{x_{2}x_{2}}(u_{1}^{1})^{1}+\partial_{x_{1}x_{3}}(u_{1}^{1})^{3}\big)\Big|\nonumber\\
&\leq\frac{C|x'|^{m-2}}{\delta(x')},
\end{align}
and
\begin{align}\label{est dim3 u12}
\Big|(\mathcal{L}_{\lambda, \mu}u_{1}^{1})^{2}\Big|&=\Big|\lambda\big(\partial_{x_{1}x_{2}}(u_{1}^{1})^{1}+\partial_{x_{3}x_{2}}(u_{1}^{1})^{3}\big) +\mu\big(\partial_{x_{2}x_{1}}(u_{1}^{1})^{1}+\partial_{x_{2}x_{3}}(u_{1}^{1})^{3}\big)\Big|\nonumber\\
&\leq\frac{C|x'|^{m-2}}{\delta(x')}.
\end{align}
By \eqref{est dime3 bad}, we obtain
\begin{align}\label{sum=0}
(\lambda+\mu)\partial_{x_{3}x_{1}}(\bar{u}_{1}^{1})^{1}+(\lambda+2\mu)\partial_{x_{3}x_{3}}(\tilde{u}_{1}^{1})^{3}=0.
\end{align}
Then \eqref{est dimen3} and \eqref{sum=0} imply that
\begin{align}\label{est dim3 u13}
\Big|(\mathcal{L}_{\lambda, \mu}u_{1}^{1})^{3}\Big|&=\Big|\lambda\big(\partial_{x_{1}x_{3}}(u_{1}^{1})^{1}+\partial_{x_{3}x_{3}}(u_{1}^{1})^{3}\big)\nonumber\\ &\quad+\mu\big(\partial_{x_{1}x_{1}}(u_{1}^{1})^{3}+\partial_{x_{3}x_{1}}(u_{1}^{1})^{1}+\partial_{x_{2}x_{2}}(u_{1}^{1})^{3}+2\partial_{x_{3}x_{3}}(u_{1}^{1})^{3}\big)\Big|\nonumber\\
&=\Big|\mu\big(\partial_{x_{1}x_{1}}(\tilde{u}_{1}^{1})^{3}+\partial_{x_{2}x_{2}}(\tilde{u}_{1}^{1})^{3}\big)\Big|\leq C
\begin{cases}
\frac{|x'|}{\delta(x')}+1,&~m=2,\\
|x'|^{m-3},&~m\geq3.
\end{cases}
\end{align}
Hence, \eqref{est dim3 u11}, \eqref{est dim3 u12}, and \eqref{est dim3 u13} give
\begin{align*}
\Big|\mathcal{L}_{\lambda, \mu}u_{1}^{1}\Big|\leq C
\begin{cases}
\frac{1}{\delta(x')},&~m=2,\\
\frac{|x'|^{m-2}}{\delta(x')}\left(1+\frac{\varepsilon}{|x'|}\right),&~m\geq3.
\end{cases}
\end{align*}
Similarly, we can get
\begin{align*}
\Big|\mathcal{L}_{\lambda, \mu}u_{1}^{2}\Big|, ~\Big|\mathcal{L}_{\lambda, \mu}u_{1}^{3}\Big|\leq C
\begin{cases}
\frac{1}{\delta(x')},&~m=2,\\
\frac{|x'|^{m-2}}{\delta(x')}\left(1+\frac{\varepsilon}{|x'|}\right),&~m\geq3.
\end{cases}
\end{align*}
For the corresponding estimates for $u_{1}^{\alpha}$, $\alpha=4,5,6$, we note that
\begin{align*}
\Big|\partial_{x_{k}x_{l}}u_{1}^{\alpha}\Big|\leq\frac{C|x'|^{m-1}}{\delta(x')},~\Big|\partial_{x_{k}x_{3}}u_{1}^{\alpha}\Big|,\Big|\partial_{x_{3}x_{3}}u_{1}^{\alpha}\Big|\leq\frac{C}{\delta(x')},\quad k,l=1,2.
\end{align*}
We thus obtain
\begin{align*}
\Big|\mathcal{L}_{\lambda, \mu}u_{1}^{\alpha}\Big|\leq\frac{C}{\delta(x')},\quad\alpha=4,5,6.
\end{align*}
Therefore, \eqref{auxi es claim} is proved.

{\bf Step 2. The proof of the boundedness of the global energy.}
We obtain from \cite{bll1} that the global energy of $\nabla (v_{1}^{\alpha}-\bar{u}_{1}^{\alpha})$ are bounded, $\alpha=1,\cdots,d(d+1)/2$. Moreover, $\int_{\Omega\setminus\Omega_{R}}|\nabla\tilde{u}_{1}^{\alpha}|^2$ are also bounded because of $\|u_{1}^{\alpha}\|_{C^{2}(\Omega\setminus\Omega_{R})}\leq C$. So it suffices to prove the bundedness of $\int_{\Omega_{R}}|\nabla\tilde{u}_{1}^{\alpha}|^2$, $\alpha=1,\cdots,d$, since $\tilde{u}_{1}^{\alpha}=0$, $\alpha=d+1,\cdots,d(d+1)/2$. 

When $d=2$. Recalling the definition of $\tilde{u}_{1}^{\alpha}$ in \eqref{auxiliary improved}, we have
\begin{align*}
\nabla(\tilde{u}_{1}^{1})^{1}=0,\quad\Big|\partial_{x_{1}}(\tilde{u}_{1}^{1})^{2}\Big|\leq C,\quad \Big|\partial_{x_{2}}(\tilde{u}_{1}^{1})^{2}\Big|\leq\frac{C|x_{1}|^{m-1}}{\delta(x_{1})};
\end{align*}
and
\begin{align*}
\nabla(\tilde{u}_{1}^{2})^{2}=0,\quad\Big|\partial_{x_{1}}(\tilde{u}_{1}^{2})^{1}\Big|\leq C,\quad \Big|\partial_{x_{2}}(\tilde{u}_{1}^{2})^{1}\Big|\leq\frac{C|x_{1}|^{m-1}}{\delta(x_{1})}.
\end{align*}
Hence,
\begin{align*}
\int_{\Omega_{R}}|\nabla\tilde{u}_{1}^{\alpha}|^2\ dx\leq\int_{\Omega_{R}}\frac{C|x_{1}|^{2(m-1)}}{\delta^{2}(x_{1})}\ dx
\leq C\int_{|x_{1}|<R}\frac{|x_{1}|^{2(m-1)}}{\delta(x_{1})}\ dx_{1}\leq C.
\end{align*}

When $d=3$. We obtain from \eqref{auxiliary improved dim3} that
\begin{align*}
\nabla(\tilde{u}_{1}^{\alpha})^{\beta}=0,\quad\Big|\nabla_{x'}(\tilde{u}_{1}^{\alpha})^{3}\Big|\leq C,\quad \Big|\partial_{x_{3}}(\tilde{u}_{1}^{\alpha})^{3}\Big|\leq\frac{C|x'|^{m-1}}{\delta(x')},\quad\alpha,\beta=1,2;
\end{align*}
and
\begin{align*}
\Big|\nabla_{x'}(\tilde{u}_{1}^{3})^{\beta}\Big|\leq C,\quad \Big|\partial_{x_{3}}(\tilde{u}_{1}^{3})^{\beta}\Big|\leq\frac{C|x'|^{m-1}}{\delta(x')},\quad\nabla(\tilde{u}_{1}^{3})^{3}=0,\quad\beta=1,2.
\end{align*}
Thus, 
\begin{align*}
\int_{\Omega_{R}}|\nabla\tilde{u}_{1}^{\alpha}|^2\ dx\leq\int_{\Omega_{R}}\frac{C|x'|^{2(m-1)}}{\delta^{2}(x')}\ dx
\leq C\int_{|x'|<R}\frac{|x'|^{2(m-1)}}{\delta(x')}\ dx'\leq C.
\end{align*}
Therefore, the boundedness of the global energy of $\nabla(v_{1}^{\alpha}-u_{1}^{\alpha})$ is established.

{\bf Step 3. Proof of}  
\begin{align}\label{local bound}
\int_{\Omega_{\delta}(z')}|\nabla w_{1}^{\alpha}|^{2}\ dx\leq C\delta^{d}(z')
\begin{cases}
\delta^{2(1-\frac{2}{m})}(z'),&\quad\alpha=1,\cdots,d,\\
1,&\quad\alpha=(d+1),\cdots,d(d+1)/2.
\end{cases}
\end{align}
where
$$\Omega_{s}(z'):=\left\{(x',x_{d})\in \mathbb{R}^{d}~\big|~h_{2}(x')<x_{d}<\varepsilon+h_{1}(x'),~|x'-z'|<r\right\},\quad s<R,$$
and $w_{1}^{\alpha}=v_{1}^{\alpha}-u_1^{\alpha}$, $\alpha=1,\cdots,d(d+1)/2$, satisfying
\begin{align}\label{equation w}
\begin{cases}
\mathcal{L}_{\lambda,\mu}w_{1}^{\alpha}=-\mathcal{L}_{\lambda,\mu}u_{1}^{\alpha},&
\hbox{in}\  \Omega,  \\
w=0, \quad&\hbox{on} \ \partial\Omega.
\end{cases}
\end{align}

We will use the iteration scheme developed in  \cite{bll1, bll2, llby} to prove \eqref{local bound}. For $0<t<s<R$, let $\eta$ be a smooth cutoff function satisfying $\eta(x')=1$ if $|x'-z'|<t$, $\eta(x')=0$ if $|x'-z'|>s$, $0\leq\eta(x')\leq1$ if $t\leq|x'-z'|\leq\,s$, and $|\nabla_{x'}\eta(x')|\leq\frac{2}{s-t}$. Multiplying the equation in (\ref{equation w}) by $w_{1}^{\alpha}\eta^2$ and applying integration by parts, H\"older's inequality, and Cauchy inequality, we get
\begin{align}\label{integrationbypart}
\int_{\Omega_t(z')}|\nabla w_{1}^{\alpha}|^2\ dx\leq\frac{C}{(s-t)^2}\int_{\Omega_{s}(z')}|w_{1}^{\alpha}|^2\ dx +C(s-t)^2\int_{\Omega_{s}(z')}|\mathcal{L}_{\lambda,\mu}u_{1}^{\alpha}|^2\ dx.
\end{align}
On one hand, we obtain from H\"{o}lder's inequality that
\begin{align}\label{w}
\int_{\Omega_{s}(z')}|w_{1}^{\alpha}|^2\ dx=\int_{\Omega_{s}(z')}\left|\int_{h_2(x')}^{x_{d}}\partial_{x_{d}}w_{1}^{\alpha}(x', \xi)\ d\xi\right|^2\ dx\leq\,C\delta^{2}(z')\int_{\Omega_{s}(z')}|\nabla w_{1}^{\alpha}|^2\ dx.
\end{align}
On the other hand, we estimate the second term on the right hand side of \eqref{integrationbypart} according to the following two cases.

{\bf Case 1. $|z'|\leq\varepsilon^{1/m}$.} By using \eqref{auxi es claim}, we have for $0<s<\varepsilon^{1/m}$, 
\begin{align}\label{local est u}
\int_{\Omega_s(z')}|\mathcal{L}_{\lambda,\mu}u_{1}^{\alpha}|^{2}\ dx\leq C s^{d-1}
\begin{cases}
\varepsilon^{\frac{2(m-2)}{m}-1},&\quad\alpha=1,\cdots,d,\\
\varepsilon^{-1},&\quad\alpha=(d+1),\cdots,d(d+1)/2.
\end{cases}
\end{align}
This is an improvement of \cite[(3.32),(3.35)]{bll1}. Denote
$$F(t):=\int_{\Omega_{t}(z')}|\nabla w_{1}^{\alpha}|^{2}.$$
Then substituting \eqref{w} and \eqref{local est u} into \eqref{integrationbypart}, we have
\begin{equation}\label{energy_w1}
F(t)\leq\left(\frac{c_{1}\varepsilon}{s-t}\right)^{2}F(s)+C(s-t)^2s^{d-1}
\begin{cases}
\varepsilon^{\frac{2(m-2)}{m}-1},&~\alpha=1,\cdots,d,\\
\varepsilon^{-1},&~\alpha=(d+1),\cdots,d(d+1)/2,
\end{cases}
\end{equation}
where $c_{1}$ is a {\it universal canstant.} 

Let $k=\left[\frac{1}{4c_{1}\varepsilon^{1/m}}\right]$ and $t_{i}=2c_{1}i\varepsilon, i=1,\cdots,k$. Then by \eqref{energy_w1} with $s=t_{i+1}$ and $t=t_{i}$, we have
$$F(t_{i})\leq\frac{1}{4}F(t_{i+1})+C(i+1)^{d-1}\varepsilon^{d}
\begin{cases}
\varepsilon^{\frac{2(m-2)}{m}},&\quad\alpha=1,\cdots,d,\\
1,&\quad\alpha=(d+1),\cdots,d(d+1)/2.
\end{cases}$$
After $k$ iterations, using the global boundedness of $\nabla w_{1}^{\alpha}$, we have
$$F(t_{1})\leq C\varepsilon^{d}
\begin{cases}
\varepsilon^{\frac{2(m-2)}{m}},&\quad\alpha=1,\cdots,d,\\
1,&\quad\alpha=(d+1),\cdots,d(d+1)/2.
\end{cases}$$

{\bf Case 2. $\varepsilon^{1/m}<|z'|<R$.}  For $0<s<\frac{2}{3}|z'|$, \eqref{local est u} becomes
\begin{align*}
\int_{\Omega_s(z')}|\mathcal{L}_{\lambda,\mu}u_{1}^{\alpha}|^{2}\ dx\leq Cs^{d-1}
\begin{cases}
|z'|^{m-4},&\quad\alpha=1,\cdots,d,\\
|z'|^{-m},&\quad\alpha=(d+1),\cdots,d(d+1)/2.
\end{cases}
\end{align*}
\eqref{energy_w1} becomes
\begin{equation*}
F(t)\leq\left(\frac{c_{2}|z'|^{m}}{s-t}\right)^{2}F(s)+C(s-t)^2s^{d-1}\begin{cases}
|z'|^{m-4},&\quad\alpha=1,\cdots,d,\\
|z'|^{-m},&\quad\alpha=(d+1),\cdots,d(d+1)/2,
\end{cases}
\end{equation*}
where $c_{2}$ is another {\it universal canstant.} Let $k=\left[\frac{1}{4c_{2}|z'|}\right]$ and $t_{i}=2c_{2}i|z'|^{m}, i=1,\cdots,k$. Then by \eqref{energy_w1} with $s=t_{i+1}$ and $t=t_{i}$, we have
$$F(t_{i})\leq\frac{1}{4}F(t_{i+1})+C(i+1)^{d-1}|z'|^{md}
\begin{cases}
|z'|^{2(m-2)},&\quad\alpha=1,\cdots,d,\\
1,&\quad\alpha=(d+1),\cdots,d(d+1)/2.
\end{cases}$$
After $k$ iterations, using the global boundedness of $\nabla w_{1}^{\alpha}$, we have
$$F(t_{1})\leq C |z'|^{md}\begin{cases}
|z'|^{2(m-2)},&\quad\alpha=1,\cdots,d,\\
1,&\quad\alpha=(d+1),\cdots,d(d+1)/2.
\end{cases}$$
So \eqref{local bound} is proved.

{\bf Step 4. Scaling and $L^{\infty}$-estimates.}  It follows from \cite[(3.40)]{bll1} that
\begin{align*}
\|\nabla w_{1}^{\alpha}\|_{L^\infty(\Omega_{\delta/2}(z'))}\leq\frac{C}{\delta}\left(\delta^{1-\frac{d}{2}}\|\nabla w_{1}^{\alpha}\|_{L^2(\Omega_\delta(z'))}+\delta^2(z')\|\mathcal{L}_{\lambda,\mu}u_{1}^{\alpha}\|_{L^\infty(\Omega_\delta(z'))}\right).
\end{align*}
By using \eqref{local bound} and \eqref{auxi es claim}, we obtain 
\begin{equation*}
|\nabla w_{1}^{\alpha}(z',x_{d})|\leq\,C,\quad h_{2}(z')<x_{d}<\varepsilon+h_{1}(z').
\end{equation*}
Theorems \ref{coro v1} and \ref{coro v1d+1} are proved.
\end{proof}

\noindent{\bf{\large Acknowledgements.}} H.G. Li was partially supported by  NSFC (11631002, 11971061) and BJNSF (1202013). The authors are grateful to professor Yanyan Li for his encouragement and very helpful suggestions.


\begin{thebibliography}{99}

\bibitem{abtv} H. Ammari; E. Bonnetier; F. Triki; M. Vogelius. Elliptic estimates in composite media with smooth inclusions: an integral equation approach. Ann. Sci. \'{e}c. Norm. Sup\'{e}r. (4)  48  (2015),  no. 2, 453-495.

\bibitem{ackly} H. Ammari; G. Ciraolo; H. Kang; H. Lee; K. Yun. Spectral analysis of the
Neumann-Poincar\'{e} operator and characterization of the stress concentration in antiplane
elasticity. Arch. Ration. Mech. Anal. 208 (2013), 275-304.

\bibitem{adkl}  H. Ammari; H. Dassios; H. Kang; M. Lim. Estimates for the electric field in the presence of adjacent perfectly conducting spheres. Quat. Appl. Math. 65 (2007), 339-355.

\bibitem{adn} S. Agmon; A. Douglis; L. Nirenberg. Estimates near the boundary for solutions of elliptic partial differential equations satisfying general boundary conditions. I. Comm. Pure Appl. Math. 12 (1959), 623-727.

\bibitem{adn0} S. Agmon; A. Douglis; L. Nirenberg. Estimates near the boundary for solutions of elliptic partial differential equations satisfying general boundary conditions. II. Comm. Pure Appl. Math. 17 (1964), 35-92.

\bibitem{akl} H. Ammari; H. Kang; M. Lim. Gradient estimates to the conductivity problem. Math.
Ann. 332 (2005), 277-286.

\bibitem{aklll}  H. Ammari; H. Kang; H. Lee; J. Lee; M. Lim. Optimal estimates for the electrical
field in two dimensions. J. Math. Pures Appl. 88 (2007), 307-324.

\bibitem{ba} I. Babu\v{s}ka; B. Andersson; P. Smith; K. Levin. Damage analysis of fiber composites. I. Statistical analysis on fiber scale. Comput. Methods Appl. Mech. Engrg. 172 (1999), 27-77.

\bibitem{bjl}  J.G. Bao; H.J. Ju; H.G. Li. Optimal boundary gradient estimates for Lam\'{e} systems with partially infinite coefficients. Adv. Math. 314 (2017), 583-629.


\bibitem{bll1} J.G. Bao; H.G. Li; Y.Y. Li. Gradient estimates for solutions of the Lam\'{e} system with partially infinite coefficients. Arch. Ration. Mech. Anal.  215  (2015),  no. 1, 307-351.

\bibitem{bll2} J.G. Bao; H.G. Li; Y.Y. Li. Gradient estimates for solutions of the Lam\'{e} system with partially infinite coefficients  in dimensions greater than two. Adv. Math. 305 (2017), 298-338.

\bibitem{bly1} E.S. Bao; Y.Y. Li; B. Yin. Gradient estimates for the perfect conductivity problem. Arch. Ration. Mech. Anal. 193 (2009), 195-226.

\bibitem{bly2} E. Bao; Y.Y. Li; B. Yin. Gradient estimates for the perfect and insulated conductivity problems with multiple inclusions. Comm. Partial Differential Equations 35 (2010), 1982-2006.

\bibitem{bt} E. Bonnetier; F. Triki. On the spectrum of the Poincar\'{e} variational problem for two close-to-touching inclusions in $2D$. Arch. Ration. Mech. Anal. 209 (2013), no. 2, 541-567.

\bibitem{bv} E. Bonnetier; M. Vogelius. An elliptic regularity result for a composite medium with ``touching'' fibers of circular cross-section. SIAM J. Math. Anal. 31  (2000) 651-677.

\bibitem{bc} B. Budiansky; G.F. Carrier. High shear stresses in stiff fiber composites. J. App. Mech. 51 (1984), 733-735.

\bibitem{cheng} H.W. Cheng. On the method of images for systems of closely spaced conducting spheres. SIAM J. Appl. Math. 61 (2000), no. 4, 1324-1337.

\bibitem{cg} H.W. Cheng; L. Greengard. A method of images for the evaluation of electrostatic fileds in systems of closely spaced conducting cylinders.  SIAM J. Appl. Math. 58 (2006), no. 1, 122-141.

\bibitem{dong} H.J. Dong. Gradient estimates for parabolic and elliptic systems from linear laminates. Arch. Ration. Mech. Anal. 205 (2012), no. 1, 119-149.

\bibitem{dl} H.J. Dong; H.G. Li. Optimal Estimates for the Conductivity Problem by Green's Function Method. Arch. Ration. Mech. Anal. 231 (2019), no. 3, 1427-1453.

\bibitem{dx} H.J. Dong; L.J. Xu. Gradient estimates for divergence form elliptic systems arising from composite material. SIAM J. Math. Anal. 51 (2019), no. 3, 2444-2478.

\bibitem{fk} J.E. Flaherty; J.B. Keller. Elastic behavior of composite media. Comm. Pure. Appl. Math. 26 (1973), 565-580.

\bibitem{gk} Y. Grabovsky; R.V. Kohn. Microstructure minimizing the energy of a two phase elastic composite in two space dimensions. II: the Vigdergauz microstructure. J. Mech. Phys. Solids. 43 (1995), 949-972.

\bibitem{g} Y. Gorb. Singular behavior of electric field of high-contrast concentrated composites. Multiscale Model. Simul. 13 (2015), no. 4, 1312-1326.

\bibitem{gb} Y. Gorb and L. Berlyand. Asymptotics of the effective conductivity of composites with closely spaced inclusions of optimal shape. Quart. J. Mech. Appl. Math., 58 (2005), pp. 83-106.

\bibitem{gn} Y. Gorb and A. Novikov. Blow-up of solutions to a $p$-Laplace equation. Multiscale Model. Simul., 10 (2012), pp. 727-743.

\bibitem{hl} Y.Y. Hou; H.G. Li. The convexity of inclusions and gradient's concentration for the Lam\'{e} systems with partially infinite coefficients. arXiv: 1802.01412v1.

\bibitem{jlx} H.J. Ju; H.G. Li; L.J. Xu. Estimates for elliptic systems in a narrow region arising from composite materials. Quart. Appl. Math. 77 (2019), 177-199.

\bibitem{kleey} H. Kang; H. Lee; K. Yun. Optimal estimates and asymptotics for the stress concentration between closely located stiff inclusions. Math. Ann. 363 (2015), no. 3-4, 1281-1306.

\bibitem{kly} H. Kang; M. Lim; K. Yun. Asymptotics and computation of the solution to the
conductivity equation in the presence of adjacent inclusions with extreme conductivities.
J. Math. Pures Appl. (9) 99 (2013), 234-249.

\bibitem{kly2} H. Kang; M. Lim; K. Yun. Characterization of the electric field concentration between two adjacent spherical perfect conductors. SIAM J. Appl. Math. 74 (2014),
125-146.

\bibitem{ky} H. Kang; S. Yu. A proof of the Flaherty-Keller formula on the effective property of densely packed elastic composites. Calc. Var. Partial Differential Equations. 59 (2020), no. 1, 13 pp.

\bibitem{keller1} J.B. Keller. Conductivity of a medium containing a dense array of perfectly conducting spheres or cylinders or nonconducting cylinders. J. Appl. Phys. 34 (1963), 991-993.

\bibitem{keller2} J.B. Keller. Stresses in narrow regions. Trans. ASME J. Appl. Mech. 60 (1993), 1054-1056.

\bibitem{l} H.G. Li. Lower bounds of gradient's blow-up for the Lam\'{e} system with partially infinite coefficients. to appear in J. Math. Pures Appl. (2020). DOI: https://doi.org/10.1016/j.matpur.2020.09.007.

\bibitem{lhg} H.G. Li. Asymptotics for the electric field concentration in the perfect conductivity problem. SIAM J. Math. Anal. 52 (2020), no.4, 3350-3375.

\bibitem{ll} H.G. Li; Y. Li. An extension of flaherty-keller formula for density packed m-convex inclusion. arXiv: 1912.13261v1.

\bibitem{Li-Li} H.G. Li; Y.Y. Li.
Gradient estimates for parabolic systems from composite
material.  Sci. China Math. 60 (2017), no. 11, 2011-2052.

\bibitem{llby} H.G. Li; Y.Y. Li; E.S. Bao; B. Yin. Derivative estimates of solutions of elliptic systems in narrow regions. Quart. Appl. Math.  72  (2014),  no. 3, 589-596.

\bibitem{lly} H.G. Li; Y.Y. Li; Z.L. Yang. Asymptotics of the gradient of solutions to the perfect conductivity problem. Multiscale Model. Simul. 17 (2019), no. 3, 899-925.

\bibitem{ln}  Y.Y. Li; L. Nirenberg. Estimates for elliptic system from composite material. Comm. Pure Appl. Math. 56 (2003), 892-925.

\bibitem{lwx} H.G. Li; F. Wang; L.J. Xu. Characterization of Electric Fields between two Spherical Perfect Conductors with general radii in 3D. J. Differential Equations (2019), 267(11), 6644-6690.


\bibitem{lv} Y.Y. Li; M. Vogelius. Gradient stimates for solutions to divergence form elliptic equations with discontinuous
coefficients. Arch. Rational Mech. Anal. 153 (2000), 91-151.

\bibitem{lx} H.G. Li; L.J. Xu. Optimal estimates for the perfect conductivity problem with inclusions close to the boundary. SIAM J. Math. Anal. 49 (2017), no. 4, 3125-3142.

\bibitem{lz} H.G. Li; Z.W. Zhao. Boundary blow-up analysis of gradient estimates for Lamé systems in the presence of $M$-convex hard inclusions.  SIAM J. Math. Anal. 52 (2020), no. 4, 3777-3817. 

\bibitem{lyu1} M. Lim; S. Yu. Asymptotics of the solution to the conductivity equation in the
presence of adjacent circular inclusions with finite conductivities. J. Math. Anal. Appl.
421 (2015), 131-156.

\bibitem{lyu} M. Lim; S. Yu. Stress concentration for two nearly touching circular holes. arXiv: 1705.10400v1. (2017)

\bibitem{ly2} M. Lim; K. Yun. Blow-up of electric fields between closely spaced spherical perfect conductors. Comm. Partial Differential Equations, 34 (2009), 1287-1315.

\bibitem{m} X. Markenscoff, Stress amplification in vanishingly small geometries. Computational Mechanics 19 (1996), 77-83.

\bibitem{mmn} V.G. Maz'ya; A.B. Movchan; M.J. Nieves. Uniform asymptotic formular for Green's
tensors in elastic singularly perturbed domains. Asym. Anal. 52 (2007), 173-206.


\bibitem{mpm} R. McPhedran; L. Poladian; G.W. Milton. Asymptotic studies of closely spaced, highly conducting cylinders. Proc. Roy. Soc. London Ser. A, 415, (1988), 185-196.


\bibitem{mt} R. Meredith, C. Tobias. Resistance to potential flow through a cubical array of spheres. J. Applied Physics,  31 (1960), no. 7,  1270-1273.

\bibitem{rl} L. Rayleigh. On the influence of obstacles arranged in rectangular order upon the properties of a medium. Phil. Mag. 34 (1892), 481-502.

\bibitem{v} S.B. Vigdergauz. Integral equation of the inverse problem of the plane theory of elasticity. PMM. 40 (1976), 518-521.
 
\bibitem{y1} K. Yun. Estimates for electric fields blown up between closely adjacent conductors with arbitrary shape. SIAM J. Appl. Math. 67 (2007),  714-730.

\bibitem{y2} K. Yun. Optimal bound on high stresses occurring between stiff fibers with arbitrary shaped cross-sections. J. Math. Anal. Appl. 350 (2009), 306-312.

\end{thebibliography}
\end{document}